\documentclass[10pt]{amsart}
\usepackage[latin1]{inputenc}
\usepackage[T1]{fontenc}
\usepackage{epsfig}
\usepackage{mathrsfs}
\usepackage{xspace}
\usepackage{color}
\usepackage[english]{babel}

\usepackage{amsthm}
\usepackage{amsmath}
\usepackage{amsfonts}
\usepackage{amssymb}
\usepackage{graphicx}
\newtheorem{lem}{Lemma}[section]
\newtheorem{prop}{Proposition}[section]
\newtheorem{cor}{Corollary}[section]
\newtheorem{thm}{Theorem}[section]
\newtheorem{defi}{Definition}[section]

\newtheorem{remark}{Remark}[section]
\newfont{\sBlackboard}{msbm10 scaled 900}

\newcommand{\vv}     {{\rm v}}

\newcommand{\dd}     {{\rm d}}
\newcommand{\mylabel}[1]{\label{#1}
            \ifx\undefined\stillediting
            \else \fbox{$#1$}\fi }
\newcommand{\BE}{\begin{equation}}

\newcommand{\EEQ}{\end{equation}}
\newcommand{\rfb}[1]{\mbox{\rm
   (\ref{#1})}\ifx\undefined\stillediting\else:\fbox{$#1$}\fi}

\newfont{\Blackboard}{msbm10 scaled 1200}

\newfont{\roma}{cmr10 scaled 1200}

\def\CC{\rm \hbox{C\kern-.56em\raise.4ex
         \hbox{$\scriptscriptstyle |$}\kern+0.5 em }}

\newcommand{\rline}  {{\mathbb R}}

\newcommand{\ud}{\mathrm{d}}

\newcommand{\R}{\mathbb{R}}
\newcommand{\C}{\mathbb{C}}
\newcommand{\N}{\mathbb{N}}

\newcommand{\mm}    {{\hbox{\hskip 0.5pt}}}

\newcommand{\bluff} {{\hbox{\raise 15pt \hbox{\mm}}}}
\makeatletter
\def\section{\@startsection {section}{1}{\z@}{-3.5ex plus -1ex minus
    -.2ex}{2.3ex plus .2ex}{\large\bf}}
\makeatother
\def\be{\begin{equation}}
\def\ee{\end{equation}}
\def\beqs{\begin{eqnarray*}}
\def\eeqs{\end{eqnarray*}}
\def\pd{\partial}
\def\ds{\displaystyle}
\let \div \relax
\DeclareMathOperator{\div}{div}

\newcommand{\nhd}{neighborhood\xspace}
\let \Re \relax
\DeclareMathOperator{\Re}{Re}
\let \Im \relax
\DeclareMathOperator{\Im}{Im}
\DeclareMathOperator{\supp}{supp}

\newcommand{\eps}{{\varepsilon}}

\newcommand{\Con}{\ensuremath{\mathscr C}}
 \DeclareMathOperator{\ops}{op_{sc}}
  \DeclareMathOperator{\op}{op}
\DeclareMathOperator{\Ai}{Ai}
 \newcommand{\est}[1]{\langle #1 \rangle}

\begin{document}
\thispagestyle{empty}
\title[Stabilization of an acoustic system]{Uniform stabilization of an acoustic system}

\author{Ka\"{\i}s AMMARI}
\address{LR Analysis and Control of PDEs, LR 22ES03, Department of Mathematics, Faculty of Sciences of Monastir, University of Monastir, 5019 Monastir, Tunisia and LMV/UVSQ/Paris-Saclay, France} \email{kais.ammari@fsm.rnu.tn}

 \author{Fathi Hassine}
\address{LR Analysis and Control of PDEs, LR 22ES03, Department of Mathematics, Faculty of Sciences of Monastir, University of Monastir, 5019 Monastir, Tunisia} \email{fathi.hassine@fsm.rnu.tn}

\author{Luc ROBBIANO}
\address{Université Paris-Saclay, UVSQ, CNRS, Laboratoire de Mathématiques de Versailles, 78000, Versailles, France.
}
\email{luc.robbiano@uvsq.fr}
\date{}

\begin{abstract}
We study the problem of stabilization for the acoustic system with a spatially distributed damping. With imposing hypothesis on the structural properties of the damping term, we identify exponential decay of solutions with growing time.  
\end{abstract}

\subjclass[2010]{35L04, 37L15}
\keywords{exponential stability,  resolvent estimate, dissipative hyperbolic system, acoustic equation}

\maketitle

\tableofcontents
%
%
\section{Introduction} \label{intro}
Sound waves propagate due to the compressibility of a medium. In fact, it is a balance between the 
compressibility and the inertia of fluid that governs the propagation of sound waves through it. 
The assumptions made for the acoustic wave equation are gravitational forces is neglected, dissipative 
effects are neglected, the medium (fluid) is homogeneous, isotropic and perfectly elastic. In addition to 
these assumptions we also assume that particle velocity is small, and there are only very small perturbations
(fluctuations) to the equilibrium pressure and density.

We consider the following acoustic damping system of equations:

\be
\label{fluide}
\left\{
\begin{array}{l}
u_t + \nabla r + b \, u = 0, \,\hbox{ in } \Omega \times \rline^+, \\
r_t + \div u = 0, \,\hbox{ in } \Omega \times \rline^+, \\
u \cdot n = 0, \,\hbox{ on } \Gamma \times \rline^+, \\
u(0,x) = u^0(x), \, r(0,x) = r^0(x), \, x \in \Omega,
\end{array}
\right.
\ee
where $\Omega$ is a bounded domain in $\rline^d, \, d\geq2$,
with a smooth boundary $\Gamma$, $\div=\nabla\cdot$ is the divergence operator and $b\in L^{\infty}(\Omega)$,
with $b \ge 0$ on $\Omega$. 
As usual $n$ denotes the unit outward normal vector along $\Gamma$.

\medskip

The system of equations (\ref{fluide}) is a linearization of the \emph{acoustic equation} governing the propagation of 
acoustic waves in a compressible medium, see Lighthill \cite{Lighthill:78,Lighthill:52,Lighthill:54}, where $b \, u$ represents a damping term of Brinkman type. 
This kind of damping arises also in the process of homogenization (see Allaire \cite{Allaire:91}), and is frequently used as a suitable \emph{penalization} 
in fluid mechanics models, see Angot, Bruneau, and Fabrie \cite{AngotBruneauFabrie:99}. Our main goal is to prove the exponential  decay of solutions of (\ref{fluide}) with growing time.  
 
\medskip

Let $L^2(\Omega)$ denote the standard Hilbert space of square integrable functions in $\Omega$
and its closed subspace $L^2_m(\Omega)=\{f\in L^2(\Omega): \int_\Omega f(x)\,dx=0\}$.
To avoid abuse of notation,
we shall write $\|\cdot\|$ for the $L^2(\Omega)$-norm or the $L^2(\Omega)^d$-norm.

\medskip

Denoting  $H = (L^2(\Omega))^d \times L^2_m(\Omega)$, we introduce the operator
$$
{\mathcal A} = \left(
\begin{array}{ll}
0  & \nabla \\
\div  & 0
\end{array}
\right) : {\mathcal D}({\mathcal A}) = \left\{(u,r) \in H, \, (\nabla r,\div u) \in H, \, u \cdot n_{|\Gamma} = 0 \right\} \subset  H \rightarrow H,
$$
and
$$
{\mathcal B} =  \left( \begin{array}{ll} \sqrt{b} \\ 0 \end{array} \right) \in {\mathcal L}((L^2(\Omega))^d,H), \,
{\mathcal B}^* = \left( \begin{array}{cc} \sqrt{b} & 0 \end{array} \right) \in {\mathcal L}(H, (L^2(\Omega))^d).
$$
We recall that for $u\in  (L^2(\Omega))^d$ with $\div u\in L^2(\Omega)$, $u \cdot n_{|\Gamma} $ make sens in $H^{-1/2}(\Gamma)$
(see Girault-Raviart~\cite[Chap 1, Theorem 2.5]{Girault-Raviart}).
\medskip

Accordingly, the problem (\ref{fluide}) can be recasted in an abstract form:
\be
\label{cauchy}
\left\{
\begin{array}{l}
Z_t (t) + {\mathcal A} Z(t) + {\mathcal B} {\mathcal B}^* Z(t) = 0, \, t > 0, \\
Z(0) = Z^0,
\end{array}
\right.
\ee
where $Z = (u, r)$,
or, equivalently, 
\be
\label{cauchybis}
\left\{
\begin{array}{l}
Z_t (t) = {\mathcal A}_d Z(t), \, t > 0, \\
Z(0) = Z^0,
\end{array}
\right.
\ee
with ${\mathcal A}_d = - {\mathcal A} - {\mathcal B}{\mathcal B}^*$ with ${\mathcal D}({\mathcal A}_d) = {\mathcal D}({\mathcal A}).$

\medskip

It can be shown (see \cite{AFN}) that for any initial data $(u^0, r^0) \in {\mathcal D}({\mathcal A})$ the 
problem \eqref{fluide} admits a unique solution
$$(u,r) \in C([0,\infty); {\mathcal D}({\mathcal A})) \cap C^1([0, \infty); H).$$ Moreover, the solution $(u,r)$ satisfies,
the energy identity
\be
\label{energyid}
E(0) - E(t) =
 \int_0^t
\left\|\sqrt{b} \, u(s)\right\|_{(L^2(\Omega))^d}^2 \dd s, \ \mbox{for all}\ t \geq 0
\ee
with
\be
\label{energy}
E(t) = \frac{1}{2} \,\left\|(u(t),r(t)) \right\|^2_{H}, \, \forall \, t \geq 0, 
\ee
where we have denoted
$$
\left\langle (u,r), (v,p)\right\rangle_H = \int_\Omega \left(u(x) . v(x) + r(x) p(x) \right) \, \ud x , \,  \left\| (u,r)\right\|_H = \sqrt{\int_\Omega \left(\left|u(x)\right|^2 + r^2(x) \right) \, \ud x}.
$$

\medskip

Using (\ref{energyid}) and a standard density argument, we can extend the solution operator for  data 
$(u^0, r^0) \in H$. Consequently, we associate with the problem (\ref{fluide}) (or to the abstract Cauchy problems (\ref{cauchy}) or
(\ref{cauchybis})) a  semigroup that is globally bounded in $H$. 

\medskip

As the energy $E$ is nonincreasing along trajectories, we want to determine the set of initial data $(u^0, r^0)$ for which 
\begin{equation}\label{stab}
E(t) \to 0 \ \mbox{as} \ t \to \infty.
\end{equation}
Such a question is of course intimately related to the structural properties of the function $b$, notably to 
the geometry of the set $\omega$ on which the damping is effective. In fact, when the damping term is 
globally distributed Ammari, Feireisl and Nicaise \cite{AFN} showed an exponential decay rate of the energy 
by the means of an observability inequality associated with the conservative problem of \eqref{fluide}. 
Besides, it is also shown that if the damping coefficient is not uniformly positive definite 
(i.e $\ds\inf_{x\in\Omega}b(x)=0$) then the system \eqref{fluide} may be no exponentially stable. 
In this paper we consider a damping which is locally distributed over the domain $\Omega$ with a 
geometrical control conditions. In this article we prove uniform  decay as given 
in \cite{AFN} but under more restrictive geometrical conditions on the damping coefficient.
More precisely, we prove exponential  decay rates of the energy. 

\medskip

The paper is organized as follows. Section \ref{section2} summarizes some well known facts concerning the 
acoustic system \rfb{fluide}. In Section \ref{Sec: Tools}, we introduce some tools and notation as Riemannian 
geometry, pseudo-differential operators, and Gearhart-Huang-Pr\"uss theorem allowing to prove exponential decay
from estimates on the resolvent.
In Section~\ref{Sec: Introduction of a semiclassical measure} we introduce a semi-classical measure 
and we prove some properties, first the measure 
is supported on characteristic set and is not identically null.
In Section~\ref{Sec: Propagation of the support} we prove  local  properties in a \nhd of boundary. 
From these results we prove in 
Section~\ref{sec: Precise description of Geometric Control Condition} global propagation of measure support. 
In Section~\ref{Sec: Estimate of boundary trace, case strictly diffractive} we prove an estimate in a \nhd of a diffractive 
point. Section~\ref{Sec: Propagation of the support}, \ref{sec: Precise description of Geometric Control Condition} and 
\ref{Sec: Estimate of boundary trace, case strictly diffractive} imply that the measure is null if $\{x,b(x)>0\}$ satisfies 
a geometric condition (GCC)(see Definition~\ref{def: GCC flat} and Section~\ref{sec: Geometry},
 Section~\ref{Sec: Geometric Control Condition (GCC)} for more precise definition).

\section{Preliminaries and main results}
\label{section2}

We start with a simple observation that the problem (\ref{fluide}) can be viewed as a bounded (in $H$) perturbation of the conservative system 
\be
\label{fluidec}
\left\{
\begin{array}{l}
u_t + \nabla r  = 0, \,\hbox{ in } \Omega \times \rline^+, \\
r_t + \div u = 0, \,\hbox{ in } \Omega \times \rline^+, 
\end{array}
\right.
\ee
which can be recast as the standard \emph{wave equation}
\[
r_{tt} - \Delta r = 0.
\]
Consequently, the basic existence theory for (\ref{fluide}) derives from that of (\ref{fluidec}).
Hence ${\mathcal A}_d$ generates a $C_0$-semigroup $(S(t))_{t\geq 0}$ in $H$ that is even of contraction because ${\mathcal A}_d$ is dissipative (see \rfb{energyid}).

The first main difficulty is that
the operator ${\mathcal A}_d$ possesses a non-trivial (and large) kernel that is left invariant by the evolution.
Indeed if  $(u,r)$ belongs to $\ker {\mathcal A}_d$, then it is solution of the ``stationary'' problem 
\begin{equation}\label{st1}
\nabla r + b u = 0, \ \div u = 0, \,\hbox{ in } \Omega.
\end{equation} 
Thus  multiplying the first identity of  (\ref{st1}) by $\overline{u}$ and integrating over $\Omega$ yields 
\[
\int_\Omega (\nabla r \cdot \overline{u}+ b |u|^2)\,\ud x=0.
\]
By an integration by parts, using the fact that $u$ is solenoidal and the 
boundary condition $u \cdot n= 0$ on  $\Gamma$, we get
\[
\int_\Omega \nabla r \cdot \overline{u}\,\ud x=\int_\Omega r\div  \overline{u}\,\ud x= 0,\]
and therefore we obtain 
\[
\int_\Omega  b |u|^2\,\ud x=0.
\]
In other words, we have
\[
 u = 0 \ \mbox{on}\ \supp b ,
\]
and coming back to (\ref{st1}), we find
\[
\nabla r = 0.
\]
Accordingly, we have shown that
\[
\ker {\mathcal A}_d = \{(u, 0)\in {\mathcal D}({\mathcal A}) \ | \ \div u = 0, \ u|_{ \supp b} = 0, 
\ u \cdot n|_{\Gamma} = 0 \}.
\]
For shortness set $E=\ker {\mathcal A}_d$ and
introduce also   its orthogonal complement $H_0$ in $H$. 

It is easy to check that 
\[
\left< {\mathcal A}_d ( w, s ), ( u, r) \right>_H = 0 \ \mbox{for any}\ 
(w, s) \in {\mathcal D}({\mathcal A}), \ (u,r) \in E;
\] 
in particular, the semigroup  associated with (\ref{fluide}) leaves both $E$ and $H_0$ invariant. 
Consequently, the decay property (\ref{stab}) may only hold  for   initial data emenating from the set $H_0$.

\medskip

The following observation can be shown by a simple density argument:

\begin{lem}
The solution $(u,r)$ of \rfb{fluide} with
initial datum in ${\mathcal D}(\mathcal{A}_d)$ satisfies
\be
\label{deriveeenergy}
E'(t) = - \int_\Omega b\left|u \right|^2dx\leq0.
\ee
Therefore the energy is non-increasing and \rfb{energyid} holds for all initial datum in $H$.
\end{lem}

As already shown in the above, the strong stability result (\ref{stab}) may hold only if 
we take the initial data
\[
(u^0, r^0) \in H_0 = \ker [{\mathcal A}_d]^\perp. 
\]
There are several ways how to show (\ref{stab}), here
we make use of the following result
due to Arendt and Batty \cite{arendt:88}:

\begin{thm}\label{thmArendtBatty}
Let $(T(t))_{t\geq0}$ be a bounded $C_0$-semigroup on a reflexive Banach space $X$. Denote by $A$ the generator of $(T(t))$ and by $\sigma(A)$ the spectrum of $A$. If $\sigma(A)\cap i\mathbb{R}$ is countable and no eigenvalue of $A$ lies on the imaginary axis, then $\ds \lim_{t\rightarrow+\infty} T(t)x = 0$ for
all $x\in X$.
\end{thm}

In view of this theorem we need to identify the spectrum of
${\mathcal A}_d$ lying on the imaginary axis, and we have according to \cite{AFN}:

\begin{itemize}
\item
Suppose that $|\omega| > 0$.
If $\lambda$ is a non-zero real number, then $i\lambda$  is not an eigenvalue of ${\mathcal A}_d$.
\item
Suppose that $|\omega| > 0$.
If $\lambda$ is a non-zero real number, then $i\lambda$ belongs to the resolvent set $\rho({\mathcal A}_d)$ of ${\mathcal A}_d$.
\end{itemize}

Now, Theorem \ref{thmArendtBatty} leads to
\begin{cor}[\cite{AFN}]
\label{cconv}
Let $(u, r)$ be the unique semigroup solution of the problem (\ref{fluide}) emanating from the initial data 
$(u^0,r^0) \in H$. Let $P_E$ be the orthogonal projection onto the space $E = \ker [{\mathcal A}_d ]$ in $H$, and let 
\[
(w,s) = P_E (u^0, r^0).
\] 
Then
\[ 
\| (u,r)(t, \cdot) - (w, s) \|_{H} \to 0 \ \mbox{as}\ t \to \infty
\]
\end{cor}

Next we give uniform decay result under 
 the Geometric Control Condition (GCC). A precise definition is given in 
Definition~\ref{def: precise GCC} and this  requires a description of geometry. 
This is done in Section~\ref{sec: Precise description of Geometric Control Condition}. In the following a definition 
for flat Laplace operator is given.
%
%
\begin{defi}
	\label{def: GCC flat}
Let  $\Omega$ be a smooth open domain in $\R^d$.
Let $\omega$ be an open subset of $\Omega$. We assume the geodesic associated with $\Delta$ are not tangent 
to infinity order to the boundary. We say that $\omega $ satisfies (GCC) if  every generalized bicharacteristics  
starting from $\Omega$ meet $\omega$.
\end{defi}
%
%
\begin{remark}
We give in Annex~\ref{sec: Geometry} a precise description of the generalized bicharacteristics. 
Here we give an idea of this notion in case of flat Laplace operator. 
In $\Omega$ a generalized bicharateristic is an integral curve of $H_p$, it is a straight line for $p(\xi)=|\xi|^2$.  When
such a curve hit the boundary, if the curve and the boundary are transverse, the generalized bicharacteristic 
is extended  following the law of reflection in geometrical optics.  If the curve is tangent to the boundary, 
essentially either the  curve stays in $\Omega$ except at intersection between curve and boundary and the 
generalized bicharacteristic is the curve or the curve go out $\Omega$ and the generalized bicharacteristic
glide on boundary. 
\end{remark}

We now state the main results of this article. 
%
%
\begin{thm} \label{princr}
Let $(u, r)$ be the unique semigroup solution of the problem (\ref{fluide}) emanating from the initial data 
$(u^0,r^0) \in H$. Let $P_E$ be the orthogonal projection onto the space 
$E = \ker [{\mathcal A}_d ]$ in $H$, and let 
$(w,s) = P_E (u^0, r^0)$. We assume that the interior of the support of $b$, $\omega = \left\{x \in \Omega; b(x) > 0 \right\}$, satisfies the
geometric control condition (GCC).
Then,
\be
\label{exp}
\| (u,r)(t, \cdot) - (w, s) \|_{H}\le   C \, e^{-c t} \,  \|(u^0, r^0) \|_{H}, \, \forall \, t \geq 0,
\ee
for some $C, c >0$ independent of $ (u^0, r^0)$.
\end{thm}

Before proving the previous theorem, we need some tools in Riemannian geometry, pseudo-differential calculus. 
Our goal is to apply characterization of exponential decay by an estimate of the resolvent on imaginary axis. To prove
such an estimate we use semiclassical measure. In next section we gives all the tools used in main part of the proof.

%
%
\section{Tools}
	\label{Sec: Tools}

In this section we give some  tools on Riemannian geometry, pseudo-differential calculus, stabilization and invariant 
form of acoustic system. If theses tools are classical, this introduction allows us to fix the notation used in this article.
%
%

\subsection{Riemannian geometry tools}
 	\label{sec: Riem tools}
 In this section we introduce notation used in this article. 
 Let $M$ be a manifold and let $g$ be a metric on $M$, $g$ is a inner product  acting on the vector fields.
 In local coordinates, $g=(g_{jk})_{1\le j,k\le d}$, where $g$ is symmetric, i.e. $g_{jk}=g_{kj}$ and $g$ depends on 
 $x\in M$. If $v=(v^j)_{1\le j\le d}$ and $w=(w^j)_{1\le j\le d}$ are vector fields, for each $x\in M$
 the inner product is given by $g_{jk}(x) v^j(x)w^k(x) $, where we use here and in what follows, the Einstein notation 
 i.e. $g_{jk} v^jw^k= \ds \sum_{1\le j,k \le d}g_{jk} v^jw^k$. We denote by $(.,.)_g$ the $L^2$ inner product on $L^2(M)$.
 The local form depends on the objects used, in what follows the inner product is considered on functions or on vector 
 fields. In local coordinates, if $f$ and $q$ are functions on $M$, 
 \begin{equation}
 (f,q)_g=\int_Mf(x)\overline{q(x)} |g|^{1/2}dx,
 \end{equation}
 where $ |g|=|\det g|$. In fact $ |g|^{1/2}dx$ is the local form of the Riemannian density associated with the metric $g$.
 It is a geometric object, invariant under change of variables.
 
 If $v$ and $w$ are vector fields, in local coordinates, we have
 \begin{equation}
 	\label{eq: inner product}
 (v,w)_g=\sum_{1\le j,k \le d}\int_Mg_{jk}v^j(x)\overline{w^k(x)} |g|^{1/2}dx.
 \end{equation}
 In what follows, we denote some time $(v^j,v^k)_g$ instead of $ (v,w)_g$, in particular when for instance $v^j$ 
 is given by a formula for each $j$. The notation $(v^j,v^k)_g$ must be interpreted as the right hand 
 side of~\eqref{eq: inner product}.
 
 We denote by $(g^{jk})_{1\le j,k \le d}$ the inverse of  $g$ and we have $g^{jk}g_{k\ell}=\delta_\ell^j$, where 
 $\delta_\ell^j$ is the Kronecker symbol.
 
 For a function $f$, we can define its covariant derivative given,  in local coordinates,  by the following formula
 \begin{equation}
	 \label{for: grad}
( \nabla_gf)^j= g^{jk}\pd_kf,
 \end{equation}
and $\nabla_gf$ is a vector field. 

For a vector field $v$, we can define its divergence given, in local coordinates, by the following formula
\begin{equation}
	\label{for: div}
\div_g v= |g|^{-1/2}\pd_k ( |g|^{1/2}v^k).
\end{equation}
The operators $\nabla_g$ and $\div_g$ are dual. If $f$, a function,  $v$, a vector field, one of  both
 supported in interior of 
$M$, we have the formula
\begin{equation}
	\label{eq: integrate parts without bd}
(f,\div_g v)_g=-(\nabla_g f, v)_g.
\end{equation}
To study the problem in a neighborhood of boundary, we can use
geodesic coordinates. In these coordinates, the boundary is locally defined by $x_d=0$, the domain by $x_d>0$, 
and the metric takes the following special form, $g_{jd}= 0$ for $1\le j\le d-1$, and $g_{dd}=1$.
We denote by $\tilde g_{jk}=g_{jk}$ for $1\le j,k\le d-1$, and $( \tilde g^{jk}) _{1\le j,k\le d-1}$ the inverse of
$( \tilde g_{jk}) _{1\le j,k\le d-1}$. It is the restriction of the metric $g$ on the hypersurfaces $x_d= constant$.
If we write, for instance $ \tilde g^{jk}\xi_k$ means the sum is on  $k$ 
between $1$ and $d-1$.

If $f$, a function, and $v$, a vector field, are supported in a neighborhood of the boundary, the integration by parts 
formula takes the form
\begin{equation}
	\label{eq: integration parts boundary}
(f,\div_g v)_g=-(\nabla_g f, v)_g - (f_{|x_d=0}, v^d_{|x_d=0})_\pd,
\end{equation}
where $(q,r)_\pd=\int_{\R^{d-1}}q(x')\overline{r(x')}|g|_{|x_d=0}^{1/2}dx'$.
Observe that in geodesic coordinates, the exterior unit  normal at the boundary is the vector $\nu=(0,\dots,0,-1)$, and
$-v^d$ is the inner product between $\nu $ and $v$. In fact we can give an invariant form to the integration formula
given above but we do not use this invariant form in this article.
We also use in local coordinates the following integration by parts formula.
\begin{equation}
	\label{eq: integration parts functions}	
(w_1,hD_d w_2)=(hD_d  w_1,w_2)- ih((w_1)_{|x_d=0},(w_2)_{|x_d=0})_\pd.
\end{equation}
 In fact we can give an invariant form to the integration formula
given above but we do not use this invariant form in this article.

%
%

 \subsection{Pseudo-differential tools}
 
 We recall in this section the notation and main definitions on pseudo-differential operators
 and semi-classical calculus. Let $a(x,\xi')\in \Con(\R^d\times\R^{d-1})$, we say $a\in S^m$ if for all $\alpha\in\N^d$ and $\beta \in\N^{d-1}$ there exist $C_{\alpha, \beta}$ such that
 \begin{equation}
 |\pd_{x}^\alpha\pd_{\xi'}^\beta a(x,\xi')|\le C_{\alpha, \beta}\langle \xi'\rangle^{m-|\beta|}, \text{ for all } x\in\R^d, \xi'\in\R^{d-1},
 \end{equation}
 where $ \alpha=(\alpha_1,\dots,\alpha_d)$, $|\alpha|=\alpha_1+\cdots+\alpha_d$, 
 $\pd_{x}^\alpha=\pd_{x_1}^{\alpha_1}\cdots\pd_{x_d}^{\alpha_d}$, and $\langle \xi'\rangle=(1+|\xi'|^2)^{1/2}$.
 
An operator is associated with a symbol by the following formula
 \begin{equation}
 \ops (a)u=\frac{1}{(2\pi)^{d-1}}\int_{\R^{d-1}}e^{ix'\xi'}a(x,h\xi')\hat u(\xi',x_d)d\xi',
 \end{equation}
where  $u$ is  in Schwartz  space,  and 
\[
\hat u(\xi',x_d)=\int_{\R^{d-1}}e^{-ix'\xi'}u(x',x_d)dx'.
\] 
The symbolic calculus is given by the following results.
Let $a\in S^m$ and $b\in S^q$.
\subsubsection{Composition}
There exists $c\in S^{m+q}$ such that  $\ops(a) \ops(b)=\ops(c)$  
 and there exists $c_{m+q-1}\in S^{m+q-1}$ such that 
\(
c=ab+hc_{m+q-1}.
\)
\subsubsection{Commutator}
There exists $d\in S^{m+q-1}$ such that $[\ops(a),\ops(b)]= h \ops(d)$ and there 
exists $d_{m+q-2}\in S^{m+q-2}$ such that $d=-i\{ a ,b\}+hd_{m+q-2}$.

We also  have an exact formula if $a$ is a differential operator, and we shall use the following formula
$[-ih\pd_d,\ops(b)]=-i h \ops(\{\xi_d,b\})$. Observe that $\{\xi_d,b\}=\pd_db$ which is a tangential symbol.
 \subsubsection{Adjoint}
 There exists $ \tilde c\in S^m$ such that $\ops(a)^*=\ops(\tilde c)$ and there exists
 $\tilde c_{m-1}\in S^{m-1}$ such that $\tilde c=\bar a+h c_{m-1}$.
\subsubsection{Operators bounded  on $L^2$}
If $a\in S^0$ then there exists $C>0$ such that $\|\ops(a)u\|_{L^2}\le C\|u\|_{L^2}$ for all $u$ is  in Schwartz  space.
Observe from this formula we can extend by density  $\ops(a) $ to an bounded operator on $L^2$.
We keep the same notation if $\ops(a)$ act on a function depending only on $x'$.
In this case the estimate take the following form \( |\ops(a)u|_{L^2}\le C|u|_{L^2} \). When no confusion is possible, 
we shall write $\|.\|$ instead of $\|.\|_{L^2}=\|.\|_{L^2(\R^d)}$ and $|.|$ instead of $|.|_{L^2}=|.|_{L^2(\R^{d-1})}$.

The constant $C$ in the estimation of $L^2$-norm depends on a finite numbers of semi-norms of the symbol $a$.
We have a precise result for $h$ sufficiently small.

Let $a\in S^0$, there exists $C>0$ depending on a finite  numbers of semi-norms of the symbol $a$ such that
\begin{equation}
	\label{est: norm L2}
	\| \ops(a) u\|\le 2\| a\|_\infty \| u\|+C \| u\|h, \text{ for all } h\in (0,1),
\end{equation}
where $\| a\|_\infty=\sup_{x\in \R^d, \xi'\in\R^{d-1}}|a(x,\xi')|$. Observe that $2$ in this estimate is arbitrary 
and may be replaced by 
$1+\delta$ for all $\delta>0$ but in this case $C$ also depends on $\delta$.
 
 For $s\in \R$, let $\Lambda^s=\ops(\langle \xi'\rangle^{s})$. We deduce from $L^2$ bound an estimation for operators 
 with symbol $a\in S^m$, we have 
 \[
 \|\ops(a)u\| \lesssim \|\Lambda^m u\|.
 \]
 We can define $H^s_{sc}$-norm by $\| u\|_{H^s_{sc}}=\|\Lambda^s u  \|$.
 \subsubsection{G\aa rding inequality}

 We state the G\aa rding inequality adapted to our purpose.
We assume there exist $K\subset \R^d$ a compact, $W\subset\R^{d-1}$ satisfying
 for a fix $R_0>0$, if $|\xi|\ge R_0$ then for all $\lambda\ge1$, $\lambda \xi\in W$.  Let $m\in\R$ and  $a\in S^m$, 
we assume there exists $C_0>0$ such that  $a$ satisfies $a(x,\xi')\ge C_0 |\xi'|^m$ for all $x\in K$,  
$\xi'\in W\cap\{\xi',\ |\xi'|\ge R_0 \}$. Let $\chi\in S^0$ supported in $K\times W$, there exist $C_1>0$ and 
$C_2>0$ such that 
 \begin{equation}
 \label{Garding ineq}
\Re (\ops(a)\ops(\chi)u,\ops(\chi)u)\ge C\| \Lambda^{m/2}\ops(\chi) u \|^2 -C_2h\| \Lambda^{(m-1)/2}u\|^2,
 \end{equation}
 for all $u $ in Schwartz space.

%
%
\subsection{Uniform stabilization}
 
The proof of the first part of Theorem~\ref{princr}, is based on Gearhart-Huang-Prüss  result we recall below.
\begin{thm}
	\label{GHP}
Let $A$ be a generator of a semigroup $e^{tA}$ on $H$. We assume the following conditions
\begin{enumerate}
\item $\exists M>0, \ \| e^{tA}\|_{{\mathcal L}(H)}\le M$ for all $t\ge 0$,
\item $A-i\mu I$ is invertible on $H$  for all $\mu\in\R$,
\item  $\exists M>0, \ \| (A-i\mu I)^{-1}\|_{{\mathcal L}(H)}\le M$.
\end{enumerate}
Then there exist $\lambda>0$ and $C>0$, such that $ \ \| e^{tA}\|_{{\mathcal L}(H)}\le Ce^{-\lambda t}$ for all $t\ge 0$.
\end{thm}
This result is classical and we refer to \cite{Gearhart-1978}, \cite{Pruss-1984}, 
\cite[Theorem 3]{Huang85}, \cite{EN-2000} for a proof. 
 In our case, the first condition is satisfied as energy is non increasing. 
 The first item is a consequence of the energy decaying.
 The second is not true for $\mu=0$ as $ {\mathcal A}_d$ as a non trivial kernel. To avoid this difficulty we only have to 
 consider $ {\mathcal A}_d$ as an operator on $\tilde H=H/\ker {\mathcal A}_d$. On this space  $A-i\mu I$ is invertible for all
 $\mu \in\R$, see \cite{AHR} for a proof. In \cite{AHR} we prove the following estimate $\| ({\mathcal A}_d-i\mu I)^{-1}\|_{{\mathcal L}(\tilde H)}\le Me^{C|\mu|}$ for some constants $M>0$ and $C>0$. In particular, for all $C>0$, there exists
 $M>0$ such that  $  \| (A-i\mu I)^{-1}\|_{{\mathcal L}(H)}\le M$ for every $\mu\in [-C,C]$. This means, only large $|\mu|$
 have to be considered to prove the third item of Theorem~\ref{GHP}.

%
%

 \subsection{Invariant form of acoustic system}
The generator of semigroup of acoustic system has been defined above (see~\eqref{cauchy} and \eqref{cauchybis}).
It is useful to introduce the following system
\be
\label{fluidec-1}
\left\{
\begin{array}{l}
-\nabla_g r-i\mu u -bu= f, \,\hbox{ in } \Omega , \\
- \div_g u-i\mu r= q, \,\hbox{ in } \Omega , 
\end{array}
\right.
\ee
where $(f,q)\in H= (L^2(\Omega))^d \times L^2_m(\Omega)$, and $(u,r)\in  {\mathcal D}({\mathcal A})$.
From geometric point of view, $u$ is a vector field or equivalently a 1-contravariant tensor
and  $g$ is the euclidian metric. The form~\eqref{fluidec-1} allows us to change variables and 
to follow the different geometric quantities, $g$ and $u$.
With the previous notation it is equivalent to estimate $({\mathcal A}_d-i\mu)^{-1}$ and estimate $(u,r)$ by $(f,q)$.
Observe the space $H$ is decomposed in direct sum $\ker {\mathcal A}_d\oplus (H/\ker {\mathcal A}_d )$, and it is
easy to see that an estimate of  $({\mathcal A}_d-i\mu)^{-1}$ on $H$   for $\mu\ne0$ implies an estimate 
on $({\mathcal A}_d-i\mu)^{-1}$ restricted to $ (H/\ker {\mathcal A}_d) $. In what follows, we do not consider the space
 $ (H/\ker {\mathcal A}_d )$ as it must be introduced only to avoid the problem coming from the fact that ${\mathcal A}_d$
 has a non trivial null space. In geodesic coordinates the system keep the same form, with the metric $g$ satisfying 
 properties view in Section~\ref{sec: Riem tools}.

To use semi-classical tools, we set $h=1/|\mu|$ and multiplying System~\eqref{fluidec-1} by $ih$ we obtain
\be
\label{fluidec-2}
\left\{
\begin{array}{l}
-ih\nabla_g r+\varpi u -ihbu=ih f, \,\hbox{ in } \Omega, \\
-ih \div_g u+\varpi r= ihq, \,\hbox{ in } \Omega, 
\end{array}
\right.
\ee
where $\varpi=\mu/|\mu|=\pm 1$. In the following we only consider case $\varpi=+1$, in fact changing $u$ in $-u$ and 
$q$ in $-q$, in System~\eqref{fluidec-2} with $\varpi=+1$, we have the system with $\varpi=-1$ and $b$ 
is changed in $-b$.
Observe that the sign of $b$ is important to have the decay of energy but plays no role in the proof of the estimate for
resolvent. Nevertheless $b$ should be non negative or non positive, here it is  the case.
 
 %
 %
 
 \section{Introduction of a semiclassical measure}
 	\label{Sec: Introduction of a semiclassical measure}

The goal is to prove the following estimate, 
there exists $C>0$ such that
  \begin{equation}
  	\label{est: unif}
 \|  u\|_{(L^2(\Omega))^d} + \| r \|_{L^2(\Omega)} \le C(   \| f \|_{(L^2(\Omega))^d} + \|  q\|_{L^2(\Omega)} ),
 \end{equation}
for all $(f,q)\in (L^2(\Omega))^d\times L^2(\Omega)$ and $(u,r)\in (L^2(\Omega))^d\times L^2(\Omega)$ solution of 
System~\eqref{fluidec-1} or equicalently System~\eqref{fluidec-2}. 

We prove Estimate~\eqref{est: unif} by contradiction. If Estimate~\eqref{est: unif}  is false, there exist a 
sequence $(h_j)_{j\in \N} $ and $(f_{h_j},q_{h_j})$ and $ (u_{h_j},r_{h_j})$ satisfying Estimate~\eqref{est: unif}
and Equation~\eqref{fluidec-2} with $h_j$ instead of $h$ such that 
\begin{align}
	\label{est: unif 2}
&h_j\to 0 \text{ as } j\to +\infty,\\
& \|  u_{h_j}\|_{(L^2(\Omega))^d} + \| r _{h_j}\|_{L^2(\Omega)} =1,\notag\\
& \| f_{h_j} \|_{(L^2(\Omega))^d} + \|  q_{h_j}\|_{L^2(\Omega)}\to 0 \text{ as } j\to +\infty.\notag
\end{align}
As view above Estimate~\eqref{est: unif} is true for $|\mu|$ bounded, thus only $|\mu|$ goes to $+\infty $ can
give such sequences.
We can always normalized  $ (u_{h_j},r_{h_j})$ by homogeneity. In what follows, we do not write the index $j$ and 
we write $h$ instead of $h_j$. The reader must keep in mind that $h$ is in fact a sequence.
Finally with these reductions and notation, we have sequences $h$, $(u_h,r_h)$ and $(f_h,q_h)$, such that 
\be
\label{fluidec-3}
\left\{
\begin{array}{l}
-ih\nabla_g r_h+ u_h -ihbu_h=ih f_h, \,\hbox{ in } \Omega, \\
-ih \div_g u_h+ r_h= ihq_h, \,\hbox{ in } \Omega, \\
u_h. n=0 \text{ on } \pd\Omega
\end{array}
\right.
\ee
and 
\begin{align}
& \|  u_{h}\|_{(L^2(\Omega))^d} + \| r _{h}\|_{L^2(\Omega)} =1,\\
& \| f_{h} \|_{(L^2(\Omega))^d} + \|  q_{h}\|_{L^2(\Omega)}\to 0 \text{ as }h\to 0.\notag
\end{align}

 %
 %
 \subsection{Semiclassical measure} 
 
 To define a semiclassical measure associated with $\gamma_h=(  {u}_h,{r}_h)$ we have to prove that 
 $\gamma_h$ converges weakly to 0. To prove that we decompose $\gamma_h$ in a Littlewood-Paley type series.
 
 In the following we use functional calculus, we refer to \cite{R-S-1980} in particular Theorem VIII.6 for precise statements.
 
 Let $\psi $ be a $C^\infty(\R)$ cutoff function such the support of  $\psi $  is contained in $|s|\le 2$ and
$\psi(s)=1$ for $|s|\le 1$. We have 
\[
1=\lim_{N\to \infty}\psi(2^{-N}s)=\psi(s)+\sum_{k=1}^\infty(\psi(2^{-k}s)-\psi(2^{-k+1}s))=\psi(s)+\sum_{k=1}^\infty\varphi(2^{-k}s),
\]
where $\varphi (s)=\psi(s)-\psi(2s)$. 

Here we use the notation introduced in Introduction for ${\mathcal A}$ and ${\mathcal B}$. As $i{\mathcal A}$ is selfadjoint, we have 
\begin{align}
	\label{form: identity decomp}
Id&=\psi(iB^{-1}h{\mathcal A})+\sum_{k=1}^\infty\varphi(2^{-k}iB^{-1}h{\mathcal A})  \\
&=\psi(iBh{\mathcal A})+\sum_{k=1}^\infty\varphi(2^{-k}iB^{-1}h{\mathcal A})+\psi(iB^{-1}h{\mathcal A})-\psi(iBh{\mathcal A}),\notag
\end{align}
where $B>0$ will be chosen sufficiently large.
%
%
\begin{lem}
	\label{lem: L-W decomp}
Let $(u_h,r_h)$ satisfying  \eqref{fluidec-3}, let $B>0$ be  sufficiently large, let $\gamma_h=(u_h,r_h)^T$, 
we have 
\begin{align*}
&\| \psi(iBh{\mathcal A})\gamma_h \|\le Ch(\| f_h\| +\|q_h \|+\|\gamma_h  \|)  \\
&\sum_{k=1}^\infty\| \varphi(2^{-k}iB^{-1}h{\mathcal A})\gamma_h \| \le Ch(\| f_h\| +\|q_h \|+\|\gamma_h  \|).
\end{align*}
\end{lem}
\begin{proof}
We can write \eqref{fluidec-3}, 
\[
-ih{\mathcal A}\gamma_h+\gamma_h-ih{\mathcal B}{\mathcal B}^*\gamma_h=h F_h
\]
with previous notation ${\mathcal B}^*=(\sqrt b,0)$ and $F_h=
i\begin{pmatrix}
f_h\\ q_h
\end{pmatrix}$.
Applying $ \psi(iBh{\mathcal A})$ on both side we have
\[
-iB^{-1}\tilde\psi(iBh{\mathcal A})\gamma_h+\psi(iBh{\mathcal A})\gamma_h
-ih\psi(iBh{\mathcal A}){\mathcal B}{\mathcal B}^*\gamma_h=h \psi(iBh{\mathcal A})F_h,
\]
where $\tilde\psi(s)=s\psi(s)$.
We obtain
\[
\| \psi(iBh{\mathcal A})\gamma_h \| \lesssim h\|\gamma_h  \|+h\|F_h  \|+B^{-1}\|  \tilde\psi(iBh{\mathcal A})\gamma_h \|.
\]
As $\tilde\psi(s)=s\chi(s)\psi(s) $ where $\chi$ is a smooth cutoff function, such that $\chi(s) = 1 $ for $|s|\le 2$.
We have 
\[
\|  \tilde\psi(iBh{\mathcal A})\gamma_h \|\lesssim \| \psi(iBh{\mathcal A})\gamma_h \| .
\]
This proves the first estimate choosing $B$ sufficiently large for absorbing this term with the left hand side.

We now apply $\varphi(2^{-k}iB^{-1}h{\mathcal A})$ on both side of  equation on $\gamma_h$, we have
\[
-i2^kB\tilde\varphi(2^{-k}iB^{-1}h{\mathcal A})\gamma_h+\varphi(2^{-k}iB^{-1}h{\mathcal A})\gamma_h
-ih\varphi(2^{-k}iB^{-1}h{\mathcal A}){\mathcal B}{\mathcal B}^*\gamma_h=h \varphi(2^{-k}iB^{-1}h{\mathcal A})F_h,
\]
where $\tilde \varphi(s)=s\varphi(s)$ and we obtain the estimate
\begin{equation}
	\label{eq: estimation tilde phi}
2^kB\| \tilde\varphi(2^{-k}iB^{-1}h{\mathcal A})\gamma_h  \|\lesssim  \|\varphi(2^{-k}iB^{-1}h{\mathcal A})\gamma_h  \|
+ h \| F_h \|+ h\|\gamma_h\|.
\end{equation}
On the support of $\varphi$, we have $1/2\le |s| \le 2$, thus $\varphi(s)\le  2 |s|\varphi(s)$. From functional calculus
$\|\tilde \varphi(\alpha ih{\mathcal A}) \|= \||\alpha ih{\mathcal A}|\varphi(\alpha ih{\mathcal A}) \|$ for all $\alpha>0$.
We deduce from~\eqref{eq: estimation tilde phi} and functional calculus.
\[
2^kB\| \varphi(2^{-k}iB^{-1}h{\mathcal A})\gamma_h  \|\lesssim  \|\varphi(2^{-k}iB^{-1}h{\mathcal A})\gamma_h  \|
+ h \| F_h \|+ h\|\gamma_h\|.
\]
If $B$ is sufficiently large, we can absorb the term $ \|\varphi(2^{-k}iB^{-1}h{\mathcal A})\gamma_h  \|$ by the left hand side 
and summing over $k$ we obtain the second estimate of the lemma.
\end{proof}
%
%
\begin{lem}
The  sequence $(u_h,r_h)$ converges weakly to $0$.
\end{lem}
\begin{proof}
Let $\chi$ be defined by $\chi(s)=\psi(B^{-1}s)-\psi(Bs)$. From \eqref{form: identity decomp} and 
Lemma~\ref{lem: L-W decomp}, we have
\begin{equation}
\label{eq: semiclassical form}
\gamma_h=\chi(ih{\mathcal A})\gamma_h+O(h).
\end{equation}
For $v\in L^2(\Omega)^{d+1}$, we have
\[
(\chi(ih{\mathcal A})\gamma_h,v)=(\gamma_h,\chi(ih{\mathcal A})v).
\]
The sequence $(\gamma_h)_h$ is bounded and $(\chi(ih{\mathcal A})v)_h$ converges to 0. Indeed
from functional calculus
\[
\| \chi(ih{\mathcal A})v\|^2=\int  \chi^2(h\lambda)(dE_\lambda v,v),
\]
where $(E_\lambda)_\lambda$ is the  family of projector associated to $i{\mathcal A}$. In particular $(dE_\lambda v,v)$
is a bounded measure. As $ \chi^2(h\lambda)$ converges every where to 0 and $ \chi^2\le 1$, dominated convergence theorem
implies  $\int  \chi^2(h\lambda)(dE_\lambda v,v)$ converges to 0. This proves the result
\end{proof}

As $u_h$ and $r_h$ are bounded functions in $L^2(\Omega)$ we extend these functions by 0 in exterior of $\Omega$,
 and we denote respectively by  $\underline {u}_h$ and  $\underline{r}_h$ these extensions. Clearly 
 $( (\underline{u}_h, \underline{r}_h)$ converges weakly to 0 in $(L^2(\R^d))^d\times L^2(\R^d)$.

 Let $a_j^k(x,\xi)$, $b^\ell(x,\xi)$, $c(x,\xi)$ symboles  of order 0. We shall consider $b=(b^\ell)$ as a vector field (or a contravariant vector field) and $a=(a_j^k) $ as a 1-covariant, 1-contravariant tensor. We assume $a$ symmetric in the 
 following sense, $g^{\ell j}a_j^k=g^{k j}a_j^\ell$ and this relation is equivalent to $g_{\ell j}a^j_k=g_{k j}a^j_\ell$. Both
 formulas are useful.

As $( (\underline{u}_h, \underline{r}_h)$ converges weakly to 0,  up to a subsequence of $h$, it is possible to  define the semiclassical measure by the following relation,
 \begin{align}
 	\label{def: measure}
& ((\ops(a_j^k) \underline {u}_h^j) ,  (\underline{u}_h^\ell) )_g+  ((\ops(b^\ell  ) \underline{r}_h) ,(\underline{u}_h^j) )_g+(  g_{jk}\ops(b^j)\underline{u}_h^k,\underline{r}_h)_g +
(\ops(c)\underline{r}_h,\underline{r}_h)_g  \notag\\
&\text{converges to }\\
&\langle \mu_j^k,a_k^j\rangle+2\langle \nu_j,b^j\rangle +\langle m, c \rangle  \text{ as } h \text{ goes to 0,}\notag
 \end{align}
 where the measure $(\mu_j^k)_{1\le j,k\le d}$ is symmetric in the sense $g^{\ell j}\mu_j^k=g^{k j}\mu_j^\ell$ which is
 related with the symmetry on $(a_j^k)_{1\le j,k\le d}$.
 From \eqref{fluidec-3}, the functions $\underline{u}_h$ and $\underline{r}_h$ satisfy the system 
 \begin{equation}
 	\label{eq: syst semi-class}
 \begin{cases}
  &\underline{u}_h-ih\nabla_g \underline{r}_h=hf_h -ih (r_h)_{|\pd \Omega} \nabla_g(\chi_\Omega),\\
& -ih\div_g \underline{u}_h+\underline{r}_h=h q_h,
 \end{cases}
 \end{equation}
 where $\nabla_g(\chi_\Omega)$ is a measure on $\pd \Omega$. We shall give in what follows a form of 
 $\nabla_g(\chi_\Omega)$ in  adapted local coordinates. As $u_h. n=0$ at boundary, second equation does not have
 measure terms at the right hand side.
  \begin{prop}
 	\label{Prop: measure support and form}
 The measures $\mu_j^k$, $\nu_k$ and $m$ are supported on the characteristic set 
 \[
\{(x,\xi)\in \R^d\times\R^d,\     (g^{jp}(x)\xi_p  \xi_j-1)=0 \text{ and } x\in \overline{\Omega}\}, 
 \]   
   and satisfy the following equations
 \begin{align*}
   &\nu_k=-\xi_k m, \\
 &  \mu_k^j = g^{jp}\xi_p \xi_k m \text{ for all } 1\le j,k\le d.
 \end{align*}
 \end{prop}
 Obviously, the measures are null  in a neighborhood of $(x,\xi)$ if $x\in\R^d\setminus \overline{\Omega}$ as   
  $(\underline{u}_h,\underline{r}_h)$ is null in a neighborhood of $x$.
We describe first the semiclassical measure in $\Omega$. The description in a neighborhood of the boundary 
 will be done in the following. 
 For  symbols $ a_j^k $ compactly supported in $\Omega$, we consider 
 \begin{equation}
 	\label{def: I1 measure description}
 I_1(h)=\big((\ops(a_j^k)(\underline{u}_h^j-ih(\nabla_g \underline{r}_h)^j)) , (\underline{u}_h^\ell)\big)_g
 +\big((\ops(a_j^k)\underline{u}_h^j), (\underline{u}_h^\ell-ih(\nabla_g \underline{r}_h)^\ell)\big)_g.
 \end{equation}
 From System~\eqref{eq: syst semi-class}, we have $I_1(h)$ goes to 0 as $h$. Observe that 
 $-ih(\nabla_g )^j\underline{r}_h=\ops(g^{jk})\xi_k)\underline{r}_h+hV\underline{r}_h$ where $V$ is bounded and 
 $-ih\div_gw=\ops(\xi_k)w^k+hWw$ where $W$ is bounded. In what follows we denote 
 by $O(h)$ all terms bounded by $h$ in $L^2$-norm.
 From integration by parts formula and symbolic calculus, we have 
 \begin{align*}
 &\big((\ops(a_j^k)(-ih(\nabla_g \underline{r}_h)^j)) , (\underline{u}_h^\ell)\big)_g= 
  \big((\ops(a_j^kg^{jp}\xi_p)\underline{r}_h), (\underline{u}_h^\ell)\big)_g+O(h)\\
 &\big((\ops(a_j^k)\underline{u}_h^j), (-ih(\nabla_g \underline{r}_h)^\ell)\big)_g
 =\big(-ih\div_g(\ops(a_j^k)\underline{u}_h^j),  \underline{r}_h\big)_g
 = (\ops(a_j^k\xi_k)\underline{u}_h^j,  \underline{r}_h)_g  +O(h).
 \end{align*}
 Then 
 \begin{equation*}
 I_1(h)=2\big((\ops(a_j^k)\underline{u}_h^j) , (\underline{u}_h^\ell)\big)_g
 +\big((\ops(a_j^kg^{jp}\xi_p)\underline{r}_h), (\underline{u}_h^\ell)\big)_g
 + \big(\ops(a_j^k\xi_k)\underline{u}_h^j,  \underline{r}_h\big)_g +O(h) .
 \end{equation*}
 We can apply \eqref{def: measure} as, by the symmetry on $a$, we have 
 $g_{qk}a_j^kg^{jp}\xi_p=g_{jk}a_q^kg^{jp}\xi_p= \delta_k^pa_q^k\xi_p=a_q^p\xi_p$ thus, the condition on the
 $b^j$ is satisfied.
 We obtain
 \[
2 \langle \mu_k^j  , a_j^k  \rangle + 2\langle  \nu _k, a_j^kg^{jp}\xi_p \rangle  = 0,
 \]
 for all  symbols $ a_j^k $ compactly supported in $\Omega$, then 
 \begin{equation}
 	\label{eq: first on measure}
  \mu_k^j +  g^{jp}\xi_p \nu _k=0 \text{ for all } 1\le j,k\le d.
 \end{equation}
 
  For  symbols $ b^k $ compactly supported in $\Omega$, we consider 
 \begin{align}
 	\label{def: I2 measure description}
 I_2(h)&=\big((\ops(b^k)(-ih\div_g \underline{u}_h+\underline{r}_h)) , (\underline{u}_h^\ell)\big)_g
 +\big(\ops(g_{jk}b^j) (\underline{u}_h^k-ih(\nabla_g \underline{r}_h)^k)   ,\underline{r}_h\big)_g  \\
 &\quad+\big((\ops(b^q)\underline{r}_h) ,  (\underline{u}_h^k-ih(\nabla_g \underline{r}_h)^k) \big)_g
 + \big(\ops(g_{jk}b^j) \underline{u}_h^k, -ih\div_g \underline{u}_h+\underline{r}_h\big)_g . \notag
 \end{align}
  From System~\eqref{eq: syst semi-class}, we have $I_2(h)$ goes to 0 as $h$. 
 From integration by parts formula and symbolic calculus, we have 
 \begin{align*}
 I_2(h)&=\big( (\ops(b^k\xi_j)\underline{u}_h^j) , (\underline{u}_h^\ell) \big)_g 
 +\big( (\ops(b^k)\underline{r}_h) ,   (\underline{u}_h^\ell)\big)_g 
 +\big( \ops(g_{jk}b^j) \underline{u}_h^k , \underline{r}_h \big)_g \\
 &\quad+\big( \ops(g_{jk}g^{kq}b^j\xi_q) \underline{r}_h  ,\underline{r}_h  \big)_g 
 +\big((\ops(b^q)\underline{r}_h) ,  (\underline{u}_h^k) \big)_g 
 +\big(  \ops(b^q\xi_q)\underline{r}_h , \underline{r}_h )_g   \\
 &\quad +\big(( \ops(g_{jk}b^jg^{pq}\xi_p) \underline{u}_h^k ), (\underline{u}_h^\ell) \big)_g 
 + \big(\ops(g_{jk}b^j) \underline{u}_h^k, \underline{r}_h\big)_g +O(h)\\
 &=\big( (\ops(b^k\xi_j+g_{jq}b^qg^{pk}\xi_p)\underline{u}_h^j) , (\underline{u}_h^\ell) \big)_g 
 +2\big( (\ops(b^k)\underline{r}_h) ,   (\underline{u}_h^\ell)\big)_g 
 +2\big( \ops(g_{jk}b^j) \underline{u}_h^k , \underline{r}_h \big)_g \\
 &\quad+2\big( \ops(b^j\xi_j) \underline{r}_h  ,\underline{r}_h  \big)_g+O(h).
 \end{align*}
 To apply  \eqref{def: measure} we have to respect some symmetry.
 We have 
 \[
 g_{k\ell} (b^k\xi_j+g_{jq}b^qg^{pk}\xi_p)=g_{k\ell} b^k\xi_j+g_{jq}b^q\xi_\ell,
 \]
 which is symmetric with respect $j$ and $\ell$.
 We then obtain, as 
 \[
  \langle \mu_k^j  ,  b^k\xi_j+g_{jq}b^qg^{pk}\xi_p \rangle +4 \langle \nu_k,  b^k \rangle + 2 \langle m  ,  b^k\xi_k \rangle 
  = \langle \xi_j\mu_k^j  ,  b^k\rangle
  + \langle g_{jq}g^{pk}\xi_p  \mu_k^j  ,  b^q\rangle +4 \langle \nu_k,  b^k \rangle + 2 \langle \xi_k  m  ,  b^k\rangle ,
 \]
 the following equation
 \begin{equation}
 	\label{eq: second on measure}
	\xi_j\mu_k^j   +  g_{jk}g^{pq}\xi_p  \mu_q^j  +4\nu_k +2  \xi_k  m =0 \text{ for all } 1\le k\le d.
 \end{equation}
   For  symbols $ c $ compactly supported in $\Omega$, we consider 
\begin{align}
	\label{def: I3 measure description}
I_3(h)=\big( \ops (c)(-ih\div_g \underline{u}_h+\underline{r}_h ),\underline{r}_h \big)_g  
+\big( \ops(c)  \underline{r}_h , -ih\div_g \underline{u}_h+\underline{r}_h \big)_g.
\end{align}
 From System~\eqref{eq: syst semi-class}, we have $I_3(h)$ goes to 0 as $h$ and integrating by parts we obtain
 \begin{align}
 I_3(h)= 2\big( \ops (c) \underline{r}_h ,\underline{r}_h  \big)_g 
 +\big( \ops (c\xi_j)\underline{u}_h^j,\underline{r}_h  \big)_g  
 +\big( \ops(cg^{kj} \xi_j)  \underline{r}_h ,  \underline{u}_h^j\big)_g  +O(h).
 \end{align}
Observe that $g_{pk}g^{kj}c\xi_j=c\xi_p$ we may apply~\eqref{def: measure} and we obtain
\begin{align}
	\label{eq: third on measure}
	m+g^{kj}\xi_j\nu_k=0.
\end{align}
 From~\eqref{eq: first on measure}, \eqref{eq: second on measure} and \eqref{eq: third on measure} we obtain 
 the following system on measure $\mu_k^j$, $\nu_k$ and $m$.
 \begin{equation}
 	\label{eq: syst on measures}
 \begin{cases}
  \mu_k^j +  g^{jp}\xi_p \nu _k=0 \text{ for all } 1\le j,k\le d,\\
 \xi_j\mu_k^j   +  g_{jk}g^{pq}\xi_p  \mu_q^j  +4\nu_k +2  \xi_k  m =0 \text{ for all } 1\le k\le d,\\
 m+g^{kj}\xi_j\nu_k=0.
 \end{cases}
 \end{equation}
 From symmetry on $(\mu_j^k)_{1\le j,k\le d}$ we may write the second equation
 \[
  \xi_j\mu_k^j   +  g_{jk}g^{jq}\xi_p  \mu_q^p  +4\nu_k +2  \xi_k  m =0 \text{ for all } 1\le k\le d,
 \]
thus we have 
 \[
 2 \xi_j\mu_k^j   +4\nu_k +2  \xi_k  m =0 \text{ for all } 1\le k\le d.
 \]
 Replace $\mu_k^j $ in this equation from the first equation of \eqref{eq: syst on measures} we have
 \begin{equation}
 	\label{eq: on mu and nu}
 - g^{jp}\xi_p  \xi_j\nu _k  +2\nu_k +  \xi_k  m =0 \text{ for all } 1\le k\le d.
 \end{equation}
 Multiplying this equation by $g^{kq}\xi_q$ and from the third equation of  \eqref{eq: syst on measures} which is 
  $ g^{jp} \xi_p\nu _k =-m$, this yields
 \[
  g^{jp}\xi_p  \xi_jm  -2m + g^{kq}\xi_q \xi_k  m =0.
 \]
 Then 
 \begin{equation}
   (g^{jp}\xi_p  \xi_j-1)m  =0 ,
 \end{equation}
 which implies $m$ supported on $  (g^{jp}\xi_p  \xi_j-1)=0$.
 If $  (g^{jp}\xi_p  \xi_j-1)\ne 0$, from \eqref{eq: syst on measures} and  as $m=0$ we have the following system 
 \begin{equation}
 	\label{eq: syst on measures 2}
 \begin{cases}
  \mu_k^j +  g^{jp}\xi_p \nu _k=0 \text{ for all } 1\le j,k\le d,\\
  \xi_j\mu_k^j   +2\nu_k  =0 \text{ for all } 1\le k\le d,\\
 g^{kj}\xi_j\nu_k=0.
 \end{cases}
 \end{equation}
 Multiplying the first equation by $g^{k\ell}\xi_\ell$ we obtain
 \[
g^{k\ell}\xi_\ell \mu_k^j + g^{k\ell}\xi_\ell g^{jp}\xi_p \nu _k=0 \text{ for all } 1\le j\le d.
 \]
 From the third equation of \eqref{eq: syst on measures 2} and the symmetry $g^{k\ell}\mu_k^j =g^{kj}\mu_k^\ell$, 
 we have
  \[
g^{kj}\xi_\ell \mu_k^\ell =0 \text{ for all } 1\le j\le d,
 \]
 and as $(g^{kj})_{1\le j,k\le d}$ is invertible we obtain $\xi_\ell \mu_j^\ell =0 \text{ for all } 1\le j\le d$ and  the second
 equation of \eqref{eq: syst on measures 2} implies $\nu_k=0$ for $1\le k\le d$ and the first $ \mu_k^j =0$ 
 for $1\le k,j\le d$.
 Thus all the measures $ \mu_k^j$, $\nu_k$ and $m$ are supported on $ (g^{jp}\xi_p  \xi_j-1)  =0$.
 
 Now from \eqref{eq: on mu and nu} as on the support of measures $g^{jp}\xi_p  \xi_j=1$, we have
 \[
 \nu_k=-\xi_k m, 
 \]
 and from the first equation of \eqref{eq: syst on measures} we have
 \[
 \mu_k^j = g^{jp}\xi_p \xi_k m \text{ for all } 1\le j,k\le d.
 \]
This achieves the proof of Proposition~\ref{Prop: measure support and form} in $\Omega$. 

Now we have to prove same result at the boundary. The proof is a consequence of the two following lemmas.
%
%
 \begin{lem} \label{lem: est tensor pd}
 Let $a$ be a symbol compactly supported, then there exists $C>0$, such that, we have 
 \[
 \|\ops(a)(w\otimes \delta_{x_d=0})\|\le C h^{-1/2}|w|_{L^2(\R^{d-1})}  ,
 \]
 for all $w$ in the Schwartz space of $\R^{d-1}_{x'}$.
 \end{lem}
%
%
 \begin{lem}
 	\label{lem: first est trace}
 Let $\underline{u}_h$ and $\underline{r}_h$ solution of System~\eqref{eq: syst semi-class} then 
$ \ds \lim_{h\to 0}h|\Lambda^{1/2}({r}_h)_{|x_d=0}|^2=0$.
 \end{lem}
With these lemmas we can prove Proposition~\ref{Prop: measure support and form}. Let us write 
System~\eqref{eq: syst semi-class} in geodesic coordinates. We have 
 \begin{equation}
 	\label{eq: syst semi-class-loc}
 \begin{cases}
  &\underline{u}_h-ih\nabla_g \underline{r}_h=hf_h -ih (r_h)_{|x_d=0}\otimes \delta_{x_d=0},\\
& -ih\div_g \underline{u}_h+\underline{r}_h=h q_h.
 \end{cases}
 \end{equation}
Following the proof given above, in $I_1(h)$ (see \eqref{def: I1 measure description}), we have two extra terms to control, 
\[(\ops(a) (h(r_h)_{|x_d=0}\otimes \delta_{x_d=0}), \underline{u}_h^j)  \text{ and}\]  
\[(\ops(a)( \underline{u}_h^j ), h(r_h)_{|x_d=0}\otimes \delta_{x_d=0})=  
( \underline{u}_h^j , \ops(a)^*(h(r_h)_{|x_d=0}\otimes \delta_{x_d=0})) .\]
 For $I_2(h) $ (see \eqref{def: I2 measure description}) we have to control same type of terms where 
 $\underline{u}_h^j$ is replaced by $ \underline{r}_h$. The term $I_3(h)$ (see \eqref{def: I3 measure description})
 does not have an extra boundary term.
As all these terms write analogous, we only treat  one term. We apply  Lemma~\ref{lem: est tensor pd}, we have
\[
| (\ops(a) (h(r_h)_{|x_d=0}\otimes \delta_{x_d=0}), \underline{u}_h^j)    |
\lesssim h^{1/2} \|  \underline{u}_h^j \| |(r_h)_{|x_d=0}|  \to 0 \text{ as }h,
\]
from Lemma~\ref{lem: first est trace}.
\begin{proof}[Proof of Lemma \ref{lem: est tensor pd}]
 We can assume that $a(x,\xi)=a_1(x',\xi')\chi(x_d)\ell(\xi_d)$, where $a_1$, $\chi$ and $\ell$ are compactly supported,
as all symbols can be approximated by a sum of such 
 symbols. From definition of operators associated with a symbol, we have 
 \[
 \ops(a)(w\otimes \delta_{x_d=0})=\chi ( \ops(a_1)w)\otimes \ops(\ell)\delta_{x_d=0}.
 \]
 A direct computation in Fourier variables gives 
 $\| \ops(\ell) \delta_{x_d=0}\|_{L^2(\R)}\lesssim h^{-1/2}$, and  as $a_1\in S^{-N}$ for all $N\in \N$,
$|  \ops(a_1)w|\lesssim  | w|_{L^2(\R^{d-1})} $, which implies the lemma.
 \end{proof}
\begin{proof}[Proof of Lemma~\ref{lem: first est trace}]
We work in a neighborhood of a boundary point in geodesic coordinates, we can write $x_0=(x_0',0)$.
We use the notation $\tilde g$ defined in Section~\ref{sec: Riem tools}. The system written in interior is given by
\begin{equation}
 	\label{eq: syst semi-class inter}
 \begin{cases}
  &u_h^k+\ops(\tilde g^{jk}\xi_j)r_h=h\tilde f_h^k , \text{ for } k=1,..., d-1, \\ 
  &u_h^d+ hD_dr_h=h\tilde f_h^d  \\  
&hD_d u_h^d+ \ds \sum_{j=1}^{d-1} \ops(\xi_j )u_h^j+r_h=h\tilde q_h,    
 \end{cases}
 \end{equation}
 where $\tilde f_h$ and $\tilde q_h$ are bounded in $L^2$.
Indeed System~\eqref{eq: syst semi-class inter} comes from System~\eqref{fluidec-3} and pseudo-differential calculus.
From \eqref{for: grad} and \eqref{for: div},  we have 
$-ih\nabla_{\tilde g}r_h= \ops(\tilde g^{jk})r_h+O(h)$, $hbu_h=O(h)$, and
$-hi\div_g u_h= hD_d u_h^d+ \ds \sum_{j=1}^{d-1}\ops(\xi_j)u_h^j +O(h)$ where we have  denoted by $O(h)$  
terms bounded by $Ch$ in $L^2$-norm.

Using the first equations, we can replace $u_h^j$ in the last equation of \eqref{eq: syst semi-class inter}. We then
obtain the 2 by 2 system
 \begin{equation}
 	\label{eq: syst semi-class 2 by 2}
 \begin{cases}
  &u_h^d+ hD_dr_h=h\tilde f_h^d \\
& hD_du_h^d - \ds \sum_{j=1}^{d-1} \ops(\xi_j )\ops(\tilde g^{jk}\xi_k)r_h+r_h
=h q_h + h \ds \sum_{j=1}^{d-1} \ops(\xi_j )\tilde f_h^j.
 \end{cases}
 \end{equation}
From symbolic calculus, $\sum_{j=1}^{d-1} \ops(\xi_j )\ops(\tilde g^{jk}\xi_k)=\ops(R(x,\xi'))+h\ops(r_1)$  where $r_1$
is of order 1, and $R(x,\xi')=\tilde g^{jk}\xi_k\xi_j $.
System \eqref{eq: syst semi-class 2 by 2} can be recast as 
 \begin{equation}
 	\label{eq: syst semi-class 2 by 2 bis}
 \begin{cases}
  &u_h^d+ \ops(\xi_d)r_h=h\tilde f_h^d \\
& \ops (\xi_d) u_h^d -\ops(R(x,\xi')-1)r_h
=hG,
 \end{cases}
 \end{equation}
 where $G= q_h +  \ds \sum_{j=1}^{d-1} \ops(\xi_j )\tilde f_h^j+\ops(r_1)r_h$.
\def\rh{w_h}
\def\uh{v_h}
\def\tch{\tilde \chi}

We define three types of points, hyperbolic points, the set where $R(x,\xi')-1<0$, glancing points, the set where 
$R(x,\xi')-1=0$ and elliptic points where $R(x,\xi')-1>0$. Observe that in both cases hyperbolic and glancing, $\xi'$
is bounded. 
In what follows we made some computations useful for hyperbolic or glancing cases.

Let $\chi(x,\xi')$ be compactly supported on a neighborhood of $(x_0, \xi'_0)$. In particular 
$\chi \in S^{-N}$ for all $N\in \N$.
We denote by $\rh=\ops(\chi) r_h$ and $ \uh=\ops(\chi) u_h^d$, in particular $(\uh)_{|x_d=0}=0$ from boundary conditions. From symbolic calculus we have
$hD_d\ops(\chi) z=\ops(\chi)hD_d z+ A$ where $\| A\|\lesssim h\| z\|$. We also have
\[
\ops(\chi) \ops(R(x,\xi')-1) = \ops(R(x,\xi')-1) \ops(\chi)+ [\ops(\chi), \ops(R(x,\xi')-1)],
\]
 and the symbol of $ [\ops(\chi), \ops(R(x,\xi')-1)]$ is in $hS^{-N}$ for all $N\in \N$.
 We then have 
 \[\|  [\ops(\chi), \ops(R(x,\xi')-1)] z  \|\lesssim h\| z\|.\]
Let $\tch$ compactly supported and $\tch=1$ on the support of $\chi$.
The symbol of operator defined by $ \ops(R(x,\xi')-1) \ops(\chi)$ is $(R(x,\xi')-1) \chi +ha_0$ with 
$a_0 \in hS^{-N}$ for all $N\in \N$.
And the symbol of   $  \ops(\tch^2(R(x,\xi')-1)) \ops(\chi)$ is $ (R(x,\xi')-1) \chi+ha_1 $ with $a_1 \in hS^{-N}$ 
for all $N\in \N$. Then we have 
\[
\|  ( \ops(R(x,\xi')-1) \ops(\chi) - \ops(\tch^2(R(x,\xi')-1)) \ops(\chi))z  \|\lesssim h\| z\|.
\]
 Applying $\ops(\chi)$ to System~\eqref{eq: syst semi-class 2 by 2 bis}, using the previous analysis, we obtain the following system:
 \def\fh{h\tilde f_h}
 \def\Gh{h\tilde  G_h}
 \begin{equation}
 	\label{eq: syst semi-class 2 by 2 hyp-gl}
 \begin{cases}
  &hD_d \rh +\uh =\fh\\
& hD_d\uh -\ops(\tch^2(R(x,\xi')-1))\rh
=\Gh  \\
&(\uh)_{|x_d=0}=0,
 \end{cases}
 \end{equation}
 where $\| \tilde G_h\|\lesssim \|q_h  \|+ \ds \sum_{j=1}^{d-1}\| \tilde f_h^j\|+\|r_h\|+\| u_h^d  \|$.
%
%
%
 \subsubsection{Hyperbolic region}
 
 We treat now the hyperbolic points
 \[{\mathcal H }=\{(x,\xi')\in\R^d\times \R^{d-1}, R(x,\xi')-1<0  \}. 
 \]
 Let $\chi $ be 
 compactly supported in ${\mathcal H }$ and $\tch$ as defined above.
 Let $\lambda(x,\xi')=\tch(x,\xi')\sqrt{1-R(x,\xi')}$. From symbolic calculus, we have 
 $\ops(\lambda^2)\op(\chi)=-\ops(\tch^2(R-1))+h\ops(a_0)$ where $a_0$ is of order 0.
System~\eqref{eq: syst semi-class 2 by 2 hyp-gl} writes 
  \begin{equation}
 	\label{eq: syst semi-class 2 by 2 hyp}
 \begin{cases}
  &hD_d \rh +\uh =\fh\\
& hD_d\uh +\ops(\lambda^2)\rh
=\Gh   \\
&(\uh)_{|x_d=0}=0,
 \end{cases}
 \end{equation}
 where $ ( \tilde f_h ,\tilde G_h)$ has been modified but are still bounded in $L^2$ uniformly with respect $h$.
 In what follows we keep alway $ ( \tilde f_h ,\tilde G_h)$ for $L^2$ bounded terms.
 
 Let $z_k=\uh+(-1)^k\ops(\lambda)\rh$ for $k=1,2$. We have from~\eqref{eq: syst semi-class 2 by 2 hyp}
 \begin{align*}
 hD_d z_k&= -\ops(\lambda^2)\rh   -(-1)^k  \ops(\lambda)      \uh +\Gh   
   +(-1)^k  \ops(\lambda)  \fh +h\ops(r_0)\rh\\
   &=-(-1)^k\ops(\lambda)  z_k +\Gh   
   +(-1)^k  \ops(\lambda)  \fh +h\ops(r_0')\rh,
 \end{align*}
 where $r_0$ and $ r_0'$ are of order 0 and come from $[hD_d, \ops(\lambda) ]$ and formula 
 $\ops(\lambda) \ops(\lambda) = \ops(\lambda^2)+h\ops(r_0'')$ where $r_0''$ is of order 0.
 Then we obtain
 $ hD_d z_k+(-1)^k\ops(\lambda)  z_k =hH_k$ where $H_k$ bounded in $L^2$  uniformly with respect $h$.
 To obtain an estimation for $z_k$ we apply the following lemma.
 \begin{lem}
 	\label{lem: first order hyperbolic}
Let $a\in S^1$. We assume $a$ supported on a \nhd of $x_d=0$. There exists $C>0$, such that for all  $w$ satisfying 
$hD_dw+\ops (a)w=hf$, we have $|w_{|x_d=0}| \le C(\| f\|+\|w\|)$.
 \end{lem}
 \begin{proof}
In one hand, we have 
\begin{equation}
	\label{est: champ reel}
	|\Im ( hf , w )|\le h\| f\|\| w\|. 
\end{equation}
 In other hand, we have
\begin{align*}
2i\Im(hD_d+\ops (a)w, w)=(hD_dw, w)-(w, hD_dw)+(\ops(a)w,w)-(w,\ops(a) w).
\end{align*}
We now apply integration by parts formula~\eqref{eq: integration parts functions} and as $a$ is real valued, 
$\ops(a)^*=\ops(a)+h\ops(r_0)$ where the symbol $r_0$ is of order 0, we obtain
\begin{align*}
2i\Im(hD_d+\ops (a)w, w)=ih|w_{|x_d=0}|^2+h(\ops(r_0) w,w),
\end{align*}
and from~\eqref{est: champ reel} above, we deduce Lemma~\ref{lem: first order hyperbolic}.
 \end{proof}
From Lemma~\ref{lem: first order hyperbolic}, applied to $z_k$ we obtain,
\begin{equation}
|(z_k)_{|x_d=0}|\lesssim \|z_k\|+ \|H_k\|
\end{equation}
From boundary condition, $(\uh)_{|x_d=0}=0$, and as $ \|z_k\|\lesssim \|\uh \| +\| \rh \| $ then we obtain
\begin{equation}
|(\ops(\lambda)\rh)_{|x_d=0}|\lesssim  \|\uh \| +\| \rh \|+ \|H_k\|.
\end{equation}
Let $\chi_2$ be compactly supported in ${\mathcal H }$ and such that $\chi_2\tch=\tch$.
We have by symbolic caculus
\[
\ops(\chi_2(1-R(x,\xi'))^{-1/2})\ops(\lambda)=\ops (\tch)+h\ops(a_0),
\]
 where $a_0\in S^0$.
 We can deduce 
 \[
 |(\ops(\tch)\rh)_{|x_d=0}|\lesssim  \|\uh \| +\| \rh \|+ \|H_k\|+h|( \rh)_{|x_d=0}|.
 \]
As   $\tch=1$ on the support of $\chi$, we have 
\(
\ops(\tch)\rh=\rh+h\ops(a_0)r_h,
\)
where $a_0\in S^0$.
We thus  obtain  
\begin{equation}
\label{est: hyperbolic final}
|(\rh)_{|x_d=0}|\lesssim   \|\uh \| +\| \rh \|+ \|H_k\|+h|( r_h)_{|x_d=0}|.
\end{equation}
The estimate above is better in hyperbolic region than the statement of Lemma~\ref{lem: first est trace}.
Observe that in hyperbolic region we do not need boundary condition as by linear combination between $z_1$
and $z_2$ we can estimate both $( \rh)_{|x_d=0}$ and $( \uh)_{|x_d=0}$.
%
%
%

\subsubsection{Glancing region}

We treat now the glancing region. Let $\eps>0$ we define 
\[  {\mathcal G }_\eps=\{(x,\xi')\in\R^d\times \R^{d-1}, |R(x,\xi')-1|<\eps  \}.\]
Let $\chi$ compactly supported in $ {\mathcal G }_\eps$. We keep same notation as in hyperbolic region, i.e.
$\rh=\ops(\chi) r_h$ and $ \uh=\ops(\chi) u_h^d$ satisfying System~\eqref{eq: syst semi-class 2 by 2 hyp-gl}.
We have for $0\le x_d\le \eps^{-1/2}h$,
\begin{align*}
|\uh(x',x_d)|^2\le \left(  \int_0^{\eps^{-1/2}h} |  \pd_{x_d}\uh(x',\sigma)|d\sigma\right)^2.
\end{align*}
Integrating with respect $x_d$, we obtain
\begin{align*}
 \int_0^{\eps^{-1/2}h}|\uh(x',x_d)|^2dx_d\le  \eps^{-1/2}h\left(  \int_0^{\eps^{-1/2}h} |  \pd_{x_d}\uh(x',\sigma)|d\sigma\right)^2
\le \eps^{-1}h^2 \int_0^{\eps^{-1/2}h} |  \pd_{x_d}\uh(x',\sigma)|^2d\sigma.
\end{align*}
From System~\eqref{eq: syst semi-class 2 by 2 hyp-gl} and with an integration with respect $x'$ 
we obtain
\begin{align}
	\label{eq: est speed glancing}
\int_{\R^{d-1}} \int_0^{\eps^{-1/2}h}|\uh(x',x_d)|^2dx_d dx' \le   \eps^{-1}\| \ops(R-1) \rh \|^2 +h^2\|  \Gh \|^2.
\end{align}
From $\rh(x',0)=\rh(x',x_d)-\int_0^{x_d}\pd_{x_d}\rh(x',\sigma)d\sigma$, we have for $x_d\in(0,\eps^{-1/2}h)$
\begin{align*}
|\rh(x',0)|^2\lesssim |\rh(x',x_d)|^2+ \left(\int_0^{\eps^{-1/2}h}|\pd_{x_d}\rh(x',\sigma)|d\sigma\right)^2.
\end{align*}
With an integration with respect $x_d$ and Cauchy-Schwarz inequality, we obtain
\begin{align*}
\eps^{-1/2}h|\rh(x',0) |^2\lesssim \int_0^{\eps^{-1/2}h}|\rh(x',x_d)|^2dx_d
+\eps^{-1} \int_0^{\eps^{-1/2}h}|h\pd_{x_d}\rh(x',x_d)|^2dx_d.
\end{align*}
With an integration with respect $x'$ and from System~\eqref{eq: syst semi-class 2 by 2 hyp-gl}, this yields
\begin{align*}
\eps^{-1/2}h|(\rh)_{|x_d=0}|^2\lesssim \|\rh \|^2+\eps^{-1}   \int_{\R^{d-1}} \int_0^{\eps^{-1/2}h} |\uh (x',x_d)|^2dx_d dx' 
+\eps^{-1} h^2\|\tilde  f_h^d\|^2.
\end{align*}
Applying~\eqref{eq: est speed glancing}, we have
\begin{align}
	\label{est: glancing end}
\eps^{-1/2}h|(\rh)_{|x_d=0}|^2\lesssim \|\rh \|^2+\eps^{-2} \| \ops(R-1) \rh \|^2+\eps^{-1} h^2\| \tilde G_h^d\|^2
            +\eps^{-1} h^2\| \tilde f_h ^d\|^2.
\end{align}
By symbolic calculus, $ \ops(R-1) \rh=\ops((R-1)\chi)r_h+ h\ops(a_0) r_h$ where $a_0\in S^0$. As by assumption
$|(R-1)\chi|\le \eps$, for all $\eps>0$, from~\eqref{est: norm L2} there exists $C_\eps>0$ such that
 \[
 \| \ops(R-1) \rh \|^2\le 4 \eps^2\|r_h \|^2 +C_\eps h^2 \|r_h \|^2.
 \] 
Then from~\eqref{est: glancing end} we obtain
\begin{equation}
	\label{est: glancing final}
\limsup_{h\to 0}h|(\rh)_{|x_d=0}|^2\lesssim \eps^{1/2}.
\end{equation}

%
%
%
\subsubsection{Elliptic region}

The last zone to consider is the elliptic region. Let ${\mathcal E }=\{(x,\xi')\in\R^d\times \R^{d-1}, R(x,\xi')-1>0  \}$.
As ${\mathcal E }$ is not compact in $\R^d\times \R^{d-1}$ we have to treat the large $|\xi'|$.
Let $\chi$ be a smooth function from $\R^d\times \R^{d-1}$ to $\R$, satisfying, $\chi(x,\xi')=1 $ for $(x,\xi')$ such that
$R(x,\xi')\ge 1+\eps$ and $|x|\le \delta$,  $\chi(x,\xi')=0$ for $|x|\ge 2\delta$. We assume a 0-homogeneity condition
that is  $\chi(x,\lambda\xi')= \chi(x,\xi')$ for $(x,\xi')$ such that $R(x,\xi')\ge 1+\eps$ and $\lambda\ge 1$. The 
arguments below are close to the one used in hyperbolic region but we take care to large $|\xi'|$.

We denote by $\rh=\ops(\chi) r_h$ and $ \uh=\ops(\chi) u_h^d$. From symbolic calculus we have
$hD_d\ops(\chi) z=\ops(\chi)hD_d z+ A$ where $\| A\|\lesssim h\| z\|$. Indeed, the symbol of $[hD_d,\ops(\chi)]$ is
$hD_{d} \chi\in hS^0$.  

We also have
\[\ops(\chi) \ops(R(x,\xi')-1) = \ops(R(x,\xi')-1) \ops(\chi)+ [\ops(\chi), \ops(R(x,\xi')-1)],\]
 and the symbol of $ [\ops(\chi), \ops(R(x,\xi')-1)]$ is in $hS^1$.
 
We then have 
 $\|  [\ops(\chi), \ops(R(x,\xi')-1)] z  \|\lesssim h\| \Lambda^1 z\|$. 
Let $\tch$  be a smooth function, supported on   ${\mathcal E } $ and $\tch=1$ on the support of $\chi$.
We assume  $\tch(x,\xi')=1 $ for $(x,\xi')$ such that $R(x,\xi')\ge 1+\eps/2$ and $|x|\le 2\delta$,  $\tch(x,\xi')=0$ 
for $|x|\ge 3\delta$ and   a 0-homogeneity condition
that is  $\tch(x,\lambda\xi')= \chi(x,\xi')$ for $(x,\xi')$ such that $R(x,\xi')\ge 1+\eps/2$ and $\lambda\ge 1$. We 
then have $\chi\tch=\chi$.

The symbol of $ \ops(R(x,\xi')-1) \ops(\chi)$ is $(R(x,\xi')-1) \chi+ha_1$  
where  $a_1$ is in $S^1$. 
The symbol of $\ops(\tch(R(x,\xi')-1) )\ops(\chi)$ is $(R(x,\xi')-1) \chi+h\tilde a_1$ where  $\tilde a_1$ is in $S^1$. 
We can then replace $\ops(R(x,\xi')-1) \ops(\chi)$ by $\ops(\tch^2(R(x,\xi')-1) )\ops(\chi)$ 
  modulo a term $ h a_1$ where  $ a_1$ is in $S^1$.
  
 Applying $\ops(\chi)$ to System~\eqref{eq: syst semi-class 2 by 2 bis}, using the previous analysis, 
 we obtain the following:
 \begin{equation}
 	\label{eq: syst semi-class 2 by 2 elliptic}
 \begin{cases}
  &hD_d \rh +\uh =\fh\\
& hD_d\uh -\ops(\tch^2(R(x,\xi')-1))\rh
=\Gh \\
&(\uh)_{|x_d=0}=0,  
 \end{cases}
 \end{equation}
 where 
$ \|\Lambda^{-1} \tilde  G_h \|\lesssim 
\| \Lambda^{-1} G\|
+\| r_h \|  +\|  u^d_h\|$ with 
 notation used in \eqref{eq: syst semi-class 2 by 2 bis}.

Let $\lambda(x,\xi')=\tch(x,\xi')\sqrt{R(x,\xi')-1}$. The symbol of $\ops((1-\tch^2)(R-1))\ops(\chi)$ is in $h^NS^{-N}_{1,0}$
for all $N>0$ as $(1-\tch^2)\chi=0$ the asymptotic expansion of the symbol is null.
 We may write system~\eqref{eq: syst semi-class 2 by 2 elliptic} 
  \begin{equation}
 	\label{eq: syst semi-class 2 by 2 elliptic-2}
 \begin{cases}
  &hD_d \rh +\uh =\fh\\
& hD_d\uh -\ops(\lambda^2)\rh
=\Gh\\
&(\uh)_{|x_d=0}=0,  
 \end{cases}
 \end{equation}
 where  $\tilde G_h$ satisfy $ \|\Lambda^{-1} \tilde  G_h \|\lesssim 
\| \Lambda^{-1} G\|
+\| r_h \|  +\|  u^d_h\|$ estimations as above.

We set  $z_k=\uh+(-1)^ki\ops(\lambda)\rh$ for $k=1,2$. We have
 \begin{align*}
 hD_d z_k&= hD_d    \uh+(-1)^ki hD_d  \ops(\lambda)\rh\\
 &=\ops(\lambda^2)\rh+\Gh+(-1)^ki   \ops(\lambda)   hD_d  \rh  +h\ops(a_1)  \rh \\
   &= (-1)^{k+1}i  (  \ops(\lambda)  \uh + (-1)^k i\ops(\lambda^2)\rh)+\Gh+(-1)^ki   \ops(\lambda)  \fh +h\ops(a_1)  \rh \\
  &= (-1)^{k+1}i   \ops(\lambda) (    \uh + (-1)^k i\ops(\lambda)\rh)+\Gh+(-1)^ki   \ops(\lambda)  \fh +h\ops(a_1)  \rh ,
 \end{align*}
 as the symbol of $[D_d,\ops(\lambda)]$ is in $S^1$ and $\ops(\lambda)^2=\ops(\lambda^2)$ modulo an operator 
 with symbol in $hS^1$. In previous formula $a_1$ is a symbol in $S^1$, changing line from line eventually.
 \def\truc{H_h}
 Then $z_k$ satisfies the following equation
 \begin{equation}
 	\label{eq: elliptic eq zk}
  hD_d z_k+  (-1)^{k}i   \ops(\lambda)  z_k =\Gh+(-1)^ki   \ops(\lambda)  \fh +h\ops(a_1)  \rh= h\truc,
 \end{equation}
 where we have $\| \Lambda^{-1} \truc \|  \lesssim  \| \tilde f_h \|+\| \Lambda^{-1}  G\|+\| r_h \|  +\| u^d_h\|$.

 We set
 \begin{align*}
  I=\Re ( \Lambda^{-1}( hD_d z_k+  (-1)^{k}i   \ops(\lambda)  z_k ), \Lambda^{-1} (-1)^ki\ops(\lambda) z_k),
 \end{align*}
and in one hand we have
\begin{align}
\label{eq:elliptic 0}
  |I|\lesssim h \|  \Lambda^{-1} \truc \|  \|z_k \| ,
\end{align}
and in other hand 
\begin{align}
\label{eq: elliptic 1}
2I&=2\|  \Lambda^{-1} \ops(\lambda)  z_k \|^2\\
&\quad+\Re (hD_d  \Lambda^{-1} z_k, \Lambda^{-1} (-1)^ki\ops(\lambda) z_k)
+( \Lambda^{-1} (-1)^ki\ops(\lambda) z_k  , hD_d  \Lambda^{-1} z_k )\notag \\
&=I_1+\Re I_2+\Re I_3.\notag
\end{align}
To compute $I_2+I_3$ we integrate by parts applying~\eqref{eq: integration parts functions} 
and by symbolic calculus we have 
$\ops(\lambda)^*=\ops(\lambda)+h\ops(a_0)$,  the following commutator results, 
$[\Lambda^{-1},\ops(\lambda)]=h\ops(a_{-1})$, $[hD_d,\Lambda^{-1}]=0$ and 
$[hD_d,\ops(\lambda)]=h\ops(a_{1})$where $a_j\in S^j_{1,0}$. Then we have
\begin{align*}
&( \Lambda^{-1} (-1)^ki\ops(\lambda) z_k  , hD_d  \Lambda^{-1} z_k )\\
&= ( hD_d  \Lambda^{-1} (-1)^ki\ops(\lambda) z_k  , \Lambda^{-1} z_k )
-ih( \Lambda^{-1} (-1)^ki\ops(\lambda) (z_k)_{|x_d=0}   ,  \Lambda^{-1}( z_k )_{|x_d=0} )_\pd,
\end{align*}
 and
 \begin{align*}
 ( hD_d  \Lambda^{-1} (-1)^ki\ops(\lambda) z_k  , \Lambda^{-1} z_k )
 &= ( \Lambda^{-1} (-1)^ki\ops(\lambda) hD_d  z_k  , \Lambda^{-1} z_k )+h( \ops(a_0) z_k,  \Lambda^{-1} z_k)\\
 &=  ( \Lambda^{-1} (-1)^ki hD_d  z_k  , \ops(\lambda) ^*\Lambda^{-1} z_k )\\
&\quad+   ((-1)^ki[ \Lambda^{-1}, \ops(\lambda)] hD_d  z_k  , \Lambda^{-1} z_k )\\
&\quad +h( \ops(a_0) z_k,  \Lambda^{-1} z_k)\\
 &= ( \Lambda^{-1} (-1)^ki hD_d  z_k  , \ops(\lambda) \Lambda^{-1} z_k )\\
&\quad+  h (\Lambda^{-1} hD_d  z_k  , \ops(a_{-1}) z_k )
 +h( \ops(a_0) z_k,  \Lambda^{-1} z_k)\\
 &= -I_2+ h (\Lambda^{-1} hD_d  z_k  , \ops(a_{-1}) z_k )
 +h( \ops(a_0) z_k,  \Lambda^{-1} z_k).
  \end{align*}
 From the two previous computations we obtain
 \begin{align*}
 I_2+I_3&=-ih( \Lambda^{-1} (-1)^ki\ops(\lambda) (z_k)_{|x_d=0}   ,  \Lambda^{-1}( z_k )_{|x_d=0} )_\pd\\
&\quad +  h (\Lambda^{-1} hD_d  z_k  , \ops(a_{-1}) z_k )
 +h( \ops(a_0) z_k,  \Lambda^{-1} z_k).
 \end{align*}
 Then from previous computations and \eqref{eq: elliptic 1} we obtain
 \begin{align}
 \label{eq: elliptic 2}
 &2I=2\|  \Lambda^{-1} \ops(\lambda)  z_k \|^2
 +h\Re ( \Lambda^{-1} (-1)^k\ops(\lambda) (z_k)_{|x_d=0}   ,  \Lambda^{-1}( z_k )_{|x_d=0} )_\pd+ R\\
&\text{where }
 |R|\lesssim  h \|\Lambda^{-1} hD_d  z_k \| \|  \Lambda^{-1}z_k \|
 +h  \|  \Lambda^{-1}z_k \| \| z_k \| .\notag
 \end{align}
 From Equation~\eqref{eq: elliptic eq zk}  we have
 \[
  \|\Lambda^{-1} hD_d  z_k \| \lesssim  \|\Lambda^{-1} \truc \|+ \| z_k \|.
 \]
We obtain from the previous estimates, \eqref{eq: elliptic 1}, \eqref{eq: elliptic 2}, and 
using that $(z_1)_{|x_d=0}=-(z_2)_{|x_d=0}$,
\begin{align}
&\sum_{k=1,2} \| \Lambda^{-1} \ops(\lambda)  z_k \|^2
 + h\Re ( \Lambda^{-1}\ops(\lambda) (z_2)_{|x_d=0}   ,  \Lambda^{-1}( z_2)_{|x_d=0} )_\pd\\
 &\qquad \lesssim h
\sum_{k=1,2}  (\|  \Lambda^{-1}z_k\| \| z_k \| + \|  \Lambda^{-1}z_k\| \|\Lambda^{-1} \truc \|).  \notag
\end{align}
 As $\uh=(z_1+z_2)/2$ and $\ops(\lambda)\rh=(z_2-z_1)/2i$, we obtain 
 \begin{align}
 \label{est: elliptic 3}
 & \| \Lambda^{-1} \ops(\lambda)  \uh \|^2+ \| \Lambda^{-1} \ops(\lambda)^2  \rh \|^2
   +h\Re ( \Lambda^{-1}\ops(\lambda) (z_2)_{|x_d=0}   ,  \Lambda^{-1}( z_2)_{|x_d=0} )_\pd\\
   &\qquad \lesssim h
\sum_{k=1,2}  (\|  \Lambda^{-1}z_k\| \| z_k \| + \|  \Lambda^{-1}z_k\| \|\Lambda^{-1} \truc \|).  \notag\\
&\qquad \lesssim h  (\| z_k \|^2+\|\Lambda^{-1} \truc \|^2).\notag
 \end{align}
By G\aa rding inequality~\eqref{Garding ineq} there exist $C_0>0$ and $C_1>0$ such that 
\[
  \| \Lambda^{-1} \ops(\lambda)  \uh \|^2\ge C_0 \| \uh\|^2-C_1h\|u_h^d \|^2,
\]
as $\uh=\ops(\chi)u_h^d$, the  symbol of $ \ops(\lambda)^* \Lambda^{-2} \ops(\lambda) $ is 
$(1+|\xi'|)^{-2}(R(x,\xi')-1)+h a_0$ and $(1+|\xi'|)^{-2}(R(x,\xi')-1)\ge C>0$ on the support of $\chi$.

By analogous arguments, we have 
\[
  \| \Lambda^{-1} \ops(\lambda)^2  \rh \|^2\ge C_0\|\Lambda^1 \rh \|^2- C_1h\| \Lambda ^{1/2}r_h\|^2.
\]
 As $r_h=\rh+\ops(1-\chi)r_h$ and $\Lambda ^{s}\ops(1-\chi) $   is bounded on $L^2$ for all $s\in\R$ 
 as $\chi$ is compactly supported, 
 we have
 \[
  \| \Lambda^{-1} \ops(\lambda)^2  \rh \|^2\ge C'_0\|\Lambda^1 \rh \|^2- C_1h\| r_h\|^2,
\]
and
\[
\|z_k\|\lesssim \| u_h^d \|+\| \Lambda  r_h \|\lesssim \| u_h^d \|+\| \Lambda^1  \rh \|+\| r_h \|.
\]
For trace term in \eqref{est: elliptic 3}, as $(z_2)_{|x_d=0}=\ops(\lambda)\ops(\chi)r_h$, we can apply
G\aa rding inequality to obtain
\[
\Re ( \Lambda^{-1}\ops(\lambda) (z_2)_{|x_d=0}   ,  \Lambda^{-1}( z_2)_{|x_d=0} )_\pd
\ge C_0|\Lambda^{1/2} (\rh)_{|x_d=0}|^2-C_1h|(r_h)_{|x_d=0}|^2.
\]
Applying the previous estimates and \eqref{est: elliptic 3}, we obtain
\begin{align}
	\label{est: elliptic final}
\|\Lambda^1 \rh \|^2+ \| \uh\|^2 +h |\Lambda^{1/2} (\rh)_{|x_d=0}|^2 
\lesssim  h  (\| u_h^d \|^2+\|  r_h \|^2+\|\Lambda^{-1} \truc \|^2+h |(r_h)_{|x_d=0}|^2).
\end{align}
%
%
%
\subsubsection{End of proof}
 Taking three cutoffs on each region, hyperbolic, glancing and elliptic, respectively $\chi_H$, $\chi_G$ and $\chi_E$, 
 we obtain 
\begin{align*}
&h |\Lambda^{1/2} (r_h)_{|x_d=0}|^2\\
&\qquad \lesssim h |\Lambda^{1/2} (\ops(\chi_H)r_h)_{|x_d=0}|^2
+h |\Lambda^{1/2} (\ops(\chi_G)r_h)_{|x_d=0}|^2 +h |\Lambda^{1/2} (\ops(\chi_E)r_h)_{|x_d=0}|^2.
\end{align*}
From \eqref{est: hyperbolic final}, we have, as $\chi_H$ is compactly supported,
\[
|\Lambda^{1/2} (\ops(\chi_H)r_h)_{|x_d=0}|^2\lesssim C +h|(r_h)_{|x_d=0}|^2\lesssim  C 
+h|\Lambda^{1/2}(r_h)_{|x_d=0}|^2.
\]
From \eqref{est: elliptic final}, we have
\[
|\Lambda^{1/2}  (\ops(\chi_E)r_h)_{|x_d=0}|^2 \lesssim  C +h|\Lambda^{1/2}(r_h)_{|x_d=0}|^2,
\]
as $ \|\Lambda^{-1} \truc \|$ is bounded from \eqref{eq: elliptic eq zk}.

We  obtain
\[
h|\Lambda^{1/2} (r_h)_{|x_d=0}|^2\lesssim Ch
+h|\Lambda^{1/2}  (\ops(\chi_G)r_h)_{|x_d=0}|^2+h^2|\Lambda^{1/2}(r_h)_{|x_d=0}|^2,
\]
then for all $h$ sufficiently small we have
\[
h|\Lambda^{1/2} (r_h)_{|x_d=0}|^2\lesssim Ch
+h|\Lambda^{1/2}  (\ops(\chi_G)r_h)_{|x_d=0}|^2.
\]
From  \eqref{est: glancing final} we have
\[
\limsup_{h\to 0}h|\Lambda^{1/2} (r_h)_{|x_d=0}|^2\lesssim \eps^{1/2},
\]
for all $\eps>0$, this proves Lemma~\ref{lem: first est trace}.
\end{proof}
 
 In the following we prove that the measure defined is not null and it is null on the set $\{x,\ b(x)>0\}$.

%
%

\subsection{The measure  is not identically null}

%
%
\begin{prop}
The  measure $m$ is not identically null.
\end{prop}
\begin{proof}
From \eqref{fluidec-3}, multiplying the first equation by $u_h$, integrating by parts (see \eqref{eq: integration parts boundary}) and taking account the boundary condition $(u^d_h)_{|x_d=0}=0$, we have
\[
O(h)=(-ih\nabla_g r_h+ u_h, u_h)_g=\| u_h\|^2_g+(r_h, -ih\div_gu_h)_g.
\]
From second equation of  \eqref{fluidec-3} we replace $ -ih\div_gu_h$ by $r_h$ and we obtain
\[
O(h)=\| u_h\|^2_g -\|r_h \|^2_g.
\]
As from \eqref{est: unif 2}, we have $\| u_h\|+\|r_h\|=1$, we have $\|r_h\|=1/2 +O(h)$ and  $\|u_h\|=1/2 +O(h)$.
From the first equation of  \eqref{fluidec-3}, we obtain $\| h\nabla_g r_h\|\lesssim 1$.

We prove now that $m$ is not null. Let $\phi$ be a smooth function such that $\phi(s) =1 $ for  $ |s|\le 1$ and
supported in $|s|\le 2$. For $R>0$, we set  $ \phi_R (\xi)=\phi(|\xi|/R)$.
For $\alpha\in (0,1/2)$ we have
\begin{equation}
	\label{est: H s sc}
\| \ops(1- \phi_R)\underline{r}_h\|\lesssim R^{-\alpha}\|\ops(\langle \xi\rangle^\alpha )\underline{r}_h\|
\lesssim R^{-\alpha}(\|r_h\|_{L^2(\Omega)}+\|h\nabla_g r_h\|_{L^2(\Omega)})\lesssim  R^{-\alpha}.
\end{equation}
(See for instance \cite[Lemma 4.3]{Cor-Rob}).

Then for $h$ sufficiently small and $R $ sufficiently large,
\[
(\ops( \phi_R)\underline{r}_h,  \underline{r}_h)\ge 1/4,
\]
and from the definition of $m$ (see \eqref{def: measure}), $m$ is not null. this gives the result.
\end{proof}

%
%

\subsection{The semiclassical measure null on support of damping}

In this section we prove the following proposition.
%
%
\begin{prop}
	\label{prop: bm null}
Let $m$ be the semiclassical measure defined  in Proposition~\ref{Prop: measure support and form},
we have $b(x)m=0$.
\end{prop}
\begin{proof}
From System~\eqref{fluidec-3} we have
\begin{align}
(ihf_h,u_h)_g+(ihq_h,r_h)_g
=(-ih\nabla_g r_h+ u_h -ihbu_h,u_h)_g+(-ih \div_g u_h+ r_h,r_h)_g.
\end{align}
Integrating by parts, using~\eqref{eq: integration parts boundary} and taking account 
the boundary condition $(u^d_h)_{|x_d=0}=0$,  we have
\begin{align} 
(-ih\nabla_g r_h  ,u_h)_g+(-ih \div_g u_h  ,  r_h)_g=2\Re (-ih\nabla_g r_h  ,u_h)_g.
\end{align}
From the two previous formula, we obtain
\begin{align}
\Im \big(
(ihf_h,u_h)_g+(ihq_h,r_h)_g
\big)=ih(bu_h,u_h)_g.
\end{align}
As the left hand side is a $o(h)$, we obtain $(bu_h,u_h)_g\to 0$ as $h\to 0$. 
Then from   Proposition~\ref{Prop: measure support and form} and Proposition~\ref{prop: limit tangential} below, we have
$\langle m \sum_{k=1}^d g^{kk}\xi_k^2 , b\rangle=0$.

As $b$ and $m$ are non negative,  $ \ds \sum_{k=1}^d g^{kk}\xi_k^2  \gtrsim  \ds \sum_{j,k=1}^d g^{kj}\xi_k\xi_j=1  $ on the support of $m$, 
we obtain the result.
\end{proof}
  
 %
 %
\section{Propagation of the support}
	\label{Sec: Propagation of the support}
The goal is to prove the following proposition.
\begin{prop}
	\label{prop: propagation support semiclassical measure}
We assume that 	 \(\omega\) satisfy GCC. 
Let  \(m\) the semiclassical
measure constructed from \((u_h,r_h)_h\) and satisfying \eqref{fluidec-3}, 
we have \(m=0\).
\end{prop}

First in this section we give some formulas yielding to the propagation results.  
This is stated in the following proposition.
%
%

\begin{prop}
	\label{prop: propagation formula}
We denote by $p(x,\xi)=g^{jk}\xi_j\xi_k$  and  $R(x,\xi')=\tilde g^{jk}\xi_k\xi_j $.

We have $ \{ p,m\}=0 $ in $\Omega\times \R^d$.

In a \nhd of $\pd \Omega$ i.e. $x_d=0$ in geodesic coordinates,
let $a(x,\xi') $ be a tangential symbol compactly supported in a \nhd of the boundary, we have,
\begin{align}
&  \langle \{ p,m\}, a\rangle =0 ,\label{prop: propagation formula 2}\\
&  \langle \{ p,m\},  a\xi_d\rangle
=- \lim_{h\to 0}(\ops(a(R-1))(r_h)_{|x_d=0},(r_h)_{|x_d=0})_{\pd}. \label{prop: propagation formula 3}
\end{align}
\end{prop}
Observe that we do not know if $(r_h)_{|x_d=0}$ is bounded in $L^2$, Lemma~\ref{lem: first est trace}
is too weak. But the quantity $(\ops(a(R-1))(r_h)_{|x_d=0},(r_h)_{|x_d=0})_{\tilde g}$ has a limit.
A priori the measure $m$ is not defined on functions in $(x,\xi')$ but as $m$ is supported on 
$g^{jp}\xi_j\xi_p-1=0$, which is a bounded set on $\xi$, it is easy to extend $m$ to functions in  $(x,\xi')$.
At boundary we assume the metric written in geodesic coordinates (see Section \ref{sec: Riem tools}).

%
%
\subsubsection{Computations in interior} 
We introduce the following quantity, where $a$ is a  symbol compactly supported in $\Omega\times \R^d$.
In what follows we shall choose $b^j_k$ depending on $a$
but the choice of $b^j_k$ is more natural after the computations below. Note that $b^j_k$ are also 
compactly supported in $\Omega\times \R^d$. In interior $\underline{u}_h$ and $\underline{r}_h$ are $u_h$ and $r_h$. Rigorously we need to introduce a cutoff to do that but this  introduces error terms of order $O(h^\infty)$.
\begin{align}
	\label{eq: A interior tangential}
A&=(\ops(a)(u_h-ih\nabla_gr_h),u_h)_g+(\ops(a)(r_h-ih\div_g u_h),r_h)_g\\
&\quad -(\ops(a)u_h,u_h-ih\nabla_g r_h)_g-(\ops(a)r_h,r_h-ih\div_g u_h)_g\notag\\
&\quad +(ih\ops(b^j_k)(u_h^k-ih(\nabla_gr_h)^k),u_h^\ell)_g. \notag
\end{align}
Here and in what follows, we use Einstein notation, see Section~\ref{sec: Riem tools}.

Obviously, the assumptions on System~\eqref{eq: syst semi-class propagation} imply that $h^{-1}A\to 0$ as $h\to 0$.

We can simplify $A$ and we have,
\begin{align}
	\label{eq: A formula 0}
A&=-(ih\ops(a)\nabla_gr_h,u_h)_g-(ih\ops(a)\div_g u_h,r_h)_g\\
&\quad +(\ops(a)u_h,ih\nabla_g r_h)_g+(\ops(a)r_h,ih\div_g u_h)_g\notag\\
&\quad +(ih\ops(b^j_k)u_h^k,u_h^\ell)_g+(h\ops(b^j_k)(h\nabla_gr_h)^k,u_h^\ell)_g. \notag
\end{align}
We integrate by parts (see~\eqref{eq: integrate parts without bd}) and we write by $R$ terms estimated by $h(\| u_h\|^2+\|r_h\|^2)$ playing no role in the analysis.
Observe that no boundary term appears as $a$ is supported far away the boundary and 
 we have
 \begin{align}
 	\label{eq: A formula 1}
h^{-1}A&=(i([\nabla_g,\ops(a)]r_h,u_h)_g+(i[\div_g,\ops(a) ]u_h,r_h)_g\\
&\quad +(i\ops(b^j_k)u_h^k,u_h^\ell)_g+(h\ops(b^j_k)(\nabla_gr_h)^k),u_h^\ell)_g \notag\\
&=(i\ops\{ g ^{jk}\xi_k, a\}r_h,(u_h)^\ell)_g+(i\ops\{\xi_k,a\}u_h^k,r_h)_g\notag \\
&\quad +(i\ops(b^j_k)u_h^k,u_h^\ell)_g+(i\ops(b^j_kg^{k\mu}  \xi_\mu)r_h,u_h^\ell)_g+R. \notag
\end{align}
To use Formula~\eqref{def: measure} defining the semi-classical measure, we need a symmetry property.
Let $b^j=\{g^{jk}\xi_k, a\}+b^j_kg^{k\ell}\xi_\ell$. The symmetry required is $g_{pj}b^j=\{\xi_p,a\}$.
 Then we have
 \[
 g_{pj}b^j=g_{pj}\{g^{jk}\xi_k, a\}+g_{pj}b^j_kg^{k\ell}\xi_\ell,
 \]
 and 
 \[
 \{\xi_p,a\}= \{ g_{pj}g^{jk}\xi_k,a\}=g_{pj}\{ g^{jk}\xi_k,a\}+g^{jk}\xi_k\{ g_{pj},a\}.
 \]
 Then
 \[
 g_{pj}b^j= \{\xi_p,a\}- g^{jk}\xi_k\{ g_{pj},a\}+g_{pj}b^j_kg^{k\ell}\xi_\ell.
 \]
 Let $b_k^j=g^{jq}\{g_{qk},a\}$, we have
 \[
 g_{pj}b^j_kg^{k\ell}\xi_\ell= g_{pj}g^{jq}\{g_{qk},a\}g^{k\ell}\xi_\ell=\{g_{pk},a\}g^{k\ell}\xi_\ell= g^{jk}\xi_k\{ g_{pj},a\},
 \]
 changing the summation indexes. This choice of $b_k^j$ allows to apply~\eqref{def: measure}.
 
With the previous definition of $b^j_k$, Formula~\eqref{eq: A formula 1} writes
  \begin{align}
h^{-1}A
&=(i\ops\{ g ^{jk}\xi_k, a\}r_h,(u_h)^\ell)_g+ (i\ops\{\xi_k,a\}u_h^k,r_h)_g\notag \\
&\quad+(i\ops(g^{jq}\{g_{qk},a\})u_h^k,u_h^\ell)_g
+(i\ops( g^{jq}\{g_{qk},a\}   g^{k\mu}  \xi_\mu)r_h,u_h^\ell)_g+R. \notag
\end{align}
Then $ih^{-1}A$ converges to $-\langle \mu_j^k, g^{jq}\{g_{qk},a\}\rangle    
- 2  \langle \nu_j ,\{ g ^{jk}\xi_k, a\}+ g^{jq}\{g_{qk},a\}   g^{k\mu}  \xi_\mu \rangle =0$.
From Proposition~\ref{Prop: measure support and form},  $\mu_j^k=g^{kp}\xi_p\xi_j m$ and $\nu_j=-m\xi_j$, 
we thus obtain
\begin{equation}
	\label{eq: braket with m}
-\langle g^{kp}\xi_p\xi_j m, g^{jq}\{g_{qk},a\}\rangle    
+ 2  \langle m\xi_j,\{ g ^{jk}\xi_k, a\}+ g^{jq}\{g_{qk},a\}   g^{k\mu}  \xi_\mu \rangle =\langle m,F\rangle =0,
\end{equation}
 where
 \[
 F=- g^{kp}\xi_p\xi_jg^{jq}\{g_{qk},a\}+2\xi_j\{ g ^{jk}\xi_k, a\}+2\xi_j g^{jq}\{g_{qk},a\}   g^{k\mu}  \xi_\mu .
 \]
 Observing that $0=\{g^{jq}g_{qk},a \}=g^{jq}\{g_{qk},a \}+g_{qk}\{g^{jq},a \}$, we have
 \begin{align*}
  F&= g^{kp}g_{qk}\xi_p\xi_j\{g^{jq},a\}+2\xi_j\{ g ^{jk}\xi_k, a\}-2g_{qk}\xi_j \{g^{jq},a\}   g^{k\mu}  \xi_\mu \\
  &=\xi_q\xi_j\{g^{jq},a\}+2\xi_j\{ g ^{jk}\xi_k, a\}-2\xi_j   \xi_q \{g^{jq},a\} \\
  &=\xi_q\xi_j\{g^{jq},a\}+2 g ^{jq}\xi_j\{\xi_q, a\}.
 \end{align*}
As
\[
\{g^{jq}\xi_j\xi_q,a\}=g^{jq}\xi_j\{\xi_q,a\}+g^{jq}\xi_q\{\xi_j,a\}+\xi_j\xi_q\{g^{jq},a\}
=2g^{jq}\xi_j\{\xi_q,a\}+\xi_j\xi_q\{g^{jq},a\},
\]
by symmetry of $g^{jq}$. Then $F=\{g^{jq}\xi_j\xi_q,a\}$. Then from~\eqref{eq: braket with m} we have 
$\langle m,\{g^{jq}\xi_j\xi_q,a\}\rangle=0$ or $\langle \{g^{jq}\xi_j\xi_q,m\}, a \rangle=0$.  
This proves Proposition~\ref{prop: propagation formula} in $\Omega\times \R^d$.

 To prove Proposition~\ref{prop: propagation formula} at boundary, 
we have to use tangential operator, we consider System~\eqref{fluidec-3},  instead System~\eqref{eq: syst semi-class} 
and to conclude we have to
prove an analogous  result as Formula~\eqref{def: measure} where one integrate on $\Omega$ instead on $\R^d$.

Let us recall System~\eqref{fluidec-3}. The system satisfies by $u_h$ and $r_h$ is the following.
\begin{equation}
 	\label{eq: syst semi-class propagation}
\begin{cases}
&u_h-ih\nabla_g r_h=hf_h , \text{ in } \Omega \\
& -ih\div_g u_h+r_h=h q_h,  \text{ in } \Omega \\
&u_h. n=0 \text{ on } \pd\Omega
\end{cases}
\end{equation}
where $\|f_h \| +\| q_h\|\to 0$ as $h\to 0$, and $\| u_h\|+\|r_h\|$ bounded uniformly with respect $h$.

%
%
%
\subsubsection{Limit on quantities on $\Omega$.}

%
%
In the following result, we prove that the measure defined for $(\underline{u_h},\underline{r_h})$ also appear 
in the limit of some quantities computing on $x_d>0$.
\begin{prop}
	\label{prop: limit tangential}
We assume $a_j^k,b^j,c\in \Con^\infty(  \R^d_x\times\R^{d-1}_{\xi'}) $ compactly supported, satisfying $g^{\ell j}a_j^k=g^{kj}a_j^\ell$ and supported (in $x$) on a \nhd of a point of the boundary. We assume the metric $g$ written in geodesic 
coordinates.
Then 
\begin{enumerate}
\item $(\ops(a_j^k)u_h^j,u_h^k)_g$ converges to $ \langle \mu_j^k,a_k^j\rangle=\langle g^{kp}\xi_p\xi_j m,a_k^j\rangle$ 
as $h $ to 0.
\item  $((\ops(b^j)r_h),(u_h^k))_g+(g_{jk}\ops(b^j) u_h^k,r_h)_g$ converges to 
$ 2\langle \nu_j,b^j\rangle=-2\langle \xi_j m,b^j\rangle$ as $h $ to 0.
\item $(\ops(c)r_h,r_h)_g$ converges to 
$ \langle m, c\rangle$ as $h $ to 0.
\item $(\ops(c)(-ih\pd_d)u_h^d,u_h^d)_g$ converges to $ \langle \mu_d^d,c\xi_d\rangle=\langle \xi_d^2 m,c\xi_d\rangle$ 
as $h $ to 0.
\item  $(\ops(c)(-ih\pd_d)r_h,u_h^d)_g+(\ops(c)(-ih\pd_d) u_h^d,r_h)_g$ converges to 
$ 2\langle \nu_d,c\rangle=-2\langle \xi_dm,c\rangle$ as $h $ to 0.
\item $(\ops(c)(-ih\pd_d)r_h,r_h)_g$ converges to 
$ \langle m, c\xi_d\rangle$ as $h $ to 0.
\end{enumerate}
In previous formulas, the  integration in the inner product is taken on $\Omega$ that is, on local coordinates, 
on $x_d>0$.
\end{prop}
 \begin{proof}
From Proposition~\ref{Prop: measure support and form}, $m$ is supported on  $  (g^{jp}(x)\xi_p  \xi_j-1)=0$.
The $m$ is compactly supported and the 
expressions given at the limit make sens. 

As $r_h$ and $h\nabla_g r_h$ are in $L^2(x_d>0)$, the extension  $\underline{r}_h$ is in  $H^s_{sc}$ for $s<1/2$ and 
\[
\| \underline{r}_h\|_{H^s_{sc}}\lesssim \| r_h\|_{L^2(x_d>0)}+ \|h\nabla_g r_h\|_{L^2(x_d>0)}.
\]
To prove the three first item, we observe that we can replace $u_h$ and $r_h$  by the extension. Then we introduce 
$\ops( \phi(\xi_d/R))$ and Proposition~\ref{Prop: measure support and form}
gives the result as for instance $ \langle \mu_j^k,a_k^j\phi(\xi_d/R) \rangle= \langle \mu_j^k,a_k^j\rangle$ for
 $R$ sufficiently large. We have to estimate the remainder terms with $\ops(1- \phi(\xi_d/R))$. To do that 
 we have to distinguish different cases.
 To prove the first item, if $j\le d-1$ or $k\le d-1$ we have in the first case to estimate 
 $(\ops(a_j^k)\ops(1- \phi(\xi_d/R)) \underline{u}_h^j,\underline{u}_h^k)_g$. As $u^j=ih (\nabla_g r_h)^j+hf_h^j$
 we have 
 \[
 (\ops(a_j^k)\ops(1- \phi(\xi_d/R)) \underline{u}_h^j,\underline{u}_h^k)_g
 =  (\ops(a_j^k)\ops(1- \phi(\xi_d/R))( ih (\nabla_g \underline{r}_h)^j+h\underline{f}_h^j),\underline{u}_h^k)_g.
 \]
From the form of the metric $g$,  $ (\nabla_g r_h)^j$ does not contain term $\pd_d r_h$ and as $a_j^k$ is compactly 
supported and by symbolic calculus we can write $ \ops(a_j^k)(-ih\pd_\ell)=\ops(a_j^k\xi_\ell)$ for $\ell\le d-1$ modulo 
$h$ times a bounded operator on $L^2$. Obviously $\ops(a_j^k\xi_\ell) $ is bounded on $L^2$ and the term 
$\ops(1- \phi(\xi_d/R) \underline{r}_h$ can be estimate from estimate~\eqref{est: H s sc}

As this term goes to 0 as $R$ to $\infty$ and  $(\ops(a_j^k)u_h^j,u_h^k)_g$ does not depend on $R$, we obtain 
the result.

If $j=d$ and $k=d$, we have $u_h^d=ih\pd_d r_h+hf_h^d$. Then, from~\eqref{eq: syst semi-class propagation}, 
using $(u_h^d)_{|x_d=0}$, symbolic calculus and integrating by parts,  we have
 \begin{align*}
 2(\ops(a)u_h^d,u_h^d)&=(\ops(a)(-ih\pd_d) r_h,u_h^d)+(\ops(a)u_h^d,  (-ih\pd_d )r_h) +O(h)\\
 &=(\ops(a) r_h,(-ih\pd_d)u_h^d)+(\ops(a) (-ih\pd_d )u_h^d, r_h) +O(h)\\
 &= (\ops(a) r_h,(ih\div_{\tilde g}){\tilde u}_h-r_h)+(\ops(a)\big( (ih\div_{\tilde g} ){\tilde u}_h-r_h\big), r_h) +O(h)\\
  &= ( r_h,\ops(a)^*\big( (ih\div_{\tilde g}){\tilde u}_h-r_h\big) )+(\ops(a)\big( (ih\div_{\tilde g} ){\tilde u}_h-r_h\big), r_h) +O(h).
 \end{align*}
 As $(\ops(a) (ih\div_{\tilde g} )$ and $\ops(a)^*(ih\div_{\tilde g})$ are tangential bounded operators, we can conclude 
 as the previous terms, extending the $L^2$ functions to $\R^d$ by 0 in $x_d<0$ and introducing a cutoff 
 $\ops( \phi(\xi_d/R))$. Terms $(\ops(a) r_h, r_h)$ and  $((\ops(b^j)r_h),(u_h^k))_g+(g_{jk}\ops(b^j) u_h^k,r_h)_g$
 can be treated following  same method, as we can introduce the cutoff $\ops(1- \phi(\xi_d/R))$ on terms $r_h$.
 
 To treat the second and third items, we can follow the proof of first item, where we have already computed terms as 
 $(\ops(a) r_h, u_h^j)$ for $j\le d-1$ or $(\ops(a) r_h, u_h^d)$. 
 
 For the forth item we can replace $-ih\pd_d u_h^d$ by $ih\div_{\tilde g} {\tilde u}_h-r_h$ and we obtain analogous terms 
 treated in first item.
 
 For firth item, we can replace $ih\pd_d r_h$ by $u_h^d$ and $-ih\pd_d u_h^d$ by $ih\div_{\tilde g}{\tilde u}_h-r_h$.
 We can treat terms obtained as   previous terms.
 
 For sixth item we replace $ih\pd_d r_h$ by $u_h^d$ and we find terms already treated.
 \end{proof}

%
%

 \subsubsection{Computations at boundary, tangential symbol}
 
 Let $a(x,\xi')$ be a tangential symbol compactly supported in a \nhd of a point of boundary. We introduce the 
 same $A$ as in Formula~\eqref{eq: A interior tangential}, where $a$ is now a tangential symbol. We obtain 
 by the same simplification Formula~\eqref{eq: A formula 0}. We integrate by parts and two terms have to be consider
$ (\ops(a)r_h,ih\div_g u_h)_g$ and    $(\ops(a)u_h,ih\nabla_g r_h)_g$, as other terms do not give boundary terms. 
Applying~\eqref{eq: integration parts boundary}  the first gives
\[
 (\ops(a)r_h,ih\div_g u_h)_g= (ih\nabla_g\ops(a)r_h, u_h)_g-   (\ops(a)(r_h)_{|x_d=0},ih (u_h^d)_{|x_d=0})_\pd,
\]
and the second gives
 \[
 (\ops(a)u_h,ih\nabla_g r_h)_g=(ih\div_g\ops(a)u_h, r_h)_g   -   (\ops(a) (u_h^d)_{|x_d=0},ih (r_h)_{|x_d=0})_\pd,
 \]
 and as  $(u^d_h)_{|x_d=0}=0$, in both cases, the boundary term disappears. 
 From that, the rest of the proof does not change and we find the same conclusion with  $a$ a  tangential 
 symbol, $\langle \{g^{jq}\xi_j\xi_q,m\}, a \rangle=0$. 
 
%
%
%

\subsubsection{	Computations at boundary, tangential symbol with normal derivative}

We need computation with symbol as $a(x,\xi') \xi_d$, but we cannot compute directly considering previous formula.
For instance in the first term 
of \eqref{eq: A interior tangential} we should have a term $\pd_df_h$ if we replace $a$ by $a\xi_d$ and if we put 
$\pd_d$ at the right hand side of inner product we have term $\pd_d u_h$ and the tangential part of this term is not 
in $L^2$. We have to introduce another quantity.

From System~\eqref{fluidec-3}, we obtain the following  $2\times 2$ system, replacing $u_h^j$ in function of $r_h$ and 
$f_h^j$ as we do to obtain System~\eqref{eq: syst semi-class 2 by 2}, we have
\begin{equation}
	\label{System 2x2 boundary prop}
\begin{cases}
&-ih\pd_dr_h+ u_h^d -ihbu_h^d=h\tilde f_h^d \\
&-ih |g|^{-1/2}\pd_d ( |g|^{1/2}u_h^d)-ih\div _{\tilde g}(ih\nabla_{\tilde g}r_h)+r_h=h\tilde q_h,
\end{cases}
\end{equation}
where $\tilde f_h^d=if_h^d$, $\tilde q_h=iq_h+ih\div_{\tilde g}{\tilde f_h}+ih\div_{\tilde g}(b\tilde u)$, 
$\tilde f_h=(f^j_h)_{j= 1,\dots ,d-1} $, and $\tilde u_h=(u^j_h)_{j= 1,\dots ,d-1} $.

For a symbol $a(x,\xi')$ compactly supported, we introduce the following quantity
\begin{align}
	\label{A boundary 0}
A&= (\ops(a)(u_h^d-ih\pd_dr_h),ih\div _{\tilde g}(ih\nabla_{\tilde g}r_h)-r_h)_g \\
&\quad  -(\ops(a)(ih\div _{\tilde g}(ih\nabla_{\tilde g}r_h)-r_h),u_h^d-ih\pd_dr_h)_g   \notag \\
&\quad -(\ops(a)(-ih|g|^{-1/2}\pd_d( |g|^{1/2}u_h^d))-ih\div _{\tilde g}(ih\nabla_{\tilde g}r_h)
+r_h , u_h^d)_g\notag \\
&\quad +(\ops(a)u_h^d,-ih|g|^{-1/2}\pd_d( |g|^{1/2}u_h^d)
-ih\div _{\tilde g}(ih\nabla_{\tilde g}r_h)+r_h)_g.  \notag  
\end{align}
We can simplify some terms with $u_h^d$, and we obtain
\begin{align}
	\label{A boundary 1}
A&= (\ops(a)(-ih\pd_dr_h),ih\div _{\tilde g}(ih\nabla_{\tilde g}r_h)-r_h)_g \\
&\quad  -(\ops(a)(ih\div _{\tilde g}(ih\nabla_{\tilde g}r_h)-r_h),-ih\pd_dr_h)_g   \notag \\
&\quad -(\ops(a)(-ih|g|^{-1/2}\pd_d( |g|^{1/2}u_h^d)) , u_h^d)_g\notag \\
&\quad +(\ops(a)u_h^d,-ih|g|^{-1/2}\pd_d( |g|^{1/2}u_h^d))_g  .\notag  
\end{align}
We have to compute $A$ from two different manners.
First we integrate  by parts applying \eqref{eq: integration parts boundary}, we have 
 \begin{align*}
&(\ops(a)(-ih\pd_dr_h),ih\div _{\tilde g}(ih\nabla_{\tilde g}r_h)-r_h)_g
= ( (ih\div _{\tilde g}(ih\nabla_{\tilde g})-1)\ops(a)(-ih\pd_d)r_h,r_h)_g,\\[ 4pt]
& -(\ops(a)(ih\div _{\tilde g}(ih\nabla_{\tilde g}r_h)-r_h),-ih\pd_dr_h)_g 
=(ih\pd_d(\ops(a)(ih\div _{\tilde g}(ih\nabla_{\tilde g}r_h)-r_h)),r_h)_g \\
&\quad+ ih( \ops(a)(ih\div _{\tilde g}(ih\nabla_{\tilde g}r_h)-r_h)_{|x_d=0},(r_h)_{|x_d=0})_{\pd},\\[ 4pt]
& (\ops(a)u_h^d,-ih|g|^{-1/2}\pd_d( |g|^{1/2}u_h^d))_g=(-ih|g|^{-1/2}\pd_d( |g|^{1/2}\ops(a)u_h^d),u_h^d)_g, 
 \end{align*}
as $(u_h^d)_ {|x_d=0}=0 $ in the last computation, the boundary term is null.  

From \eqref{A boundary 1} and previous computations, we have 
 \begin{align}
 	\label{eq: first comp boundary terms prop}
 A&= ( (ih\div _{\tilde g}(ih\nabla_{\tilde g})-1)\ops(a)(-ih\pd_d)r_h,r_h)_g \\
&\quad  -(-ih\pd_d(\ops(a)(ih\div_{\tilde g}(ih\nabla_{\tilde g}r_h)-r_h)),r_h)_g \notag  \\
& \quad + ih( \ops(a)(ih\div_{\tilde g}(ih\nabla_{\tilde g}r_h)-r_h)_{|x_d=0},(r_h)_{|x_d=0})_{\pd} 
 \notag \\
&\quad -(\ops(a)(-ih|g|^{-1/2}\pd_d( |g|^{1/2}u_h^d)) , u_h^d)_g\notag \\
&\quad +(-ih|g|^{-1/2}\pd_d( |g|^{1/2}\ops(a)u_h^d),u_h^d)_g \notag  \\
&=(K_1r_h,r_h)_g+(K_2u_h^d,u_h^d)_g  
+ ih( \ops(a)(ih\div _{\tilde g}(ih\nabla_{\tilde g}r_h)-r_h)_{|x_d=0},(r_h)_{|x_d=0})_{\pd}, \notag
\end{align}
where 
\begin{align*}
&K_1= (ih\div _{\tilde g}(ih\nabla_{\tilde g})-1)\ops(a)(-ih\pd_d)  
-ih\pd_d(\ops(a)(ih\div _{\tilde g}(ih\nabla_{\tilde g})-1)) \\
&K_2=-ih|g|^{-1/2}\pd_d |g|^{1/2}\ops(a)-\ops(a)  (-ih|g|^{-1/2})\pd_d |g|^{1/2}
= [  -ih|g|^{-1/2}\pd_d |g|^{1/2}, \ops(a) ]  .
\end{align*}
The principal symbol of $K_2$ is $-ih\{ \xi_d, a\}$ which is a tangential symbol. We write
\begin{align*}
K_1&=[ ih\div _{\tilde g}(ih\nabla_{\tilde g})-1 ,  \ops(a) ] (-ih\pd_d) \\
&\quad + \ops(a)[  ih\div _{\tilde g}(ih\nabla_{\tilde g})-1 , -ih\pd_d ]  +[  \ops(a) , -ih\pd_d  ]  (ih\div _{\tilde g}(ih\nabla_{\tilde g})-1).
\end{align*}
The principal symbol of $K_1$ is then 
\begin{align*}
 &-ih \{ R(x,\xi')-1 ,a(x,\xi') \}\xi_d-ih  a(x,\xi')\{ R(x,\xi')-1 , \xi_d\}
-ih\{ a(x,\xi'),\xi_d \} (R(x,\xi')-1)\\
&=-ih \{ R(x,\xi') ,a(x,\xi') \}\xi_d-ih  a(x,\xi')\{ R(x,\xi') , \xi_d\}
-ih\{ a(x,\xi'),\xi_d \} (R(x,\xi')-1).
\end{align*}
Then from Proposition~\ref{prop: limit tangential}, we have 
\begin{align*}
&\lim _{h\to 0}h^{-1}\big((K_1r_h,r_h)_g+(K_2u_h^d,u_h^d)_g\big)\\
&= \langle m,   
-i \{ R(x,\xi') ,a(x,\xi') \}\xi_d-i a(x,\xi')\{ R(x,\xi') , \xi_d\}
-i\{ a(x,\xi'),\xi_d \} (R(x,\xi')-1)-i \xi_d^2   \{ \xi_d, a(x,\xi')\} \rangle\\
&=  \langle m,   -i \{ R(x,\xi') ,\xi_da(x,\xi') \}  -2i \xi_d^2   \{ \xi_d, a(x,\xi')\} \rangle,
\end{align*}
as $m$ is null on $\xi_d^2+R(x,\xi') -1$.
Observing that
\[
\{ \xi_d^2+R(x,\xi') , \xi_da(x,\xi') \} =2\xi_d^2\{\xi_d,a(x,\xi')\}+\{R(x,\xi') ,\xi_d a(x,\xi') \} ,
\]
we obtain 
\begin{align}
	\label{eq: intern term boundary prop}
\lim _{h\to 0}h^{-1}\big((K_1r_h,r_h)_g+(K_2u_h^d,u_h^d)_g\big)
= \langle m,-i\{ \xi_d^2+R(x,\xi') , \xi_da(x,\xi') \} \rangle.
\end{align}
Second we compute  the limit of $h^{-1}A$, using $2\times 2$ system. From~\eqref{System 2x2 boundary prop} and 
\eqref{A boundary 0} we have
\begin{align*}
A&=
(\ops(a) (h\tilde f_h^d+ihbu_h^d), ih\div _{\tilde g}(ih\nabla_{\tilde g}r_h)-r_h)_g  \\
& \quad
-(\ops(a)(ih\div _{\tilde g}(ih\nabla_{\tilde g}r_h)-r_h),ihbu_h^d+h\tilde f_h^d)_g  
-(h\ops(a)\tilde q_h,  u_h^d)_g
+(\ops(a)u_h^d,h\tilde q_h )_g.
\end{align*}
Observe that $\pd_j \ops(a) $ and $\ops(a) \pd_j$ for $j=1,\dots d-1$ are  tangential operators bounded 
on $L^2$ as $a(x,\xi')$ is compactly supported. Some terms in $\tilde q_h$ 
can be write $\ops(a)\div_{\tilde g} \tilde f_h$ or  $\ops(a)^*\div_{\tilde g} \tilde f_h$, these terms are 
bounded on $L^2$ and disappear passing to the limit. 
Then we have
\begin{align*}
A&=(ih\ops(a)( bu_h^d), ih\div _{\tilde g}(ih\nabla_{\tilde g}r_h)-r_h)_g  
-(\ops(a)(ih\div _{\tilde g}(ih\nabla_{\tilde g}r_h)-r_h),ihbu_h^d)_g  \\
& \quad-(h\ops(a)ih\div_{\tilde g}   (b\tilde u_h)  ,  u_h^d)_g
+(\ops(a)u_h^d,ih^2\div_{\tilde g}   (b\tilde u_h)   )_g+ o(h).
\end{align*}
Integrating by parts, the two last terms, we have
\begin{align*}
&-(h\ops(a)ih\div_{\tilde g}   (b\tilde u_h)  ,  u_h^d)_g
+(\ops(a)u_h^d,ih^2\div_{\tilde g}   (b\tilde u_h)   )_g\\
&\quad=-(h\ops(a)ih\div_{\tilde g}   (b\tilde u_h)  ,  u_h^d)_g
+(ih^2 b \nabla_{\tilde g}    \ops(a)u_h^d,\tilde u_h  )_g\\
&=h(\ops (M)u_h,u_h)_g,
\end{align*}
where $M$ is a matrix of symbole supported on the support of $b$. This term gives 0 passing to the limit in $h$ 
as $bm=0$ from Proposition~\ref{prop: bm null}.

From Proposition~\ref{prop: limit tangential}, item (2), choosing $b^d= iab(R-1)$ and $b^j=0$ for $j=1,\dots d-1$, we obtain
\begin{equation}
	\label{eq: second comp boundary prop}
\lim_{h\to 0}h^{-1}A=\langle  m ,  -2i\xi_d   a(x,\xi')b(x)(R(x,\xi')-1)        \rangle =0,
\end{equation}
from Proposition~\ref{prop: bm null} and 
from~\eqref{eq: first comp boundary terms prop}, \eqref{eq: intern term boundary prop} 
and \eqref{eq: second comp boundary prop}, we obtain that 
\[
 ( \ops(a)(ih\div _{\tilde g}(ih\nabla_{\tilde g}r_h)-r_h)_{|x_d=0},(r_h)_{|x_d=0})_{\pd} ,
\]
has a limit when $h$ goes to 0, and we have
\[
\langle\{ \xi_d+R(x,\xi')    ,  m \}  , \xi_da(x,\xi')\rangle 
= \lim_{h\to 0}  (- \ops(a)(ih\div _{\tilde g}(ih\nabla_{\tilde g}r_h)-r_h)_{|x_d=0},(r_h)_{|x_d=0})_{\pd}.
\]

As the principal symbol of $\ops(a)(ih\div _{\tilde g}(ih\nabla_{\tilde g})-1)$ is
$a(x,\xi') (R(x,\xi') -1)(x,\xi')$, we obtain the last result of Proposition~\ref{prop: propagation formula}.

Proposition~\ref{prop: propagation support semiclassical measure} is a consequence of 
Proposition~\ref{prop: propagation formula}. To prove that
we follows works \cite{Gerard-1991, GL-1993, Leb, lebeau, Burq-1997, Burq-Lebeau2001, Miller-2000}. To be complete, we give the proof in Annex~\ref{sec: Precise description of Geometric Control Condition}. 
This proof follows the one  given in \cite{Cor-Rob}.

 %
 %
 
 \appendix

%
%

\section{Geometric Control Condition and propagation}
	\label{sec: Precise description of Geometric Control Condition}

%
%

\subsection{Geometry}\label{sec: Geometry}
Here we give the geometrical notion we use in this article. This framework comes from Melrose and 
Sj\"ostrand~\cite{MS-1978,MS-1982} and 
the reader may also find in 
H\"ormander~\cite[Chapter 24]{HormanderV3-2007} more informations and proofs. The characterisation of symplectic sub-manifold 
is probably classical and more details can be found in Grigis~\cite{Grigis-1976}.

\medskip
\textbf{Assumption on the symbol.}
We recall the definition of  the symbol of $P$ given in Proposition~\ref{prop: propagation formula},
\begin{align}\label{def: symbol h2P-1}
p(x,\xi)=\sum_{1\le j ,k\le d}g^{jk}(x)\xi_j\xi_k -1,
\end{align}
where $g^{jk}$ are $\Con^\infty(\overline{ \Omega})$.
Locally in a \nhd of the boundary we can define $\Omega $ by $\varphi>0$ with $d\varphi\ne 0$. 
We can also choose coordinates  (i.e. \emph{normal geodesic coordinates}) such that
$\varphi(x)=x_d$ and $p(x,\xi)= \xi_d^2+R(x,\xi')-1$ where $x=(x',x_d) $ and $\xi=(\xi',\xi_d)$. 

\medskip
\textbf{Symplectic sub-manifold $\Sigma$.}
We can define a symplectic manifold $\Sigma$,
 contained into $T^*\R^d\cap\{ \varphi=0 \}$. We set $\Sigma=\{  (x,\xi),\ \varphi(x)=0 \text{ and } 
 \{\varphi, p\} (x,\xi)=0\}$. The set $\Sigma$ is a 
 symplectic manifold as $\{\varphi  ,\{\varphi,p\}\}\ne0$. In coordinates $(x',x_d)$, we have 
 $\Sigma=\{x_d=0,\xi_d=0\}$, this manifold is isomorphic to 
 $T^*\pd\Omega$ and described by coordinates $(x',\xi')$. 
 
 The Hamiltonian vector field $H_p$ is not a 
 vector field on $\Sigma$, but for all $X$ a vector field on $T^*\R^d$, we can find
 unique fonctions $\alpha$ and $\beta$ such $X+\alpha H_\varphi+\beta H_{\{\varphi, p\} }$ is a vector 
 field on  $\Sigma$. For $H_p$ we denote the associated vector $H'_p$
 and an elementary computation leads to 
 \begin{align*}
 H'_p=H_p+\frac{ \{ p ,\{ p ,\varphi  \} }{ \{ \varphi ,\{ \varphi  , p \}}  H_\varphi .
 \end{align*}
 In coordinates  $(x',x_d)$,   $H'_p$ only depends  on $R$ and we have
  \begin{align*}
   H'_p=H'_R=\sum_{j=1}^{d-1} (\pd_{\xi_j}R(x',0,\xi')\pd_{x_j} - \pd_{x_j}R(x',0,\xi' )  \pd_{\xi_j}).
  \end{align*}
In particular the integral curves associated with $H'_p$ starting from a point into $\Sigma$ stay into $\Sigma$. In coordinates $(x',x_d)$, we denote the integral curve 
starting from $(x',\xi')$, either $\gamma_g(s; x',\xi')$, either  $\gamma_g( x',\xi')$, if $s$ is implicit or  $\gamma_g(s)$, if $(x',\xi')$ is implicit.

\medskip
\textbf{Description and topology of $T^*_b\Omega$.}
Let $T^*_b\Omega=T^*\pd\Omega\cup T^*\Omega$,
 this set is equipped with the following topology.

  First if $\rho\in T^*\Omega$, 
  a set $V$ is a \nhd of $\rho$ if $V$ contains an open set $W$ of 
  $ T^*\Omega$  such that $\rho\in W$.  
  
  Second if $\rho=(x_0',\xi_0')\in T^*\pd\Omega$, a set $V$ is a 
  \nhd of $\rho$ if $V$ contains a set
  \begin{align*}
 &  \{ (x',\xi')\in T^*\pd\Omega,\ |x'_0-x'|+|\xi'_0-\xi'|\le\eps \}\\
  &  \cup \{ (x,\xi)\in T^*\Omega,\ 
   |x'_0-x'|+|\xi'_0-\xi'|\le\eps \text{ and } (x_d,\xi_d)\in
  U\cap \{ x_d>0 \} \},
  \end{align*}
 where $\eps>0$ and $U$ is a \nhd of $\{ (x_d,\xi_d)\in\R^2, \ x_d=0\}$ in $\R^2$.
 
 In local coordinates where $\Omega $ is defined by $x_d>0$, we define 
 $j: T^*\overline \Omega \to T^*_b\Omega$
 by $j(x,\xi)=(x,\xi)$ if $x\in\Omega$, and $j(x,\xi)=(x',\xi')$ if $x_d=0$. 
 The map $j$ is continuous for the topology given above.
  We can define more intrinsically  $j$ with the previous notation where $\Omega$ is 
  given by $\varphi(x)>0$. For 
  $(x,\xi)\in T^*\overline \Omega $, $j(x,\xi)=(x,\xi)$ if $x\in\Omega$ and 
  $j(x,\xi)=(x, \xi dx-(\{p,\varphi\}/H^2_\varphi p)d\varphi)$, if 
  $\varphi(x)=0$. We verify, in this last case that $j(x,\xi)\in \{  \varphi=\{p,\varphi\}=0 \}$, 
  as $\{p,\varphi\}(x,d\varphi)
  =\{\varphi,\{ \varphi,p\}\}$.

  As usually we define the map $\pi: T^*_b\Omega\to \overline \Omega$, in local coordinates, 
  as $\pi(x,\xi)=x$, if $(x,\xi)\in T^*\Omega$ and $\pi(x',\xi')=(x',0)$, if $(x',\xi')\in T^*\pd\Omega$.
  
\medskip
\textbf{Bicharacteristic and generalized flow.}

\medskip

For $(x,\xi)\in T^*\R^d$, we denote by $\gamma(s;x,\xi)$ the integral curve of 
$H_p$ starting from $(x,\xi)$. We use the same short notation $\gamma(s)$ 
and $ \gamma(x,\xi)$ as above.

Now we define the generalized bicharacteristic denoted by $\Gamma(s,\rho)$ for 
$\rho\in T^*_b\Omega$. To describe this curve in a 
\nhd of the boundary we use the coordinates 
$(x',x_d,\xi',\xi_d)$ and we identify $\Sigma'$ and $T^*\pd\Omega$ and locally 
$\Omega=\{x\in \R^d , \ x_d>0 \}$.  
Moreover, we assume 
$\rho\in \text{char} (P) = \{(x,\xi)\in T^*\Omega ,\ p(x,\xi)=0  \}\cup 
\{(x',\xi')\in T^*\pd\Omega,\ R(x',0,\xi')-1\le0  \}$.

  Now we define the curve $\Gamma(s,\rho)$ locally for each $(s_0,\rho)$ and 
  we use the group property of the flow, namely 
  $\Gamma(s+t,\rho)=\Gamma(s,\Gamma(t,\rho))$ 
  to extend this function for every $s\in\R$.
  
  \medskip
For  $\Gamma(s_0,\rho)\in  
  T^*\Omega$, $\Gamma(s,\rho)=
  \gamma(s-s_0;\Gamma(s_0,\rho))$ if $\gamma(s-s_0;\Gamma(s_0,\rho))\in T^*\Omega$ for $s-s_0$ in an
  interval containing 0. In particular, 
  this defined $\Gamma(s,\rho)$ at least 
  for $s$ in a \nhd of $s_0$ as $\gamma(s-s_0;\Gamma(s_0,\rho))$ stay in $T^*\Omega $ for small $|s-s_0|$. 
  Observe that $p(\gamma(s-s_0; \Gamma(s_0,\rho)))=0$.
  
  \medskip
  For $\rho=(x_0',\xi'_0)\in T^*\pd\Omega$, we have to distinguish different cases, first if $R(x'_0,0,\xi'_0)<1$ 
  and second if $R(x'_0,0,\xi'_0)=1$, where
  $\Gamma(s,\rho)$ depends on the properties of $\gamma(s; x'_0, 0, \xi'_0,0)$. In what follows we only 
  define the flow in a \nhd of $s=0$. We can extend the flow by the group property.
  
  We begin to treat hyperbolic points and we recall 
  \(
  {\mathcal H}=\{ (x'_0,\xi'_0)\in T^*\pd\Omega,\ R(x'_0, 0,\xi'_0)<1 \}.
  \)
If $R(x'_0,0,\xi'_0)<1$, let $\xi^\pm=\pm\sqrt{1-R(x'_0,0,\xi'_0)}$. Let 
  $\gamma(s;x_0,\xi_0)=(x(s;x_0,\xi_0), \xi(s;x_0,\xi_0))$, as $\dot x_d=2\xi_d$, we have 
  $x_d(s;x'_0,0,\xi'_0,\xi^+)>0$ for $s>0$ sufficiently small, and  $x_d(s;x'_0,0,\xi'_0,\xi^-)>0$ for 
  $s<0$ sufficiently small. Then we set 
  $\Gamma(0,\rho)=\rho$, $ \Gamma(s,\rho)= \gamma(s;x_0',0,\xi_0',\xi^+)$ for $s>0$ sufficiently small and 
  $\Gamma(s,\rho)= \gamma(s;x_0',0,\xi_0',\xi^-)$ for $s<0$ sufficiently small. Observe that for 
  $s\ne0$ sufficiently small,
  $\Gamma(s,\rho)\in T^*\Omega$ and $p(\Gamma(s,\rho))=0$. 
%
%
    \begin{defi}[Finite contact with the boundary]
  	\label{def: contact fini}
  Let $(x'_0,\xi'_0)$ be such that $R(x'_0,0,\xi'_0)=1$. We say that  the  bicharacteristic $ \gamma(s;x'_0,0,\xi'_0,0)=\gamma(s)$ 
  \emph{does not have  an infinite contact with the boundary} if 
there exists $ k\in \N,\ k\ge 2,\  \alpha\ne0$ such that  $ x_d(s;x'_0,0,\xi'_0,0)= x_d(s)= 
  \alpha s^k+{\mathcal  O}(s^{k+1})$ in a \nhd of $s=0$. We denote by ${\mathcal G}^k$ the set of such points.
  
For   \(k=2\) we distinguish two cases.
\begin{itemize}
\item The  diffractive points, and we denote 
\[
{\mathcal G}_d=  \{ (x',\xi')\in  T^*\pd\Omega,  \  R(x',0,\xi')=1,\  \pd_{x_d}R(x',0,\xi')<0 \}.
\]
In this case $\alpha$ defined above is positive.
\item The gliding points, and we denote  
\[
{\mathcal G}_g=  \{ (x',\xi')\in  T^*\pd\Omega,  \  R(x',0,\xi')=1,\  \pd_{x_d}R(x',0,\xi')>0 \}.
\]
In this case $\alpha$ defined above is negative.
\end{itemize}
 \end{defi}
      \begin{remark}
    By Taylor's theorem and 
 as $x_d(0)=0$ and 
  $\dot x_d(0)=2\xi_d(0)=0$, we always have $x_d(s)={\mathcal O}(s^2)$.
  \end{remark} 
We have four cases to treat.
\begin{enumerate}
\item[-]   $k$ even, $\alpha>0$. In this case $x_d(s)>0$ for $s\ne0$ sufficiently small. 
We define $\Gamma(0,\rho)=\rho$ and $\Gamma(s,\rho)=\gamma(s;x_0',0,\xi'_0,0)\in T^*\Omega$ 
for $s\ne0$ sufficiently small.
\item[-] $k$ even, $\alpha<0$. In this case $x_d(s)<0$ for $s\ne0$ sufficiently small. 
We define $\Gamma(s)=\gamma_g(s,\rho)\in T^*\pd\Omega$ for $s$ sufficiently small.
\item[-] $k$ odd, $\alpha>0$. In this case $x_d(s)>0$ for $s>0$ sufficiently small and   
$x_d(s)<0$ for $s<0$ sufficiently small. We define 
$\Gamma(s,\rho)=\gamma_g(s,\rho)\in T^*\pd\Omega$ for $s\le0$ sufficiently small, and $\Gamma(s,\rho)=\gamma(s;x_0',0,\xi'_0,0)   \in T^*\Omega$ for $s>0$, sufficiently small.
\item[-] $k$ odd, $\alpha<0$. In this case $x_d(s)<0$ for $s>0$ sufficiently small and   $x_d(s)>0$ for $s<0$ sufficiently small. We define 
$\Gamma(s,\rho)=\gamma(s;x_0',0,\xi'_0,0)  \in T^*\Omega$ for $s<0$ sufficiently small, and $\Gamma(s,\rho)=\gamma_g(s,\rho)\in T^*\pd\Omega$ for $s\ge0$, sufficiently small.
\end{enumerate}
This local description of $\Gamma(s)$ allows to extend $\Gamma(s) $ for every $s\in\R$. 
The function $\Gamma(s, \rho) $ defined on $\R\times\text{char}(P)$ is 
continuous for the topology of $\R\times T^*_b\Omega$, where the topology 
of $T^*_b\Omega$ is defined above. 

%
%
%

\subsection{Geometric Control Condition (GCC)}
	\label{Sec: Geometric Control Condition (GCC)}

We may give here a precise definition of (GCC).
%
%
\begin{defi}
	\label{def: precise GCC}
Let  $\Omega$ be a smooth open domain in $\R^d$.
Let $\omega$ be an open subset of $\Omega$. We assume the bicaracteristics do not have 
contact of infinite order with
 the boundary (see Definition~\ref{def: contact fini}). We say that $\omega $ satisfies (GCC) if  for every  
$\rho\in T^*_b\Omega$ there exists  $s>0$ such that  $\pi(\Gamma(s,\rho))\in \omega$.
\end{defi}
%
%
\begin{remark}
By symmetry the assumption $s>0$ is not relevant as if $\rho=(x,\xi) $ and $\tilde \rho=(x,-\xi)$ we have 
$\Gamma(s,\rho)=\Gamma(-s,\tilde \rho)$.

By compactness argument if (GCC) is satisfied, there exist $s_0$ such that for every $\rho\in T^*_b\Omega$, there
exist $s\in[0,s_0]$ such that $\pi(\Gamma(s,\rho))\in \omega$.  The value of $s_0$ play a role for control problems 
as $s_0$ is related with the minimal time of control, but is not relevant for stabilization problems.
\end{remark}

 %
 %
 %
 \subsection{Properties of measure}
 
In this section, we give several properties on measure $m$, in particular in \nhd of boundary. This description is 
related to  the property of point in $T^*(\pd \Omega)$.

In next lemma, we decompose the measure into an interior measure and a measure  supported at boundary. Moreover this last measure is supported
on $\xi_d=0$.
\begin{lem} \label{lem: measure on boundary}
There exists a non negative Radon  measure $m^\pd$ on $x_d=\xi_d=0$ such that $m=1_{x_d>0}m +m^\pd\otimes \delta _{x_d=0}\otimes\delta_{\xi_d=0}$. 
Furthermore $m^\pd$ is supported on $R(x',0,\xi')-1=0$.
\end{lem}
\begin{proof}
We apply Formula~\eqref{prop: propagation formula 3}.
 We observe that $H_p=2\xi_d\pd_{x_d}-(\pd_{x_d}R)\pd_{\xi_d}+H'_R$ where $H'_R=
\ds \sum_{j=1}^{d-1}\big((\pd_{\xi_j}R)\pd_{x_j}-(\pd_{x_j}R)\pd_{\xi_j}\big)$. We have $H_p(a\xi_d )=2\xi_d^2(\pd_{x_d}a) -(\pd_{x_d}R) a+\xi_dH'_R a$. Let 
$a=a^\eps=\eps\chi(x_d/\eps)\ell(x,\xi')$, where $\chi\in\Con_0^\infty(\R)$, such that $\chi(0)=0$ and $\chi'(0)=1$, $\ell\in\Con_0^\infty(\R^d\times\R^{d-1})$. We have
\begin{align*}
H_p(\xi_d a^\eps)& =2\xi_d^2\chi'(x_d/\eps)\ell(x,\xi')+2\eps \chi(x_d/\eps)
\xi_d^2\pd_{x_d}\ell-(\pd_{x_d}R)\eps\chi(x_d/\eps)\ell(x,\xi')
+\xi_d\eps\chi(x_d/\eps)H'_R \ell.
\end{align*}
Clearly $H_p(\xi_da^\eps)$ is uniformly bounded on the support of $m$
and $H_p(\xi_d a^\eps)\to 
2\xi_d^2\chi'(0)\ell 1_{x_d=0}$ everywhere as $\eps\to0$.
Then by Lebesgue's dominated convergence theorem we have 
$\est{m,H_p(\xi_da^\eps)}\to\est{m, 2\xi_d^2 \ell 1_{x_d=0}}$ as $\eps\to0$. As $\chi(0)=0$, the 
right hand side of~\eqref{prop: propagation formula 3} is 0 for every $h$. 
Then $\est{m, 2\xi_d^2 \ell 1_{x_d=0}}=0$. 
This means that $1_{x_d=0}m$ is supported on $\xi_d=0$. We denote $m^\pd$ the measure $1_{\xi_d=0}1_{x_d=0}m$. 
As $m=1_{x_d>0}m+1_{x_d=0}m$, 
we have  $m=1_{x_d>0}m+m^\pd\otimes \delta _{x_d=0}\otimes\delta_{\xi_d=0}$. We have by 
Proposition~\ref{Prop: measure support and form}, $(\xi_d^2+R(x,\xi')-1)1_{x_d=0}m=0 $, 
then $(R(x',0,\xi')-1)m^\pd=0$. This gives the conclusion of Lemma.
\end{proof}

 %
 %
 %

The Hamiltonian of the interior measure is a priori a distribution of order one supported 
on $x_d=0$. The following lemma says that this quantity is a measure if the 
Hamiltonian vector field is transverse to the boundary.
\begin{lem}\label{lem: Hp mu xd positive is a measure}
We have $H_p(m 1_{x_d>0})=\delta_{x_d=0}\otimes m_0$, where  $m_0$ is a distribution  
of order 1 defined on $x_d=0$.
Moreover, in a \nhd where $H_p x_d>0$, $m_0$ is a non negative Radon measure, and in a \nhd where $H_p x_d<0$, $m_0$ is a non positive Radon measure.
\end{lem}
\begin{proof}
The support of $H_p(m 1_{x_d>0})$ as a distribution if $x_d\ge0$ and $H_p(m)=0$ on $x_d>0$. Then $H_p(m 1_{x_d>0})$ is supported on $x_d=0$.
This implies that there exist $n\ge0$ and $m_j$ distributions on $x_d=0$ such that $H_p(m 1_{x_d>0})= \ds \sum_{j=0}^n\delta_{x_d=0}^{(j)}\otimes m_k$.
Suppose that $n\ge 1$  and 
let $\chi\in\Con_0^\infty(\R)$, be such that $\chi^{(k)}(0)=0$ for $k=0,\dots n-1$ and $\chi^{(n)}(0)=1$. Let 
$b\in \Con_0^\infty(\R^d\times\R^{d-1})$,
we have 
\begin{align}\label{eq: to calculate Hp mu}
-\est{m 1_{x_d>0},H_p\big(\eps^n\chi(x_d/\eps)b(x,\xi')\big)}& = \est{H_pm, \eps^n\chi(x_d/\eps)b(x,\xi')} \notag  \\
&=\est{\sum_{j=0}^n\delta_{x_d=0}^{(j)}\otimes m_k, \eps^n\chi(x_d/\eps)b(x,\xi')} \notag \\
&=(-1)^n \est{m_n, b(x',0,\xi')}.
\end{align}
We also have 
\begin{align*}
H_p\big(\eps^n\chi(x_d/\eps)b(x,\xi')\big)= 2\xi_d\eps^{n-1}\chi'(x_d/\eps) b(x,\xi')+\eps^n\chi(x_d/\eps)H_p b(x,\xi'),
\end{align*}
which is bounded uniformly with respect $\eps$ and supported in a fixed compact set.
Then if $n\ge 2$,  $H_p\big(\eps^n\chi(x_d/\eps)b(x,\xi')\big)\to0$ everywhere as $\eps\to0$ 
and~\eqref{eq: to calculate Hp mu}, Lebesgue's dominated convergence theorem imply that $m_n=0$ for $n\ge2$.
If $n=1$,  $H_p \big(\eps\chi(x_d/\eps)b(x,\xi')\big)\to  2\xi_d 1_{x_d=0}b(x,\xi')$ everywhere as $\eps\to 0$. Lebesgue's dominated convergence 
theorem and~\eqref{eq: to calculate Hp mu} imply that
\begin{align*}
\est{m 1_{x_d>0},2\xi_d 1_{x_d=0}b(x,\xi')}=\est{m_1,b(x',0,\xi')},
\end{align*}
as $ \est{m 1_{x_d>0},2\xi_d 1_{x_d=0}b(x,\xi')}=0$ we find that $m_1=0$. Then we have 
$  H_p(m 1_{x_d>0})=\delta_{x_d=0}\otimes m_0$, where $m_0$ is a distribution of order  1.

If  $H_px_d\not=0$, let $\gamma(s;x,\xi)=(x(s; x',\xi), \xi(s;x',\xi))$ be the flow associated with   $H_p$. We verify that the map
$(s,x',\xi)\to (x(s; x',\xi), \xi(s;x',\xi))$ locally is one to one and transforms $\pd_s$ in $H_p$. Moreover, $s=0$ is transformed in 
$x_d=0$ and if $H_p x_d>0$, $s>0$ is transformed in $x_d>0$, if  $H_p x_d<0$, $s<0$ is transformed in $x_d>0$. In coordinates $(s,x',\xi)$ the equation
$  H_p(m 1_{x_d>0})=\delta_{x_d=0}\otimes m_0$ is transformed in $\pd_s (m 1_{s>0})=\delta_{s=0}\otimes m_0$ if $H_p x_d>0$ and 
$\pd_s (m 1_{s<0})=\delta_{s=0}\otimes m_0$ if $H_p x_d<0$, where we keep the notations $m, m_0$ in variables  $(s,x',\xi)$  for the images of 
$m, m_0$. If $H_p x_d>0$, we have $m 1_{s>0}=(1_{s>0}ds)\otimes m_0$ and if $H_p x_d<0$, we have $m 1_{s<0}=-(1_{s<0}ds)\otimes m_0$. As 
$m$ is non negative, we obtain that  $m_0$ is a nonnegative measure in the first case and a nonpositive measure in the second case.
\end{proof}

 %
 %
 %

At the hyperbolic region the measure $m_0$ has a particular structure given by next lemma. 
\begin{lem}
	\label{lem: measure hyperbolic points}
Let $(x_0',\xi_0')$ be a hyperbolic point (i.e. $R(x'_0,0,\xi'_0)<1$). Locally in a \nhd of  $(x_0',\xi_0')$, there exists $m^+$ a non negative 
measures on $\R^{d-1}_{x'}\times \R^{d-1}_{\xi'}$ such that 
\[
m_0=m^+\otimes \delta_{\xi_d=\sqrt{1-R(x',0,\xi')}}-m^+\otimes \delta_{\xi_d=-\sqrt{1-R(x',0,\xi')}}.
\]
\end{lem}
\begin{proof}
As $m $ is supported on $\xi_d^2+R(x,\xi')-1=0$, we have $m=m1_{x_d>0}$  in a \nhd of hyperbolic points and  
 $\delta_{x_d=0}\otimes m_0$ is supported on  $\xi_d^2+R(x,\xi')-1=0$.
Then $ m_0$ is supported on $\xi_d=\pm \sqrt{1-R(x',0,\xi')}$. Moreover $H_p=2\xi_d\pd_{x_d}+X$ where 
$X$ is a vector field tangent to $x_d=0$, 
in particular $H_p x_d=2\xi_d$ on $x_d=0$. This implies that $H_p x_d>0$ if $\xi_d= \sqrt{1-R(x',0,\xi')}$ 
and  $H_p x_d<0$ if $\xi_d=- \sqrt{1-R(x',0,\xi')}$. 
From Lemma~\ref{lem: Hp mu xd positive is a measure}, we obtain 
$m_0=m^+\otimes \delta_{\xi_d=\sqrt{1-R(x',0,\xi')}}-m^-\otimes \delta_{\xi_d=-\sqrt{1-R(x',0,\xi')}}$. From 
\eqref{prop: propagation formula 2}
we have $\est{H_pm, a}=0$ which implies $\est{m^+-m^-,a_{|x_d=0}}=0$ 
as $a_{|x_d=0}$ describes every $\Con_0^\infty$ functions supported in a \nhd of $(x_0',\xi_0')$, this implies 
that $m^+=m^-$ in a \nhd of  $(x_0',\xi_0')$.
\end{proof}

 %
 %
 %

We recall   that 
 \[\mathcal{G}_d=\{ (x',\xi'),\ R(x',0,\xi')=1 \text{ and }\pd_{x_d}R(x',0,\xi')<0  \}\] is the set of diffractive points.

%
%
\begin{lem}
	\label{lem: diffractive points non in measure support}
We have $1_{\mathcal{G}_d}m^\pd=0$. 
\end{lem}
\begin{proof}

We apply Proposition~\ref{prop: propagation formula} and Formula~\eqref{prop: propagation formula 3}.
We have to choose an adapted function $a$. Let $\chi\in\Con^\infty$ be such that $\chi(\sigma)=0$ if $|s|\ge 2$, $ \chi(\sigma)=1$ if  $|\sigma|\le 1$.
We apply \eqref{prop: propagation formula 3} with 
$a(x, \xi')= \chi((1-R(x,\xi'))/\eps)\ell(x,\xi')\chi(x_d/\eps)$, where  $\eps>0$ will be chosen in what follows and $\ell$ 
is supported  in a \nhd of a point of $\pd\Omega\times\R\times\R^{d-1}$.
We recall that  $H_p=2\xi_d\pd_{x_d}-(\pd_{x_d}R)\pd_{\xi_d}+H'_R$ (see the proof of Lemma~\ref{lem: measure on boundary}).
We have
\begin{align}
	\label{eq: calcul de Hp a xi d}
H_p(\xi_d a)&=  2\xi_d^2\Big(    -(\pd_{x_d}R  ) \chi'((1-R(x,\xi'))/\eps)\ell(x,\xi')\chi(x_d/\eps)/\eps  \\
&\quad + \chi((1-R(x,\xi'))/\eps)(\pd_{x_d}\ell(x,\xi')) \chi(x_d/\eps) \notag \\
&\quad + \chi((1-R(x,\xi'))/\eps) \ell(x,\xi')\chi'(x_d/\eps)/\eps   \Big) \notag \\
&\quad  -(\pd_{x_d}R  )  \chi((1-R(x,\xi'))/\eps)\ell(x,\xi')\chi(x_d/\eps) \notag \\
&\quad  + \chi((1-R(x,\xi'))/\eps)\chi(x_d/\eps)  \xi_d H'_R  \ell(x,\xi').\notag
\end{align}
We 
claim 
\begin{align}\label{property: limit Hp a xi d}
&H_p(a \xi_d) \text{ is uniformly bounded on } \xi_d^2+R(x,\xi')=1, \\
&H_p(a \xi_d)\to -(\pd_{x_d}R(x',0,\xi'))1_{R(x',0,\xi')=1}1_{x_d=0}\ell (x',0,\xi')  \text{ as }\eps\to 0 \notag\\ 
&\text{ for all } (x,\xi), \text{ such that } \xi_d^2+R(x,\xi')=1.\notag
\end{align}
As $m$ is supported on $\xi_d^2+R(x,\xi')-1=0$, then  $\xi_d^2/\eps=(1-R(x,\xi'))/\eps$ on the support of 
$m$, this implies that the three first terms 
in~\eqref{eq: calcul de Hp a xi d} are bounded. It is easy to prove that they converge to $0$ as $\eps$ to $0$.
The fourth term is bounded and converges to 
$$-(\pd_{x_d}R(x',0,\xi'))1_{R(x',0,\xi')=1}1_{x_d=0}\ell(x',0,\xi'). $$ 
In the last term as $|    R(x,\xi')-1 |/\eps$ is bounded, thus 
$|\xi_d|$ is bounded by $C\sqrt \eps$ and then this term converges to 0 as $\eps\to0$. This 
proves~\eqref{property: limit Hp a xi d}.
From that we can conclude that 
$\est{m, H_p(\ell \xi_d)}$ converges to 
\begin{multline}
	\label{Form: diffractive case, boundary measure}
\est{m, -(\pd_{x_d}R(x',0,\xi'))1_{R(x',0,\xi')=1}1_{x_d=0}\ell (x',0,\xi') } 
\\ =\est{m^\pd, -(\pd_{x_d}R(x',0,\xi')) 1_{R(x',0,\xi')=1}\ell(x',0,\xi') } ,
\end{multline}
as $\eps\to 0$.

Let 
\begin{align*}
A_\eps&= - \lim_{h\to 0}(\ops(a(R-1))(r_h)_{|x_d=0},(r_h)_{|x_d=0})_{\pd}\\
&= - \lim_{h\to 0}(\ops(  \chi((1-R(x,\xi'))/\eps)\ell(x,\xi') (R(x,\xi')-1))(r_h)_{|x_d=0},(r_h)_{|x_d=0})_{\pd}.
\end{align*}

From Proposition~\ref{prop: boundary term diffractive Neumann}, \(A_\eps\) goes to 0 as 
\( \eps \to 0\). Then from \eqref{Form: diffractive case, boundary measure} 
\(1_{\mathcal{G}_d}m^\pd=0\) as $\ell$ can be chosen arbitrary.
\end{proof}

 %
 %
 %

This lemma describes how the support of the boundary measure propagates along the boundary.
%
%
\begin{lem}  
	\label{lem: propagation on boundary if mu0 supported on xi-d equal 0}   
Let $(x'_0,\xi'_0)\in T^*\pd \Omega$.
Let $(x,\xi)\in T^*\Omega$,
 we denote 
$\gamma{(x,\xi)}$ the integral curve of $H_p$ starting from $(x,\xi)$.
We assume that $m$ is locally supported in a \nhd of $(x'_0,0, \xi'_0, 0)$,  in the set $\{(x,\xi), \ \xi_d^2+ R(x,\xi')=1, \text{ such that } 
 \gamma{(x,\xi)} \text{ hits }x_d=0,\    \ \xi_d=0    \}$. 
  In particular 
$m_0$ is supported 
on $\xi_d=0$ and  $m_0= \tilde m_0\otimes  \delta_{\xi_d=0}
+ \tilde m_1\otimes  \delta_{\xi_d=0}'$, where $\tilde m_0$ is a distribution
and $  \tilde m_1 $ is a  Radon measure.  
Then $H'_Rm^\pd+\tilde m_0=0$, $\tilde m_1=0$ and $H_pm=-\pd_{x_d}R(x',0,\xi') m^\pd\otimes  \delta_{x_d=0}\otimes \delta_{\xi_d=0}'$, in a \nhd of $(x'_0,0, \xi'_0, 0)$.
\end{lem}
\begin{proof}
We have $m=1_{x_d>0}m+m^\pd\otimes  \delta_{x_d=0}\otimes \delta_{\xi_d=0}$. Then we have
\begin{align}
	\label{eq: Hp form propagation}
H_pm&= m_0\otimes  \delta_{x_d=0} + 2\xi_d m^\pd  \otimes\delta_{x_d=0}' \otimes \delta_{\xi_d=0} -(\pd_{x_d} R(x',0,\xi')) 
m^\pd  \otimes \delta_{x_d=0}\otimes \delta_{\xi_d=0}' \\
&\quad +H'_Rm^\pd \otimes\delta_{x_d=0}\otimes \delta_{\xi_d=0} \notag\\
&=  m_0\otimes  \delta_{x_d=0}-(\pd_{x_d} R(x',0,\xi')) 
m^\pd  \otimes \delta_{x_d=0}\otimes \delta_{\xi_d=0}'  +H'_Rm^\pd\otimes \delta_{x_d=0}\otimes \delta_{\xi_d=0},
\notag 
\end{align}
as $\xi_d m^\pd  \delta_{x_d=0}' \otimes \delta_{\xi_d=0}=0$. 
 As 
$m_0$ is of order 1 and supported on $\xi_d=0$,  we have $m_0= \tilde m_0\otimes  \delta_{\xi_d=0}
+ \tilde m_1\otimes  \delta_{\xi_d=0}'$, where $\tilde m_j$ are distributions. 

To prove that $\tilde m_1$ is a Radon measure, we test $m_0$ on 
$ \varphi=\xi_d\psi(x',\xi')\chi(\xi_d/\eps)$, where $\chi$ is defined in the proof of Lemma~\ref{lem: diffractive points non in measure support}. The supremum of $\varphi$ goes to 0 with $\eps$. 
The first derivatives of $\varphi$ are estimated 
by the supremum of 
$\psi(x',\xi')\chi(\xi_d/\eps)$, $\psi(x',\xi')\chi'(\xi_d/\eps)\xi_d/\eps$ and
$\xi_d\pd \psi(x',\xi')\chi(\xi_d/\eps)$. When $\eps\to 0$ the supremum on $\varphi $ is 
estimated by supremum of $\psi(x',\xi')$. We also have $\est{m_0,\varphi}$ converging to
$\est{ \tilde m_1,\psi(x',\xi') } $ as $\eps\to 0$. We deduce that $\est{ \tilde m_1,\psi(x',\xi') } $ is 
estimated by the supremum of $\psi(x',\xi')$, this implies that $ \tilde m_1$ is a Radon measure.

 By Proposition~\ref{prop: propagation formula}, Formula~\eqref{prop: propagation formula 2}
  and if 
 $a\in\Con_0^\infty(\R^d\times\R^{d-1})$ is supported in 
 $\pd\Omega$ and in a \nhd of $(x'_0,0, \xi'_0, 0)$, 
 we have $\est{H_pm, a(x,\xi')}=0$. 
 
Let $a\in\Con_0^\infty(\R^d\times \R^{d-1})$, be such that $a_{|x_d=0}=\ell\in\Con_0^\infty(\R^{d-1}\times \R^{d-1})$, and $\chi\in \Con_0^\infty(\R)$ as above.
From~\eqref{eq: Hp form propagation} we obtain 
\begin{equation*}
0= \est{H_pm , a\chi(x_d/\eps)}= \est{\tilde m_0
 + H'_R m^\pd, \ell}.
\end{equation*}
We deduce 
 that 
$H'_Rm^\pd+\tilde m_0=0$, which gives the first conclusion of Lemma. 
We deduce from~\eqref{eq: Hp form propagation}
\begin{align}\label{eq: formula mu propagation boundary D and N bis}
H_pm  = ( \tilde m_1- (\pd_{x_d} R(x',0,\xi')) 
m^\pd ) \otimes   \delta_{x_d=0}  \otimes     \delta_{\xi_d=0}'   .  
\end{align}
We then can write
\begin{align}
	\label{eq: Hp propagation boundary ter}
\est{m,H_pa(x,\xi)}=\est{ \tilde m_1- (\pd_{x_d} R(x',0,\xi')) 
m^\pd , \pd_{\xi_d}a(x',0,\xi', 0)} ,
\end{align}
for $a\in\Con_0^\infty (\R^d\times\R^d)$. 

Now we choose an adapted $a$ to apply~\eqref{eq: Hp propagation boundary ter}.
 Let $\chi\in\Con_0^\infty(\R)$ as above.
Let $\ell\in\Con_0^\infty(x',\xi') $ be supported
in a \nhd of $(x'_0,\xi'_0)$. We set $a(x,\xi)= \xi_d \ell (x',\xi')\chi(\xi_d/\eps)\chi(x_d/\eps)$, where $\eps>0$.
We have
\begin{align*}
H_p a(x,\xi)&= 2 \ell(x',\xi')\chi(\xi_d/\eps)\chi'(x_d/\eps) \xi_d^2/\eps   \\
&\quad - \ell(x',\xi')( \pd_{x_d}R(x,\xi'))\chi(x_d/\eps)   
 \big( \chi(\xi_d/\eps)+  \chi'(\xi_d/\eps)  \xi_d /\eps \big) \\  
 &\quad +\xi_d  H'_R( \ell(x',\xi'))\chi(\xi_d/\eps)\chi(x_d/\eps).
\end{align*}
As $\chi'(x_d/\eps)\chi(\xi_d/\eps) \xi_d^2/\eps $, $ \chi(x_d/\eps)\chi(\xi_d/\eps) \xi_d/\eps $ and $ \xi_d\chi(\xi_d/\eps)\chi(x_d/\eps)$ are 
uniformly bounded and goes to 0 as $\eps $ goes to 0, and $ \chi(\xi_d/\eps)\chi(x_d/\eps)  $ goes to $1_{x_d=0,\xi_d=0}$ and is uniformly 
bounded. 
We have $\est{m , H_p a}$ goes to $\est{m, -1_{x_d=0,\xi_d=0}(\pd_{x_d}R(x',0,\xi') )\ell(x',\xi')}$ as $\eps$ goes to 0 and we have 
\begin{align} \label{eq: calculus mu1 tilde bis}
\est{m, -1_{x_d=0,\xi_d=0}\pd_{x_d}R(x',0,\xi') \ell(x',\xi')}=-\est{m^\pd,(\pd_{x_d}R(x',0,\xi') ) \ell(x',\xi')}.
\end{align}
 Now we compute the limite as $\eps\to 0$ of $\est{ \tilde m_1, \pd_{\xi_d}a(x',0,\xi', 0)} $.
 We have   
 \begin{align*}
  \pd_{\xi_d}a(x',x_d,\xi', \xi_d) = \ell(x',\xi') \chi(\xi_d/\eps)\chi(x_d/\eps)+  \ell(x',\xi') \chi'(\xi_d/\eps)\chi(x_d/\eps)\xi_d/\eps.
 \end{align*}
 Then  $  \pd_{\xi_d}a(x',0,\xi', 0) = \ell(x',\xi')  $ and  
 \begin{align*}
  \est{ \tilde m_1    - (\pd_{x_d} R(x',0,\xi')) m^\pd
  , \pd_{\xi_d}a(x',0,\xi', 0)} &= \est{ \tilde m_1  - (\pd_{x_d} R(x',0,\xi')) m^\pd,\ell(x',\xi')}\\
  &= -\est{m^\pd,(\pd_{x_d}R(x',0,\xi') ) \ell(x',\xi')}
 \end{align*}
 from~\eqref{eq: Hp propagation boundary ter} and 
  \eqref{eq: calculus mu1 tilde bis}. We deduce that $\tilde m_1=0$ and the last result of 
  the lemma  
 from~\eqref{eq: formula mu propagation boundary D and N bis}.
\end{proof}

%
%
%

\subsection{Proof of Proposition~\ref{prop: propagation support semiclassical measure}}

We prove the propagation result, first in interior which is a classical result, second in a 
\nhd of a point on 
the boundary. From previous lemmas, it is a classical result. We give here the main argument to be complete. 
The proof is consequence of 
Lemmas~\ref{lem: measure on boundary}---\ref{lem: propagation on boundary if mu0 supported on xi-d equal 0}.  
%
%
%

\subsubsection{Propagation in interior domain $\Omega$}

From Proposition~\ref{prop: propagation formula}, we have $H_pm=0$. It is then classical that $m $ in invariant by 
the flow of $H_p$. More precisely, let $\rho_0\in T^*\Omega$ and  we assume that $\gamma(s,\rho_0)\in  T^*\Omega$ for $s\in[0,t]$.
Let  $a\in\Con^\infty_0(\Omega\times\R^d)$ be such that  $a(\gamma(-s,.))$
 is supported in $T^*\Omega$ for $s\in[0,t]$,
we have $\est{\gamma_*(t,.)m,a}=\est{m,a(\gamma(-t,.))}=\est{m,a}$.
 
 \subsubsection{Propagation at boundary: hyperbolic points}
 
 We use the geometry context defined in section~\ref{sec: Geometry}, in particular the function $j$ and 
 the definition of the different flows. 
 
 We prove that  the support of measure is locally empty in the future  assuming that in the past the support 
 of measure is locally empty  but by symmetry we can deduce that the support of measure is locally empty in the past if we assume that the 
 support of measure is locally empty in the future.
 
 Recall that we choose coordinates such that $p(x,\xi)=\xi_d^2+R(x,\xi')-1$ and locally $\Omega=\{ x_d>0 \}$. 

 We recall that a point $(x_0',\xi_0')\in T^*\pd\Omega$ is in ${\mathcal H}$, if $R(x_0',0,\xi_0')-1<0$. We apply  
 Lemmas~\ref{lem: measure on boundary} 
 and \ref{lem: measure hyperbolic points}. We have $m=1_{x_d>0} m$ and $H_p m=  
 m^+\otimes \delta_{\xi_d=\sqrt{1-R(x',0,\xi')}}-m^+\otimes \delta_{\xi_d=-\sqrt{1-R(x',0,\xi')}}$.

 We call $\gamma^\pm$ 
 the integral curve of $H_p$   starting 
 from  $ (x_0',x_d,\xi'_0,\pm\sqrt{1-R(x',0,\xi')})$.  If we assume that  the support of $m=1_{x_d>0} m$ is empty in a \nhd of 
 $ \gamma^-$(s) for $s<0$ and $|s|$ sufficiently small, then $-m^+\otimes \delta_{\xi_d=-\sqrt{1-R(x',0,\xi')}}=0$. 
 This implies $m^+=0$ and $H_pm=0$. 
 As $m_{|x_d<0}=0 $ 
 and $\gamma^+(s)$ is in $x_d<0$ for $s<0$ and $|s|$ sufficiently small, this implies that  $m=0$ in a \nhd of $\gamma^+(0)$.
 
%
%
%

 \subsubsection{Propagation at   gliding points}
 
We recall that a point $(x_0',\xi_0')\in T^*\pd\Omega$ is is gliding or  in ${\mathcal G}_g $, if $R(x_0',0,\xi_0')-1=0$ and 
$\pd_{x_d}R(x_0',0,\xi_0')>0$.  
Let $\gamma$ be the integral curve of $H_p$ starting from $(x'_0,0,\xi_0',0)$.  Then  $\gamma(s)$ 
into $\{x_d<0\}$ for $s\not=0$ and $|s|$ sufficiently small.
In a \nhd of $(x_0',\xi_0')$ in $  T^*\pd\Omega$ all the point are either hyperbolic, or  gliding.
We assume that $j^{-1}\big( \gamma_g(s_0;x_0',\xi_0')\big) \cap \supp m =\emptyset $   for $s_0<0$ where $|s_0|$ 
is sufficiently small.
Here $\gamma_g(s;x_0',\xi_0')=\Gamma(s; x_0',\xi_0')$, then all the point $\rho$ in a 
\nhd of $ j^{-1}\big( \gamma_g(s_0;x_0',\xi_0')\big)$ 
are not in the support of $m$.  By continuity of $\Gamma$ the curve $\Gamma(s;\rho)$ hit the boundary at $\rho'$
 in a \nhd of $(x_0',\xi_0')$. If $\rho'$  is an hyperbolic point, by the previous result  the point $j^{-1}\big( \Gamma(s;\rho)\big)$ 
 are not in the support of $m$. If $\rho'$ is a  gliding point, all the points $\Gamma(s;\rho)$ are strictly gliding. In particular this
 implies that $m$ is supported on $x_d=0$, then $1_{x_d>0}m=0$ and $m_0=0$.
We can apply 
Lemma~\ref{lem: propagation on boundary if mu0 supported on xi-d equal 0} and $m^\pd$ satisfied
$H'_Rm^\pd=0$. Let $\gamma_g$ be the integral curve of $H'_R$ starting from $(x_0',\xi_0')$. 
As by assumption $m^\pd  \otimes \delta _{x_d=0}\otimes\delta_{\xi_d=0} =m$ is 0 in a 
\nhd of $j^{-1}\big( \gamma_g(s_0;x_0',\xi_0')\big)$, we have $m^\pd=0$ in a \nhd of $\gamma_g(s_0)$ and 
$H'_Rm^\pd=0$, this implies that $\gamma_g(s)$ is not in the support of $m^\pd$ in a \nhd of $s=0$. As 
$m=m^\pd\otimes \delta _{x_d=0}\otimes\delta_{\xi_d=0}$ we have $m=0$ in a \nhd of $j^{-1} (x_0',\xi_0')$.
  
%
%
%

 \subsubsection{Propagation at  diffractive points}

We recall that a point $(x_0',\xi_0')\in T^*\pd\Omega$ is in ${\mathcal G}_d $, 
if $R(x_0',0,\xi_0')-1=0$ and
 $\pd_{x_d}R(x_0',0,\xi_0')<0$.  We keep the previous notation 
for $\gamma$. For a point $\rho$ in a \nhd of $(x_0',\xi_0')\in T^*\pd\Omega\cup T^*\Omega$,
there are three cases, first the integral curve  passing through $\rho$  hits $x_d=0$ 
at an hyperbolic point and by previous result the integral
curve is not in support of $m$, second it does not hit $x_d=0$ and the integral curve is not 
in the support of 
$m$ by propagation result in interior, third the integral
curve hits $x_d=0$ at a diffractive point.
Then the support of $m^\pd$ is in ${\mathcal G}_d $ and the support of $1_{x_d>0}m$ is into $\{(x,\xi), \ \xi_d^2+ R(x,\xi')=1, \text{ such that } 
 \gamma{(.;x,\xi)} \text{ hits }x_d=0,\    \ \xi_d=0    \}$. 

We can apply
Lemma~\ref{lem: diffractive points non in measure support}, 
then
 $m^\pd=0$  in a \nhd of \((x_0',0,\xi_0')\). As the  integral curves hitting $x_d=0$ at an 
hyperbolic points are not in the support of $1_{x_d>0}m$, then
$m_0$ is supported on $\xi_d=0$. We can apply 
Lemma~\ref{lem: propagation on boundary if mu0 supported on xi-d equal 0} to obtain $m_0=0$. Then
$H_pm=H_p(1_{x_d>0}m)=0$ and as, by assumption, $\gamma(s)$ is not in the support of $m$ for $s<0$, $|s|$ sufficiently small, 
we deduce that $\gamma(0)$ is not in the support of $m$.

 \subsubsection{Propagation at boundary: integral curves with high contact order}
 
\medskip
We recall that if $(x'_0,\xi'_0) $ is such that $R(x_0', 0, \xi'_0, 0)=1$, 
$\pd_{x_d}R(x_0', 0, \xi'_0, 0)=0$ and if we denote the integral curve of $H_p$ starting from 
$(x_0', 0, \xi'_0, 0)$, by 
$\gamma(s)=(x'(s),x_d(s), \xi'(s),\xi_d(s))$ . From the assumption made (see Definition~\ref{def: contact fini}) 
there exist \(k\in\N\), \(k\ge 3\) and \(\alpha\ne0\) such that
$x_d(s) = \alpha s^k+{\mathcal O}(s^{k+1})$.  We denote $\gamma_g(s)=(x'_g(s),\xi'_g(s))$ the 
integral curve of $H'_R$, starting from $(x'_0,\xi'_0) $.
For each $k$ we assume that we  have already proved that the integral curves hitting $x_d=0$ 
at a point in ${\mathcal G}^j$ for $j<k$ or ${\mathcal H}$ are not 
in the support of $m$.

\paragraph{\textbf{Case $k$ even, $\alpha<0$}}
The integral curve of $H_p$ starting from a point belonging to $  T^*\Omega$  in a \nhd of  $(x'_0,\xi'_0) $  in $T^*\pd\Omega\cup T^*\Omega$ 
eventually hits $x_d=0$ at a point $\rho'$, 
in ${\mathcal H}$ or ${\mathcal G}^j$ for 
$j\le k$ (see Section~\ref{sec: Geometry} for definition of ${\mathcal G}^j$). 
By assumptions and by induction this integral curve is not in the support of $m$ except if 
$\rho'$ is 
in ${\mathcal G}^k$, but in this case this integral curve is in $x_d\le 0$. 
This implies that $1_{x_d>0}m=0$, then $m_0=0$. 
By Lemma~\ref{lem: propagation on boundary if mu0 supported on xi-d equal 0}, 
we have $H'_Rm^\pd=0$ and as, by assumption, 
$\gamma_g(s)$ is not in the support of $m^\pd$ for $s<0$, $|s|$ sufficiently small, 
we deduce that 
$\gamma_g(0)$ is not in the support of $m^\pd$.
 
\paragraph{\textbf{Case $k$ odd, $\alpha<0$}}
By the same argument as in previous case, the integral curve of $H_p$ starting from a point belonging to $  T^*\Omega$ 
in a \nhd of  $(x'_0,\xi'_0) $ in $T^*\pd\Omega\cup T^*\Omega$
hits ${\mathcal G}^k $ or is not in the support of $m$. Denote by $\rho'$ the point of this integral curve hitting $x_d=0$. 
The generalized bicharacteristic starting 
from $\rho'$ is on $x_d=0$ for $s>0$ and in $x_d>0$ for $s<0$, and for $s>0$ all the points on 
the integral curve of $H'_R$ are in ${\mathcal G}_g$, if $|s|$ is sufficiently small. As, by assumption
the  generalized bicharacteristic is not in support of $m $ in the past, this means that $1_{x_d>0}m=0$ then $m_0=0$. 
We can apply Lemma~\ref{lem: propagation on boundary if mu0 supported on xi-d equal 0} then 
$H'_Rm^\pd=0$. But $\gamma_g(s)$ is not in 
support of $m=m^\pd\otimes  \delta_{x_d=0}\otimes \delta_{\xi_d=0}$ for $s<0$ and $|s|$ sufficiently small as $\gamma_g(s)\in{\mathcal G}_d$, then $\gamma_g(0)$ is 
not in the support of $m^\pd$.
 
\paragraph{\textbf{Case $k$ even, $\alpha >0$}}

By induction, only the generalized bicharacteristics with the same order of contact  $k$ and the same sign condition 
$\alpha>0$ can be in the support of 
$m$. 
Applying Lemma~\ref{lem: propagation on boundary if mu0 supported on xi-d equal 0} we have 
\(
-\pd_{x_d}R(x',0,\xi') m^\pd\otimes  \delta_{x_d=0}\otimes \delta_{\xi_d=0}'=0
\)
as by induction, $ m^\pd=0$ when $\pd_{x_d}R(x',0,\xi')\ne 0$. We deduce
 $H_pm=0$. Then the propagation the support of $m$ is invariant by the flow of $H_p$.

 \paragraph{\textbf{Case $k$ odd, $\alpha>0$}}

By induction, only the generalized bicharacteristics with the same order of contact  $k$ and the same sign condition 
$\alpha>0$ can be in the support of 
$m$.  We can apply
Lemma~\ref{lem: propagation on boundary if mu0 supported on xi-d equal 0}, and by the same argument used in the previous case, we have $H_pm=0$ and as 
$\gamma(s)$ is in $x_d<0$ for $|s|<0$ sufficiently small, $\gamma(s)$  is not in the support of $m $ for $s<0$ 
and by propagation $\gamma(0)$ is not in support of $m$.

\begin{proof}[Proof of Proposition~\ref{prop: propagation support semiclassical measure}, first case]
By assumption for a point 
\(
\rho \in  T_b\Omega= T^*\pd\Omega \cup T^*\Omega
\)
for some 
\(s_0\ge 0 \), $\pi(\Gamma(s_0,\rho))\in \omega$ and from Proposition~\ref{prop: bm null} and 
assumption (GCC), $m$ is null in a \nhd of 
$j^{-1}(\gamma(s_0,\rho))$. 
As 
\(
\supp m 
\)
is a closed set, if 
\(
\rho \in\supp m
\)
there exist \(s_1\in [0,s_0]\) such that 
\(
j^{-1}\Gamma(s_1,\rho)\cap \supp m\ne\emptyset
\) 
and 
\(
j^{-1}\Gamma(s,\rho)\cap \supp m=\emptyset
\) 
for 
\( s\in(s_1,s_0].\)
At 
\(
\Gamma(s_1,\rho) 
\)
we can apply the results obtained in this section to prove that 
\(
j^{-1}\Gamma(s_1,\rho)\cap \supp m=\emptyset,
\)
and reach a contradiction.
\end{proof}

%
%
 \section{Estimate of boundary trace, case strictly diffractive}
 	\label{Sec: Estimate of boundary trace, case strictly diffractive}
 
 In this section we prove a technical result allowing to estimate trace at boundary in a \nhd of a  diffractive point.
 This precise estimate allows to prove propagation of support in a \nhd of such points.
 
 %
 %
 \newcommand{\rr}{{\rm r}}
  \newcommand{\rrr}{{\rm y}}
  \newcommand{\ww}{{\rm w}}
    \newcommand{\rrrr}{{\rm z}}
  \newcommand{\ff}     {{\rm f}}
    \newcommand{\kk}     {{\rm k}}
     \newcommand{\trho}     {{\tilde \rho}}  

 \def\ct{\tilde \chi}   
\def\cc{\chi}

%
%

 \subsection{Review on OIF and Airy functions}   
 
 We give in this section some notation and results introduced in~\cite{Cor-Rob}  
 following Tataru~\cite{Tataru-1998}.
%
%
\begin{lem}
	\label{lem: Symplectic transformation}
Let $(x_0', \xi_0' )\in \R^{d-1}\times \R^{d-1}$ be such that $R(x_0', 0, \xi_0' )-1=0$. For all $x_d$ in a \nhd of 0,
there exist a  smooth symplectic transformation $ \kappa :U_0\to U_1$ where $U_0$  and $U_1$ are some open set 
respectively of  $\R^{d-1}_{x'}\times \R^{d-1}_{\xi'}$  and of  $\R^{d-1}_{y'}\times \R^{d-1}_{\eta'}$,
satisfying $ (x_0', \xi_0' )\in U_0 $,  $(0,0) \in U_1$,  $\kappa(x_0', \xi_0' ) = (0,0)$, and
$ \kappa^*(\eta_1)= R(x,\xi')-1$.
 Moreover $x_d$ acts as a parameter and $\kappa$ 
is smooth with respect $x_d=y_d$.
\end{lem}
This lemma is classical. We can find a proof in H\"ormander~\cite[Theorem 21.1.6]{HormanderV3-2007}. 
This means we can complete the coordinate $R-1$ in a  symplectic manner.

To avoid ambiguity even if $x_d=y_d$, we denote $x_d$ when we work in $(x,\xi')$ variables and $y_d$ otherwise. 

We call a symbol of order 0 a symbol $a\in S(1,|dx|^2+|d\xi'|^2)    
$ or in $S(1,|dy|^2+|d\eta'|^2)$. 
In this section we only use tangential symbol, but as in what follows we have to use different classes of symbols,
we prefer use everywhere the same kind of notation.

Let $\chi$ be a smooth function supported in $U_0$, $\chi\ge0$ and $\chi=1$ in a \nhd of $(x'_0,0,\xi'_0)$.
%
%
\begin{lem}
	\label{lem: FIO}
Associated with $\kappa$, there exists $F$, a semiclassical  Fourier Integral Operator 
 satisfying the following properties,
\begin{enumerate}
\item[i)] $F$ is a unitary operator uniformly with respect $x_d$.
\item[ii)] For all $\tilde a\in\Con_0^\infty( U_1) $, $F^{-1}\ops(\tilde a)F=\ops(a)$, where $a=\kappa^*\tilde a+hb$ where
$b$ is a symbol of order 0. In particular we have
$F^{-1} \ops(\eta_1\ct^2(y,\eta')) F= \ops(\cc^2(R-1)) +h\ops(b)$, where  
$\ct\in\Con_0^\infty (U_1)$, $\ct\ge 0$  and $b$ a symbol of order 0.
\item[iii)]  There exist $\theta$ a symbol of order $0$, $B$ a bounded operator on $L^2$ such that
$\ops (\theta)^*=\ops (\theta)$, $\kappa^*\tilde\chi=\chi$ and
$(\pd_{x_d} F) F^{-1}=ih^{-1}\ops (\theta)+hB$.
\item[iv)] 
If the operators $A $ and $\tilde A$ are such that \( A=F^{-1}\tilde AF \) then 
\[
\pd_{x_d}A=F^{-1}\big(  \pd_{y_d} \tilde A+ih^{-1}[\tilde A,\ops (\theta) ]+  h [  \tilde A , B ] 
\big)F
\]
where $B$ is the operator defined previously.
\item[v)]  In particular we have $\kappa^* \{ \eta_1,\theta \}= \pd_{x_d}R$ in a \nhd of $(x'_0,0,\xi_0')$. 
\end{enumerate}
\end{lem}
For a proof of this lemma we refer to~\cite[Lemma B.2]{Cor-Rob}.

From now we shall use two semiclassical quantifications of symbol, one with parameter $h$ and
the other with parameter  $h^{1/3}$. 
To avoid ambiguity or confusion between both, we do not use the notation $\ops$
but we use classical quantification. For instance, for $a$ a symbol of order 0 we have $\ops (a)=\op(a(x,h\xi'))$ 
that is we keep 
the $h$ or $h^{1/3}$ in the notation. 

Let $g=|dy|^2+h^{2/3}\est{h^{1/3}\eta_1}^{-2}|d\eta'|^2$, this metric gives  symbol classes
essentially as semiclassical symbol classes with $h^{1/3}$ for
semiclassical parameter. 
We let to the reader to check that $g$ is slowly varying and $\sigma$ temperate. The 
"$h$" defined by H\"ormander associated with $g$ is $h^{1/3}\est{h^{1/3}\eta_1}^{-1}$.
It is the quantity we gain in the asymptotic expansion 
for the 
symbol calculus.
In particular the function $\est{h^{1/3}\eta_1}^{\nu}$ is a $g$-continuous and $\sigma, g$ temperate for every $\nu\in\R$.
 We refer to \cite[Chapter 18, Sections 4 and 5]{HormanderV3-2007} for  definitions used freely here.
  
To ease notation we write $| u|$ instead of $|u|_{L^2}(x_d)$ when there is no ambiguity on the fact that the 
$L^2$ norm is taken on variables $x'$ or $y'$  at point  $x_d$ or $y_d$. By the same abuse of notation we 
write the inner product $(.|.)$ instead of  $ (.|.)_{L^2(\R^{d-1})}(x_d)$.

We recall some properties of the Airy function which is denoted by $\Ai$. It verifies the equation 
$\Ai''(z) -z\Ai(z)=0$ for $z\in\C$, 
 $\Ai $ is real on the real axis and $\overline{\Ai(z)}=\Ai(\bar z)$. 
Let $\omega=e^{2i\eps \pi/3}$ for $\eps=\pm1 $. 

For $x\in\R$, let $\alpha(x)=-\omega\Ai'(\omega x)/\Ai(\omega x)\in\Con^\infty(\R)$. As the zero of $\Ai$ are on the negative 
real axis, the function $\alpha $ is well defined for $x\in\R$ and smooth.
 The function $\alpha$ 
satisfies the following properties.
%
%
\begin{lem}
	\label{lem: alpha properties}
We have
\begin{enumerate}
\item[i)] $ \alpha(x)=-\sqrt{x}+\frac{1}{4x}+b_1(x) $ for $x>0$
\item[ii)] $ \alpha(x)=\eps i\sqrt{-x}+\frac{1}{4x}+b_2(x) $ for $x<0$
\item[iii)] $\Re \alpha(x)<0 $ for all $x\in\R$
\item[iv)]  $\alpha$ satisfies the differential equation $\alpha'(x)= \alpha^2(x)-x$,
\end{enumerate}	
where $b_j\in S( \est{x}^{-5/2}, |dx|^2)$ for \(j=1,2\). 
\end{lem}
For a proof of this lemma we refer to~\cite[Lemma B.3]{Cor-Rob}.

Let $\tilde r_d$ be such that $\kappa^*\tilde r_d= -\pd_{x_d} R$. We assume that locally $\pd_{x_d} R<0$,
 this implies that 
 $ \tilde r_d>0$ in a  \nhd of $(0,0)$.
 
  Let $\tilde a (y, \eta')=h^{1/3}\tilde \chi (y,h\eta') \tilde r_d^{1/3}(y,h\eta') \alpha(\zeta)$ 
where $\zeta= h^{1/3}\eta_1\tilde r_d^{-2/3}(y,h\eta')$.  
We assume that on the support of  $\tilde \chi (y,h\eta')$, 
$\tilde r_d(y,h\eta') >0$.
In what follows we denote $\rho=(x,h\xi')$ and $\tilde \rho=(y,h\eta')$.
We define $\tilde A=\op(\tilde a(y,\eta') )$,    $\tilde \Psi =h^{-1/3}\op (\est{h^{1/3}\eta_1}^{-1/2})$, and let 
$A=F^{-1} \tilde A  F$, $\Psi=F^{-1}\tilde \Psi F$.  

We have $\tilde a\in S(h^{1/3}\est {h^{1/3}\eta_1}^{1/2}  ,g)$  as 
$h^{-1/3} \est{h^{1/3}\eta_1}^{-1/2}\in S(h^{-1/3} \est{h^{1/3}\eta_1}^{-1/2},g)$
and from Lemma~\ref{lem: alpha properties}. Observe that $\tilde a$ is bounded as on the support of 
$\tilde \chi (y,h\eta')$, $|\eta_1|\le h^{-1} $ and $ \est{h^{1/3}\eta_1}^{1/2}\le h^{-1/3}$. 
We deduce that $\tilde A$ is bounded on $L^2$.

We state the following lemma proved in~\cite{Cor-Rob}, see Lemma~C.1.
\begin{lem}
	\label{lem: estimate by tilde Psi}
Let $\tilde \chi_3\in\Con_0^\infty(U_1)$ be such that $(1-\tilde \chi_3)\tilde \chi_1=0$ where $\kappa^* \tilde \chi_1=\chi_1$.
  We have 
\begin{align*}
& |\tilde  \Psi \op (\tilde \chi_1(\tilde \rho))w | 
\lesssim  |h^{-1/3}\op(\tilde \chi_3(\tilde \rho) \est{h^{1/3}\eta_1}^{-1/2}) \op (\tilde \chi_1(\tilde \rho))w| + |w| ,
\end{align*}
for all $w$ in $L^2$.
\end{lem}

%
%

\subsection{Estimates at the boundary, glancing points}

Let $\chi_0$ be a smooth cutoff in a \nhd of $(x'_0, 0,\xi'_0)$ such that $\chi_0$ is  supported  in the set where 
$\chi=1$ and $\chi_0=1$ in a \nhd of $(x'_0, 0,\xi'_0)$.
We consider the solution $(r_h, u_h^d)$ solution of~\eqref{eq: syst semi-class 2 by 2 bis}. We set 
$\rr=\ops(\chi_0)r_h$ and $\vv=\ops(\chi_0)u_h^d$.
We have 
\begin{align*}
&\| \ops(\chi_0)(\ops((1-\chi)^2)(R-1))r\|\lesssim h\|r_h\| , \\
& \| [\ops(\chi_0),\ops(\chi^2(R-1))]r\|\lesssim h\|r_h\|  , \\
& \| [\ops(\chi_0),hD_{x_d} ]( r_h,u_h^d)\|\lesssim h(\|r_h\| +\|u_h^d  \|).
\end{align*}
Applying  $\ops(\chi_0) $ on System~\eqref{eq: syst semi-class 2 by 2 bis}, we obtain  the system 
\begin{align} 
	\label{Small System}
&hD_{x_d}\rr+\vv=h\ff  \\
&hD_{x_d}\vv-\op(\cc^2(x,h \xi') (R(x,h\xi)-1))\rr=h\kk, \notag\\
&\text{where }\| \ff \|+\| \kk\| \lesssim \| q_h\|+\|f\|+\|r_h\|+\|u_h^d \|\notag \\
&\text{and }\rr=\ops(\chi_0)r_h,\ \vv=\ops(\chi_0)u_h^d.\notag
\end{align}

%
%
\begin{prop}
\label{prop: first estimate glancing}
Let $\rr$ and $\vv$ satisfying system~\eqref{Small System}.
There exists $C_0>0$  such that 
\begin{multline*}
| i\vv +A \rr  |^2(0)+\|  \Psi ( i\vv +A \rr )\|_{L^2(x_d>0)}^2
\\ \le C_0\big(  \|u_h^d \|_{L^2(x_d>0)}^2+\|r_h\|_{L^2(x_d>0)}^2+ \|\kk\|_{L^2(x_d>0)} ^2
+\|\ff\|_{L^2(x_d>0)}^2\big).
\end{multline*}
\end{prop}
\begin{proof}
We compute
\begin{align*}
\frac12\pd _{x_d}|i\vv +A \rr|^2 &=\Re(i\pd _{x_d}\vv +(\pd _{x_d}A)\rr +A\pd _{x_d}\rr | i\vv +A \rr)\\
& =\Re(-h^{-1} (\op(\cc^2(\rho)(R(\rho)-1))\rr+h\kk  )   +(\pd _{x_d}A)\rr +ih^{-1}A(- \vv+h\ff    )  |  i\vv +A \rr  )\\
&=I_1+I_2+I_3+I_4,
\end{align*}
where
\begin{align*}
&I_1= \Re(-h^{-1}A( i\vv +A \rr )  |  i\vv +A \rr  )  \\
&I_2=   \Re(h^{-1}  (-\op(\cc^2(\rho)(R(\rho)-1))\rr +A^2 \rr )|    i\vv +A \rr  ) \\
&I_3=    \Re(  (\pd _{x_d}A)\rr   |    i\vv +A \rr   ) \\
&I_4=      \Re(- \kk + iA\ff  |    i\vv +A \rr     ).
\end{align*}
We set $\ww=F(    i\vv +A \rr )$. As $\tilde A$ and then $A$ are bounded on $L^2$ we have $|\ww|\lesssim |\vv|+|\rr|$.
We now estimate the terms $I_k$.
\begin{align*}
I_1&= \Re(-h^{-1}FAF^{-1}\ww |  \ww )  \\
&= \Re(-h^{-1} \tilde A \ww  | \ww).
\end{align*}
From Lemma~\ref{lem: alpha properties} 
 we have
\[
- h^{-1} \Re \tilde a(y,\eta) \ge \delta  h^{-2/3} \tilde \chi^2(\tilde \rho)\est{h^{1/3}\eta_1}^{-1},
  \]
and $h^{-1} \Re   \tilde  a(y,\eta) \in S( h^{-2/3} \est{h^{1/3}\eta_1}^{1/2} ,g)$. 

Then from Fefferman-Phong inequality 
(see \cite[Theorem 18.6.8]{HormanderV3-2007})  and as the real part of symbol of 
\(
\op( h^{-1/3} \tilde \chi(\tilde \rho)\est{h^{1/3}\eta_1}^{-1/2})^*  \op(h^{-1/3} \tilde \chi(\tilde \rho)\est{h^{1/3}\eta_1}^{-1/2})
\)
 is 
\(
 h^{-2/3} \tilde \chi^2(\tilde \rho)\est{h^{1/3}\eta_1}^{-1}
\)
modulo an operator bounded on $L^2$, 
we have
\begin{equation*}
I_1\ge   \delta | h^{-1/3} \op ( \tilde \chi (\tilde \rho)\est{h^{1/3}\eta_1}^{-1/2})   \ww |^2 -C | \ww|^2, 
\end{equation*}
for $C>0$. 
From definition of $\ww$ we have 
$\ww= \op ( \tilde \chi_0 (\tilde \rho))(iFu_h^d+\tilde AFr_h) +[\tilde A, \op ( \tilde \chi_0 (\tilde \rho))]Fr_h$.
We can apply  Lemma~\ref{lem: estimate by tilde Psi} with 
$ \op (\tilde \chi_1(\tilde \rho))w= \op ( \tilde \chi_0 (\tilde \rho))(iFu_h^d+\tilde AFr_h)$ and as 
$[\tilde A, \op ( \tilde \chi_0 (\tilde \rho))]$ is an operator with symbol in $S(h^{2/3}\est {h^{1/3}\eta_1}^{-1/2} , g)$.
Then $\tilde \Psi [\tilde A, \op ( \tilde \chi_0 (\tilde \rho))] $ is a bounded operator on $L^2$. We then obtain
\begin{equation*}
I_1\ge   \delta | \tilde \Psi   \op ( \tilde \chi_0 (\tilde \rho))(iFu_h^d+\tilde AFr_h)  |^2 -C\big(  |u_h^d |^2+|r_h|^2\big).
\end{equation*}
Same argument allows us to replace $  \op ( \tilde \chi_0 (\tilde \rho))(iFu_h^d+\tilde AFr_h) $ by $\ww$. 
Finally we have
\begin{equation}
	\label{est: I 1}
I_1\ge   \delta | \tilde \Psi   \ww |^2 -C\big(  |u_h^d |^2+|r_h|^2\big).
\end{equation}

 We have 
\begin{align*}
I_2&=      \Re \big(     h^{-1}F(-\op(\cc^2(\rho)(R(\rho)-1))+A^2)  \rr   | \ww \big) \\
&=   \Re \big(     h^{-1}(-\op(h\eta_1\tilde \chi^2(\tilde \rho) )+\tilde A^2)  F\rr   | \ww\big)
+    \Re \big(F \op(b(\rho))  \rr    | \ww\big)  ,
\end{align*}
where $ \op(b(\rho)) $ is bounded on $L^2$ (see Lemma~\ref{lem: FIO}).

The symbol of $\tilde A^2$ is $\tilde a^2\in S(h^{2/3} \est{h^{1/3}\eta_1} ,g)$ 
modulo a term in $S(h,g)$. From definition of $\tilde a$ 
and Lemma~\ref{lem: alpha properties}  
we have 
\begin{align*}
\tilde a^2&= h^{2/3} \tilde \chi^2(\tilde \rho)\tilde r_d^{2/3}(\tilde \rho)\alpha^2(\zeta)   \notag  \\
&=  h^{2/3} \tilde \chi^2(\tilde \rho)\tilde r_d^{2/3}(\tilde \rho)\big(  \zeta 
+  \alpha'(\zeta)  \big)  \notag  \\
&= h\eta_1 \tilde \chi^2(\tilde \rho)+   h^{2/3} \tilde \chi^2(\tilde \rho)\tilde r_d^{2/3}(\tilde \rho)  \alpha'(\zeta).
\end{align*}
We have
\[
h^{-1/3} \tilde \chi^2(\tilde \rho)\tilde r_d^{2/3}(\tilde \rho)  \alpha'(\zeta) \in S(h^{-1/3}\est{h^{1/3}\eta_1}^{-1/2}, g),
\]
we then obtain 
\begin{equation}
	\label{est: I 2}
|I_2|\lesssim   |\rr| \big(  | \tilde \Psi \ww | +  |\ww |   \big).
\end{equation}
We have 
\begin{align*}
I_3&= -  \Re \big( F ( \pd_{x_d}A ) F^{-1} F \rr        | \ww  \big)  \\
&=  -  \Re \big(  ( \op(\pd_{y_d} \tilde a(\tilde \rho) )+ih^{-1}[ \op(\tilde a(\tilde \rho) ),\op(\theta(\tilde \rho))]  F \rr       
 | \ww\big) \\
&\quad  -  \Re \big(  h[ \op(\tilde a(\tilde \rho) ) , B]    F \rr     | \ww \big)   ,
\end{align*}
from Lemma~\ref{lem: FIO}. 

Observe that 
\(
\tilde a(\tilde \rho)\in S(1,g)
\)
as $h^{1/3}\est{h^{1/3}\eta_1}^{1/2}$ is bounded on support of  $\tilde\chi(\tilde\rho)$. Then 
\(
\op(\pd_{y_d}  \tilde a(\tilde \rho) )
\)
and 
\(
 h[ \op(\tilde a(\tilde \rho) ) , B] 
\)
are bounded operators on $L^2$.

From properties of $\alpha$ and symbolic calculus,
the symbol of 
\[
h^{-1}[ \op(\tilde a(\tilde \rho) ),\op(\theta(\tilde \rho))] 
\text{ 
is in  }
S(h^{-1/3}\est{h^{1/3}\eta_1}^{-1/2},g).
\]
 Then   we obtain  
 \begin{equation}
		\label{est: I 3}
|I_3|\lesssim   |\rr| \big(  | \tilde \Psi \ww | +  |\ww|   \big).
\end{equation}

We have 
\begin{equation}
	\label{est: estim I4}
|I_4|=|  \Re(- \kk + iA\ff  |    i\vv +A \rr     )|\lesssim (|\kk| +|\ff|) (|\vv| +|\rr|).
\end{equation}

From \eqref{est: I 1}---\eqref{est: estim I4} we have 
\begin{align*}
\frac12\pd _{x_d}|i\vv +A \rr|^2 \ge   \delta | \tilde \Psi   \ww |^2 -C\big( 
  |\kk| ^2+|\ff|^2 + |u_h^d |^2+|r_h|^2\big).
\end{align*}
Integrating this inequality between 0 and \( \sigma>0\) 
we have
\begin{multline*}
| i\vv +A \rr  |^2(0)+\delta'\int_0^\sigma |  \tilde \Psi \ww|^2(x_d)dx_d
\\ \lesssim  \|u_h^d \|_{L^2(x_d>0)}^2+\|r_h\|_{L^2(x_d>0)}^2+ \|\kk\|_{L^2(x_d>0)} ^2
+\|\ff\|_{L^2(x_d>0)}^2+|u_h^d|^2(\sigma)+|r_h|^2(\sigma).
\end{multline*}
Integrating this inequality between  two positive values of \(\sigma\) and as 
\(
 |  \tilde \Psi  \ww|(y_d)= |  \Psi (  i\vv +A \rr  ) |(x_d)
\)
 we obtain the result.
\end{proof}

We recall the notation $\tilde a (y, \eta')=h^{1/3}\tilde \chi (\tilde \rho) \tilde r_d^{1/3} (\tilde \rho)\alpha(\zeta)$ 
where $\trho=(y,h\eta')$, $\zeta= h^{1/3}\eta_1\tilde r_d^{-2/3}(\tilde \rho)$ and by assumption, $ \tilde r_d (\tilde \rho)>0$ 
on the support of $\tilde \chi (\tilde \rho)$. We state the following lemma, see \cite[Lemma B.5]{Cor-Rob} for a proof.
%
%
\begin{lem}
	\label{lem: properties beta}
There exists a function $\beta\in\Con^\infty(\R)$ satisfying the following properties
\begin{enumerate}
\item[i)]  $ -\beta'-\beta\Re \alpha\ge 0$
\item[ii)] $ \beta\in S(\est{x}^{-1/4},|dx|^2)$
\item[iii)]  $ \beta\gtrsim \est{x}^{ -1/4}$.  
\end{enumerate}
\end{lem}

Let
$\tilde c(y,\eta')=h^{-1/6}\tilde \chi_2(\tilde \rho)\beta(\zeta)$, where $\tilde \chi_2$ is supported
on $\{ \tilde \chi=1 \}$ and $\tilde \chi_2=1$ on a \nhd of $(0,0)$.
We have $\tilde a\in S(h^{1/3}\est {h^{1/3}\eta_1}^{1/2}  ,g)$ and $\tilde c\in S( h^{-1/6} \est {h^{1/3}\eta_1} ^{-1/4},g)  $.
 
We define  $\tilde C=\op (\tilde c ) $ and $C=F^{-1} \tilde C  F$.
%
%
\begin{prop}
\label{lem: alpha airy}
Let $\rr$ and $\vv$ satisfying system~\eqref{Small System}.
There exists $C_0>0$  such that 
\[
|C( i\vv +A \rr )  |^2(0)
\le C_0\big(   \| \kk  \|_{L^2(x_d>0)}^2 +   \|  \ff \|_{L^2(x_d>0)}^2 +  \| r_h  \|_{L^2(x_d>0)}^2 
+  \| u_h^d  \|_{L^2(x_d>0)}^2 \big).
\]
\end{prop}
\begin{proof}
\begin{align}
	\label{K for alpha airy}
\frac12\pd _{x_d}|C(i\vv +A \rr)|^2 &=\Re(\pd _{x_d}\big(C(i\vv +A \rr)\big) |C( i\vv +A \rr) )  \\
& =\Re( (\pd _{x_d} C) (i\vv +A \rr) +  iC\pd _{x_d}\vv +C\pd _{x_d}(A) \rr    +C A \pd _{x_d}\rr  | C( i\vv +A \rr)  )\notag\\
&=J_1+J_2+J_3+J_4+ J_5=K_1+K_2+K_3+K_4+K_5,\notag
\end{align}
where 
\begin{align*}
J_1 &=  \Re(  C^* (\pd _{x_d} C) (i\vv +A \rr)     | i\vv +A \rr  )\\
J_2 &=  \Re(  \pd _{x_d}(A) \rr       |C^*C( i\vv +A \rr) ) \\
J_3 &=   \Re(	-\kk+iA\ff	 |C^*C( i\vv +A \rr) ) \\
J_4 &=  \Re(	-ih^{-1} A\vv	 |C^*C( i\vv +A \rr) )  \\
J_5 &=   \Re(-h^{-1}\op(\cc^2(\rho)(R(\rho)-1)) \rr  |C^*C( i\vv +A \rr) ) .
\end{align*}
Taking as in the proof of Proposition~\ref{lem: alpha airy},
$\ww=F(    i\vv +A \rr )$ and from Lemma~\ref{lem: FIO},
we write 
\begin{align*}
K_1  &=  \Re(   \tilde C^* (\pd _{y_d} \tilde C) \ww     |  \ww )  \\
K_2 &=  \Re(  (\pd _{y_d}\tilde A  +h[\tilde A,B]  )F \rr       |\tilde C^* \tilde C      \ww )
 +\Re ( F\op (b(\rho)) \rr | \tilde C^*\tilde C   \ww  ) \\
K_3  &=  \Re( -F\kk+iFA\ff	 |\tilde C^*\tilde C   \ww )  \\
K_4  &=  \Re(-h^{-1}   \tilde A \ww | \tilde C^*\tilde C    \ww )  + \Re(   \tilde C^*  ( ih^{-1} [\tilde C,\op(\theta(\trho)) ]
+h\tilde C^* [\tilde C, B])F \ww     |  \ww ) \\
K_5  &=  \Re( h^{-1} (   \tilde A  ^2 - \op(h\eta_1\tilde \chi^2(\tilde \rho) ))F \rr |  \tilde C^*\tilde C  \ww )  + 
 \Re(  (ih^{-1}[\tilde A, \op(\theta(\trho))] F \rr       |\tilde C^* \tilde C      \ww ).
\end{align*}
Observe that  the symbol of 
\( \tilde C^*( \pd_{y_d} \tilde C) \) is in \(  S(h^{-1/3}\est{h^{1/3}\eta_1}^{-1/2},g) ,\) then 
\begin{equation}
	\label{est: K 1}
|K_1|\lesssim | \tilde \Psi \ww||\ww|   .
\end{equation}
From symbolic calculus the symbol of  $ ( \pd_{y_d} \tilde A  )^*  \tilde C^* \tilde C $ is in $S(1,g)$, thus this operator is 
bounded on $L^2$. Clearly the terms $[\tilde A,B]$ coming from remainder term of $\pd_{x_d}A$ (see  \textbf{iv)} 
Lemma~\ref{lem: FIO}) and  $\op (b(\rho))$ coming from remainder term of  
$F\op(\chi^2(\rho)(R(\rho)-1)) F^{-1} $ (see  \textbf{ii)} Lemma~\ref{lem: FIO}) are bounded on $L^2$. 
As $ \tilde C^* \tilde C$ has a symbol in $S(h^{-1/3}\est{h^{1/3}\eta_1}^{-1/2},g)$,
we obtain
\begin{equation}
	\label{est: K 2}
|K_2|\lesssim   |\rr| |\ww|+  |\rr|  |\tilde \Psi \ww|.
\end{equation}
From  same arguments we have
\begin{equation}
	\label{est: K 3}
|K_3|\lesssim (|\kk|+|\ff|) | \tilde\Psi \ww|   . 
\end{equation}
To estimate the last term of $K_4$ we write
\(
  \tilde C^* [ \tilde C ,B]=   \tilde C^* \tilde C B-  \tilde C^*  B\tilde C, 
\)
the first term gives a term estimated by 
\(
 | \tilde \Psi \ww |   |\ww|  
\)
and the second is estimated by
\(
| \tilde C\ww |^2=(  \tilde C^* \tilde C \ww |  \ww  )\lesssim | \tilde \Psi \ww||\ww|  .
\)

To estimate the other terms of $K_4$ observe that 
 the symbol of 
 \(
  - h^{-1} \tilde C^* \tilde C\tilde A +i h^{-1} \tilde C^* [ \tilde C , \op(\theta(\trho)) ])
\)
is  
 \(
  - h^{-1} \tilde  c^2 \tilde a + h^{-1} \tilde c \{ \tilde c , \theta(\tilde \rho) \} 
\)
modulo a symbol in 
\(
S(h^{-2/3}\est{h^{1/3}\eta_1}^{-1},g)
\)
and    
this  term can be estimate by 
\(
|  \tilde \Psi w_h|^2  .
\)
We compute
\begin{align*}
h^{1/6}\{ \tilde c,\theta(\tilde \rho)\}
&= \{ \tilde \chi_2 (\tilde \rho) ,\theta(\tilde \rho)\}\beta(\zeta)+\{\beta(\zeta) ,\theta(\tilde \rho)\} \chi_2(\tilde \rho)   \\
&= \{ \tilde \chi_2 (\tilde \rho) ,\theta(\tilde \rho)\}\beta(\zeta)+\{\zeta ,\theta(\tilde \rho)\} \chi_2(\tilde \rho)   \beta'(\zeta ).
\end{align*}
The term 
\(
 h^{-1} \tilde c  h^{-1/6}  \{ \tilde \chi_2 (\tilde \rho) ,\theta(\tilde \rho)\}\beta(\zeta) \in S(   h^{-1/3}\est{h^{1/3}\eta_1}^{-1/2},g),
\)
and the term of $K_4$ coming from this term can be estimated by $|  \tilde \Psi \ww ||  \ww|$.
For the other term we have
\begin{align*}
 h^{-1/3}\{\zeta ,\theta(\tilde \rho)\} &=   \{  \tilde r_d^{-2/3} (\tilde \rho)  ,\theta(\tilde \rho)\}  \eta_1  
 +  \{  \eta_1   ,\theta(\tilde \rho)\}  \tilde r_d^{-2/3}(\tilde \rho).
\end{align*}
The term 
\(
 h^{-1} \tilde c   h^{-1/6}  \chi_2(\tilde \rho)   \beta'(\zeta ) h^{1/3}  \{  \tilde r_d^{-2/3} (\tilde \rho)  ,\theta(\tilde \rho)\}  \eta_1  
  \in S(   h^{-1/3}\est{h^{1/3}\eta_1}^{-1/2},g)
\)
and the term of $K_4$ coming from this term can be estimate by $|  \tilde \Psi \ww||  \ww|$.

Thus, modulo remainder terms, 
 the symbol of 
\(
  - h^{-1} \tilde C^* \tilde C\tilde A +i h^{-1} \tilde C^* [ \tilde C , \op(\theta) ])
\)
is  given by 
 \begin{align*}
L=&  - h^{-1} \tilde  c^2 \tilde a 
  + h^{-1} \tilde c h^{-1/6} \tilde  \chi_2(\tilde \rho)   \beta'(\zeta )   h^{1/3} \{  \eta_1   ,\theta(\tilde \rho)\}  \tilde r_d^{-2/3}(\tilde \rho) \\
  &\qquad= -h^{-1}  h^{-1/3}\tilde \chi_2^2(\tilde \rho)\beta^2(\zeta)
   h^{1/3}\tilde \chi (\tilde \rho) \tilde r_d^{1/3} (\tilde \rho)  \alpha(\zeta)  \\
 &\qquad\quad   - h^{-1}     h^{-1/6}\tilde \chi_2(\tilde \rho)\beta(\zeta)            
  h^{-1/6}  \tilde \chi_2(\tilde \rho)   \beta'(\zeta )   h^{1/3}       \tilde r_d
    \tilde r_d^{-2/3}(\tilde \rho)
\end{align*}
from \textbf{v)} of Lemma~\ref{lem: FIO}.
We thus obtain
 \begin{align*}
 L=  h^{-1}   \tilde \chi_2^2(\tilde \rho)  \beta(\zeta)    \tilde r_d^{1/3} (\tilde \rho) \big(   - \beta(\zeta) \tilde \chi (\tilde \rho) \alpha(\zeta) 
 - \beta'(\zeta )   \big)\in S( h^{-1},g).
 \end{align*}
As $ \tilde \chi $ is equal $1$ on the support of $ \tilde \chi_2^2$, $\beta\ge 0$ and 
\(
 - \beta(\zeta) \Re \alpha(\zeta) 
 - \beta'(\zeta )  \ge0
 \)
we have $\Re L\ge0$.  We can apply sharp G\aa rding inequality (see \cite[Theorem 18.6.7]{HormanderV3-2007}), we yield, taking account remainder terms
\begin{equation}
	\label{est: K 4}
K_4\ge -C\big( |\tilde \Psi \ww|^2 + | \tilde \Psi \ww||\ww|\big)  .
\end{equation}
Now we shall estimate $K_5$.
The symbol of 
\(
   h^{-1}\tilde A^2  + i h^{-1}[ \tilde A , \op(\theta(\tilde \rho) ) ] -  h^{-1}\op (h\eta_1 \tilde \chi^2(\tilde \rho)) 
\)
is 
\(
   h^{-1}\tilde a^2  +  h^{-1}\{ \tilde a ,\theta(\tilde \rho)  \} -  \eta_1 \tilde \chi^2(\tilde \rho) 
\)
modulo a symbol in $S(1,g)$.
We have
\begin{align*}
\{ \tilde a ,\theta (\tilde\rho)\}&= h^{1/3} \alpha(\zeta)  \{(\tilde \chi  \tilde r_d^{1/3} )(\tilde \rho) ,\theta(\tilde \rho)  \}
+h^{2/3}(\tilde \chi  \tilde r_d^{1/3} )(\tilde \rho)  \alpha'(\zeta)   \eta_1 \{    \tilde r_d^{-2/3}(\tilde \rho)  ,\theta(\tilde \rho)  \}  \\
&\quad+h^{2/3}(\tilde \chi  \tilde r_d^{-1/3} )(\tilde \rho) \alpha'(\zeta) \{   \eta_1  ,\theta(\tilde \rho)  \}.
\end{align*}
The first two terms give  terms  estimated by 
\(
h^{1/3}\est{h^{1/3}\eta_1}^{1/2} \lesssim 1
\)
as  on the support of $\tilde \chi$ we have $|\eta_1|\lesssim h^{-1}$. 
Then both terms give  associated operators bounded on $L^2$.  Modulo a bounded operator on $L^2$ we have to consider
the symbol, taking account  \textbf{v)} of Lemma~\ref{lem: FIO}
\begin{align*}
&  h^{-1/3}   \tilde \chi ^2(\tilde \rho)  \tilde r_d^{2/3}(\tilde \rho)  \alpha^2(\zeta) 
-   h^{-1/3} (\tilde \chi  \tilde r_d^{2/3} )(\tilde \rho) \alpha'(\zeta)  
 -    \eta_1 \tilde \chi^2(\tilde \rho)  \\
 &=   h^{-1/3}   \tilde \chi ^2(\tilde \rho)  \tilde r_d^{2/3}(\tilde \rho)\big(
\alpha^2(\zeta) -\alpha'(\zeta) - h^{1/3}   \tilde r_d^{-2/3}(\tilde \rho) \eta_1
 \big)  \\
&\quad  -   h^{-1/3} (\tilde \chi  \tilde r_d^{2/3} ) (\tilde \rho)    (1-\tilde \chi(\tilde \rho) )   \alpha'(\zeta) .
 \end{align*}
The first term is null from differential equation satisfying by $\alpha$ and the value of $\zeta$.
We claim that 
\begin{equation}
	\label{claim: estimate on K5}
|\op \big(   h^{-1/3} (\tilde \chi  \tilde r_d^{2/3} ) (\tilde \rho)    (1-\tilde \chi(\tilde \rho) )   \alpha'(\zeta)   \big)  F\rr |	
\lesssim |r_h|.
\end{equation}
The proof of the claim is given below.
With this claim  and what we do above, the operator
\(
 h^{-1}\tilde A^2  + i h^{-1}[ \tilde A , \op(\theta(\tilde \rho)) ] -  h^{-1}\op (h\eta_1 \tilde \chi^2(\tilde \rho)))
\)
gives a term bounded by 
\(
|r_h|.
\)
As the symbol of 
\(
 \tilde C^* \tilde C
\)
is in 
\(
S(h^{-1/3} \est{h^{1/3}\eta_1}^{-1/2} ,g),
\)
we obtain that 
\begin{equation}
	\label{est: K 5}
|K_5|\lesssim  |r_h|  |\tilde \Psi \ww|   . 
\end{equation}
From  \eqref{K for alpha airy}---\eqref{est: K 4} and  \eqref{est: K 5} we obtain
\begin{equation*}
\frac12\pd_{x_d}|C(i\vv +A \rr  )|^2\gtrsim -\big( |\kk|^2+|\ff|^2+|\tilde \Psi \ww|^2 +  |u_h^d|^2    +| r_h|^2  \big).
\end{equation*}
Integrating this inequality between 0 and \( \sigma>0\),
we have, estimating  the term coming from \( |\tilde \Psi \ww | \)  by Proposition~\ref{prop: first estimate glancing}, 
\begin{equation}
	\label{est: fin second lemma a la Tataru}
|C( i\vv +A \rr )  |^2(0)
\lesssim    \| \kk  \|^2 +   \|  \ff \|^2 +  \|u_h^d\|^2    +\| r_h\|^2 +  |C( i\vv +A \rr  )  |^2(\sigma).
\end{equation}
As 
\[
  |C(  i\vv +A \rr  )  |^2(\sigma)= (\tilde C ^*\tilde C  \ww,\ww),
\]
we have      
\[
  |C(  i\vv +A \rr)  |^2(\sigma)\lesssim |\tilde \Psi \ww |^2(\sigma)+ |u_h^d|^2 (\sigma)+ | r_h|^2(\sigma).
\]
Integrating estimate~\eqref{est: fin second lemma a la Tataru}  between  two positive values of \(\sigma\) and estimating 
as above the term \(  |\tilde \Psi \ww |^2(\sigma)\), we obtain the conclusion of Proposition~\ref{lem: alpha airy}.
\end{proof}
\begin{proof}[Proof of Claim~\eqref{claim: estimate on K5}]

Let $\chi_1$ smooth cutoff such that $\chi_1$ is supported on $\chi=1$ and $\chi_0$ supported on $\chi_1=1$. 
We have $| \rr-\op(\chi_1(\rho))\op(\chi_0(\rho))r_h|\lesssim h^N|r_h|$ for all $N>0$.

As   $\kappa^* \tilde  \chi_1 =  \chi_1$, from Lemma~\ref{lem: FIO} we thus have 
\[
F^{-1}\op( \tilde  \chi_1(\tilde \rho))F= \op(  \chi_1(x,h\xi')) +h K,
\]
where $K$ is bounded on $L^2$. We then have
\(
F \rr =\op( \tilde  \chi_1(\tilde \rho))F \rr+ hK' r_h,
\)
where $K'$ is bounded on $L^2$. Then
\begin{align*}
&\op \big(   h^{-1/3} (\tilde \chi  \tilde r_d^{2/3} ) (\tilde \rho)    (1-\tilde \chi(\tilde \rho) )   \alpha'(\zeta)   \big)  F\rr \\
&\quad= \op \big(   h^{-1/3} (\tilde \chi  \tilde r_d^{2/3} ) (\tilde \rho)    (1-\tilde \chi(\tilde \rho) )   \alpha'(\zeta)   \big)
\big( \op( \tilde  \chi_1(\tilde \rho))F \rr+ hK' r_h  \big).
\end{align*}
The first term coming from \(  \op( \tilde  \chi_1(\tilde \rho)) \) gives an operator with null symbol.
As 
\[
  h^{-1/3} (\tilde \chi  \tilde r_d^{2/3} ) (\tilde \rho)    (1-\tilde \chi(\tilde \rho) )   \alpha'(\zeta)\in
  S(h^{-1/3},g),
\]
 then the second term is also bounded by \( |r_h| \). This proves the claim.
\end{proof}
In the following proposition we keep notation introduced above, precise statements on $\tilde \chi_2$ are given 
above Proposition~\ref{lem: alpha airy}.

%
%
\begin{prop}
	\label{prop: boundary term diffractive Neumann}
Let $(x_0',0)\in \pd\Omega$.
Let $\chi_4\in \Con_0^\infty(\R)$ be supported on a \nhd of 0 and $\chi_4=1 $ in a  \nhd of 0.
Let $(x_0',0,\xi_0')$ and  $U_0$ be as in the statement of Lemma~\ref{lem: Symplectic transformation}.
Let $\ell\in \Con_0^\infty(\R^{d}\times \R^{d-1} ) $ supported on $\{ \chi_2=1\}$  for every $x_d$, where 
$\kappa^*\tilde \chi_2= \chi_2$. We moreover assume  $\pd_{x_d} R(x,\xi')<0$ on support of $\ell$. 
We have 
\[
\lim_{\eps\to 0} \lim_{h\to 0}  \Big| \Big( 
\op\big(  \ell(x',0,h\xi') \chi_4\big(( R(x',0,h\xi')-1)/\eps\big)  ( R(x',0,h\xi')-1 )\big)  \rr_{|x_d=0}|  \rr_{|x_d=0}  \Big)_\pd  \Big|=0.
\]
\end{prop}
\begin{proof}
Let $\rrr=(F\rr)_{|x_d=0}$. From  Proposition~\ref{lem: alpha airy}  and as $\vv_{|x_d=0}=0$ we obtain
\(
|\tilde C \tilde A\rrr | 
\)
is bounded.  By symbolic calculus and the support properties of $\tilde \chi$ and $\tilde \chi_2$
\[
\tilde C \tilde A= h^{1/6}\op\big(\tilde \chi_2(\tilde \rho) \beta (\zeta ) 
 \tilde r_d^{1/3}(\tilde \rho)  \alpha(\zeta ) \big),
\]
modulo an operator with symbol in $S(h^{1/2}\est{ h^{1/3}\eta_1}^{-3/4}  ,g)$. From  Lemma~\ref{lem: first est trace}
this remainder  term goes to 0 as $h$.

Let $\tilde \ell$  be such that $\kappa^*\tilde \ell=\ell$ and 
$\kappa^*   \chi_4(\eta_1/\eps)= \chi_4( (R(x',0,\xi')-1)/\eps)$.  In what follow, to be 
coherent with previous notation, we define $\tilde \chi_4=\chi_4$ and we use the notation $\tilde \chi_4$ when the function 
is defined in $(y,\eta)$ variables. 

From   Lemma~\ref{lem: FIO} we have 
\begin{multline*}
F^{-1}\op(h\eta_1 \tilde \ell (\tilde \rho)  \tilde \chi_4  (h\eta_1/\eps  ))F\\
= \op \big(  \ell(x',0,h\xi') \chi_4\big(( R(x',0,h\xi')-1)/\eps\big) 
 ( R(x',0,h\xi')-1 )\big) +h\op(r_0(x,h\xi')),
\end{multline*}
where  $r_0$ is of order 0. The term coming from $r_0$ goes to 0 as $h$, from   Lemma~\ref{lem: first est trace}.
Then it is sufficient to prove that 
\(
\ds \lim_{\eps\to 0} \ds \lim_{h\to 0} \Big( \op(h\eta_1 \tilde \ell (\tilde \rho)  \tilde \chi_4  (h\eta_1/\eps  ))  \rrr| \rrr \Big)_\pd=0.
\)
Considering the symbol 
\(
h\eta_1 \tilde \ell (\tilde \rho)  \tilde \chi_4  (h\eta_1/\eps  )\in S(h^{2/3}\est{h^{1/3}\eta_1},g)
\)
and from support properties of $  \tilde \ell  $ and $\tilde \chi_2$, we have
\begin{align*}
 \op(h\eta_1 \tilde \ell (\tilde \rho)  \tilde \chi_4  (h\eta_1/\eps  )) 
=\tilde\gamma^*
\op\big(h^{2/3}\eta_1 \tilde \ell (\tilde \rho)  \tilde \chi_4  (h\eta_1/\eps  )      
\beta^{-2} (\zeta )  \tilde r_d^{-2/3}(\tilde \rho) | \alpha(\zeta ) |^{-2} \big)
 \tilde\gamma,
\end{align*}
where 
\(
 \tilde\gamma = \op\big( h^{1/6}    \tilde \chi_2 (\tilde \rho)   \beta (\zeta )  \tilde r_d^{1/3}(\tilde \rho)  \alpha(\zeta )\big),
\)
modulo an operator with symbol in $S(h,g)$, then this last term involves a term going to 0 as $h$ to 0.
Then we obtain the following estimation, modulo a term going to 0 as $h$ to 0, 
\begin{multline}
	\label{est: Neumann diffractif}
| \Big( \op(h\eta_1 \tilde \ell (\tilde \rho)  \tilde \chi_4  (h\eta_1/\eps  ))  \rrr| \rrr \Big)_\pd| 
\\
\lesssim  | \op\big(h^{2/3}\eta_1 \tilde \ell (\tilde \rho)  \tilde \chi_4  (h\eta_1/\eps  )      
\beta^{-2} (\zeta )  \tilde r_d^{-2/3}(\tilde \rho) | \alpha(\zeta ) |^{-2} \big)
\rrrr\big)|  | \rrrr|,
\end{multline}
where 
\(
\rrrr= \op\big( h^{1/6}    \tilde \chi_2 (\tilde \rho)   \beta (\zeta )  \tilde r_d^{1/3}(\tilde \rho)  \alpha(\zeta )\big)\rrr.
\)
Observe from Proposition~\ref{lem: alpha airy},  
$|\rrrr|$ is bounded, as $ h^{1/6}    \tilde \chi_2 (\tilde \rho)   \beta (\zeta )  \tilde r_d^{1/3}(\tilde \rho)  \alpha(\zeta )$ is the principal symbol of $\tilde C\tilde A$.

We claim that 
\(
h^{2/3}\eta_1 \tilde \chi_4  (h\eta_1/\eps  )\in S(\est{h^{1/3}\eta_1}^{1/2}, g).
\)
Indeed 
\(
h^{2/3}|\eta_1|\lesssim \est{ h\eta_1 }^{1/2} \est{h^{1/3}\eta_1}^{1/2},    
\)
and this gives the sought estimate.
For derivatives and  for $k\ge1$, we have
\[
\pd_{\eta_1}^k\big(  h^{2/3}\eta_1 \tilde \chi_4  (h\eta_1/\eps  ) \big) =kh^{2/3}  (h/\eps)^{k-1}    \tilde \chi_4 ^{(k-1)} (h\eta_1/\eps  )
+ h^{2/3} \eta_1 (h/\eps)^{k}    \tilde \chi_4 ^{(k)} (h\eta_1/\eps  ).
\]
As $h\eta_1$ is bounded on the support of $  \tilde \chi_4 ^{(k)} (h\eta_1/\eps  )$ both terms are bounded by $h^{k-1/3}$.
From estimate
\(
h^{2/3}\est{h^{1/3}\eta_1}\lesssim  \est{h\eta_1}
\)
we have 
\[
h^{k-1/3}\lesssim h^{k/3} \big(   \est{h\eta_1}  /  \est{h^{1/3}\eta_1}   \big)^{k-1/2},
\]
which proves the claim. Observe that the constants in above estimation are not uniform with respect $\eps$.

With the previous claim and as 
\(
\beta^{-2} (\zeta )  | \alpha(\zeta ) |^{-2}\lesssim \est{h^{1/3}\eta_1}^{-1/2}, 
\)
we have 
\[
L(y',\eta')=h^{2/3}\eta_1 \tilde \ell (\tilde \rho)  \tilde \chi_4  (h\eta_1/\eps  )      
\beta^{-2} (\zeta )  \tilde r_d^{-2/3}(\tilde \rho) | \alpha(\zeta ) |^{-2} \in S(1,g).
\]
As 
\(
h^{2/3}|\eta_1|\lesssim | h\eta_1 |^{1/2} \est{h^{1/3}\eta_1}^{1/2}\lesssim \sqrt{\eps}\est{h^{1/3}\eta_1}^{1/2}, 
\)
on the support of $  \tilde \chi_4  (h\eta_1/\eps  )   $, we deduce from G\aa rding inequality 
(see \cite[Theorem 18.6.7]{HormanderV3-2007})
that the operator norm from $L^2$ to $L^2$ of
$\op(L)$ is bounded by $ C \sqrt{\eps} + C_\eps h^{1/3}$ where $C$ is independent of $\eps$ and $ C_\eps $ may depend on 
$\eps$. From that and \eqref{est: Neumann diffractif} we deduce the result.
\end{proof}

\end{document}